\documentclass[a4paper,twoside,11pt]{article}
\usepackage[a4paper,margin=2.5cm]{geometry}
\usepackage[shortlabels]{enumitem}
\usepackage[normalem]{ulem}
\usepackage[T1]{fontenc}
\usepackage[utf8]{inputenc} 
\usepackage[main=english,french]{babel}
\usepackage{xcolor}
\usepackage{amsmath,amssymb,amsthm}
\usepackage{bm}
\usepackage{mathrsfs}
\usepackage{bbm} 
\usepackage{url}
\usepackage{multirow}
\usepackage{enumitem}
\usepackage{comment}
\usepackage{xparse}
\usepackage{graphicx}
\usepackage[hidelinks]{hyperref}
\numberwithin{equation}{section}
\theoremstyle{plain}
\newtheorem{theo}{Theorem}[section]
\newtheorem{theostar}{Theorem}

\newtheorem{lemstar}{Lemma}

\newtheorem{defstar}{Definition}

\newtheorem{prop}[theo]{Proposition}
\newtheorem{lem}[theo]{Lemma}
\newtheorem{cor}[theo]{Corollary}
\newtheorem{defi}[theo]{Definition}

\theoremstyle{remark}
\newtheorem{rem}[theo]{Remark}
\newtheorem{example}[theo]{Example}

\DeclareRobustCommand{\hatHone}{%
  \hyperlink{Hhat1}{\ensuremath{\hat{\mathbb{H}}_1}}%
}
\DeclareRobustCommand{\Hone}{%
  \hyperlink{H1}{\ensuremath{\mathbb{H}}_1}%
}
\DeclareRobustCommand{\hatHtwo}{%
  \hyperlink{Hhat2}{\ensuremath{\hat{\mathbb{H}}_2}}%
}

\DeclareRobustCommand{\Htwo}{%
  \hyperlink{H2}{\ensuremath{\mathbb{H}_2}}%
}

\DeclareRobustCommand{\negHone}{%
	\hyperlink{negH1}{\ensuremath{\neg\mathbb{H}_1}}%
}

\DeclareRobustCommand{\negHhattwo}{%
	\hyperlink{negHhat2}{\ensuremath{\neg\hat{\mathbb{H}}_2}}%
}
\addto\extrasenglish{}

\newcommand{\ddr}{\mathrm{d}}
\newcommand{\ee}{\mathsf{e}} 

\providecommand{\keywords}[1]{\small\textbf{Keywords.} #1}

\newcommand{\rhoup}{\overline{\varrho}_{\,\scriptscriptstyle\Sigma,\hat\Phi}}
 \newcommand{\rhoupepsilon}{\overline{\varrho}_{\,\scriptscriptstyle\Sigma^{\epsilon},\hat\Phi}} 
\newcommand{\rhodown}{\underline{\varrho}_{\,\scriptscriptstyle\Sigma,\hat\Phi}}
\newcommand{\rhoupdual}{\overline{\varrho}_{\,\scriptscriptstyle\hat\Sigma,\Phi}}
\newcommand{\rhodowndual}{\underline{\varrho}_{\,\scriptscriptstyle\hat\Sigma,\Phi}}

\newcommand{\thetaup}{\overline{\theta}_{\,\scriptscriptstyle\Phi,\hat{\Sigma}}}
\newcommand{\thetadown}{\underline{\theta}_{\,\scriptscriptstyle\Phi,\hat{\Sigma}}}
\newcommand{\thetaupdual}{\overline{\theta}_{\,\scriptscriptstyle\hat{\Phi},\Sigma}}
\newcommand{\thetadowndual}{\underline{\theta}_{\,\scriptscriptstyle\hat{\Phi},\Sigma}}
\newcommand{\thetadowndualepsilonn
}{\underline{\theta}_{\,\scriptscriptstyle\hat\Phi_n, \Sigma^{\epsilon}}}

\title{Continuous-state branching processes with L\'evy-Khintchine drift-interaction: 	Laplace duality and Fellerian extensions}
\date{\today}
\author{Cl\' ement Foucart\thanks{LAGA, Universit\'e Sorbonne Paris Nord, Email: foucart@math.univ-paris13.fr} \ and F\'elix Rebotier \thanks{CMAP, Ecole Polytechnique. Email: felix.rebotier@polytechnique.edu}} 

\usepackage{titling}
\usepackage{tocloft}
\begin{document}
\maketitle
\vspace{-5mm}
\begin{abstract} 
We investigate the class of continuous-state branching processes with interaction driven by a L\'evy-Khintchine type  drift  ($\mathrm{CBDI}$). These $[0,\infty]$-valued processes capture both dynamics of branching and
density-dependence, allowing for cooperation at low population sizes and competition at high densities. Although the interaction breaks the branching property, the L\'evy–Khintchine form of the drift induces a Laplace duality. This duality expresses the Laplace transform of a $\mathrm{CBDI}$ process in terms of that of another $\mathrm{CBDI}$ process, in which the branching and drift–interaction mechanisms are exchanged. The process, stopped upon hitting either boundary $0$ or $\infty$, is uniquely characterized in law by these mechanisms. A Fellerian extension is constructed when the drift is non‑Lipschitz and sufficiently strong at a boundary, allowing the process to leave that boundary continuously and possibly re‑enter it. We identify parameters, defined in terms of the mechanisms and their associated scale function and potential measure, that determine the boundary behavior at $0$ and $\infty$ (entrance, exit or regular). Settings exhibiting all regimes, including regular-for-itself and non-sticky boundaries, arise when the mechanisms are assumed to be regularly varying. Our approach combines Laplace duality and comparison principles. The duality facilitates the analysis of semigroups and the construction of sharp Lyapunov functions. Comparisons ensure monotonicity and convergence properties of first-passage times.
\end{abstract}

\vspace{0.5cm}
\vspace{0.5cm}
\setcounter{tocdepth}{1} 
\footnotesize
\renewcommand{\cfttoctitlefont}{\centering\small\bfseries}
\renewcommand{\cftaftertoctitle}{\par}
\begin{center}
\tableofcontents
\end{center}
\vspace{5mm}
\normalsize
\section{Introduction}
Continuous-state branching processes with \textit{drift-interaction} (CBDIs) have been studied in numerous works and in many directions during the past twenty years. They consist in superimposing on a random continuous population evolving by branching, a nonlinear deterministic effect representing interactions between individuals. The main interest of these processes lies in modeling a phenomenon of \textit{density-dependence}. In the setting, for instance, of a Feller diffusion with generalized drift, the process recording the population size, can be thought as solution to a stochastic equation of the form
\begin{equation}\label{eq:fellerdiffusion}
\ddr X_t=\sqrt{2\mathrm{a}X_t}\ddr B_t-\hat{\Psi}(X_t)\ddr t,\ X_0=x\in [0,\infty),\end{equation}
with $B$ a Brownian motion, $\mathrm{a}\in[ 0,\infty)$ and $\hat{\Psi}$ some real function defined on $[0,\infty)$. 
\smallskip

The function $\hat{\Psi}$ in \eqref{eq:fellerdiffusion} governs the density-dependence and reflects either a competitive pressure due to a large amount of individuals ($\hat{\Psi}>0$ near $\infty$), or, conversely, situations in which growth is favored at low population sizes ($\hat{\Psi}<0$ near $0$)\footnote{Notice the sign convention in the drift $-\hat{\Psi}$}. In the latter case, one often speaks of positive interactions or \textit{cooperation}. Such a low-density behavior is a phenomenon known in population dynamics and ecology as \textit{weak} Allee effect (positive per‑capita growth at low densities), see Courchamp et al. \cite{courchamp2008allee} and e.g. Carlos and Braumann \cite{CARLOS201726} for a study through diffusion processes.
\smallskip

Generalized Feller diffusions of the form \eqref{eq:fellerdiffusion} have been studied for different purposes, we refer for instance to Cattiaux et al. \cite{cattiaux2009} for a study of quasi-stationary distributions, Etheridge \cite{zbMATH02072698} and Hutzenthaler and Wakolbinger \cite{MR2308333} for an infinite-dimensional spatial context, see also Pardoux \cite{MR3496029} and the references therein for a study of the underlying genealogy. 

In a seminal work, Lambert \cite{MR2134113} introduced the so-called \textit{logistic} CB process, in which the diffusive branching term in \eqref{eq:fellerdiffusion} is replaced by the dynamics of a general $\mathrm{CB}$ process, thus allowing for positive jumps (both large and small, including compensated jumps), while the drift term is given by the function $\hat{\Psi}:[0,\infty)\to[0,\infty)$, $x\mapsto \hat{\mathrm{a}}x^2$.

\smallskip

A drawback of these models is that the density-dependence destroys in general the branching property and many arguments from the theory of branching processes fail to apply. This lack of structure renders the study of these processes challenging. We introduce a specific class of CBDIs that allows for a general branching dynamic and for which the density-dependence is governed by a function $\hat{\Psi}$ of \textit{L\'evy-Khintchine} form (whose definition and properties are recalled in the preliminaries). In this setting, the lost branching property is replaced by a certain structure of duality, called \textit{Laplace duality}. This will enable us to work within a framework recently developed in Foucart and Vidmar \cite{foucartvidmar2025}. 
\smallskip

Our main goal is to identify settings in which, in contrast to classical branching processes, the boundaries $\infty$ and $0$ are not necessary absorbing (though they may still be accessible). Putting this in other words, we are interested in the question of  which deterministic competition and cooperation forces, among L\'evy-Khintchine drift-interaction, enable the branching process to escape from $\infty$ (explosion) or $0$ (extinction). 
\smallskip

The phenomenon whereby a boundary $\infty$ is non-absorbing and inaccessible (that is, $\infty$ is an entrance boundary) is commonly referred to in the literature as \textit{coming down from infinity}. This topic has received considerable attention, in particular in coalescent theory, see e.g. Berestycki \cite{Beres2} and the references therein. 

For recent works concerning generalized CB processes, we refer the reader to Le and Pardoux~\cite{zbMATH07553689}, Leman and Pardo~\cite{zbMATH07317338}, Li et al.~\cite{zbMATH07120715}, Ma et al.~\cite{zbMATH07453013}, Palau and Pardo \cite{zbMATH06836271}, and Marguet and Smadi~\cite{zbMATH07373488}. 
Closely related to our work are studies of boundary behavior and the construction of extensions of positive self-similar Markov processes and other time-changed L\'evy processes; see Vuolle-Apiala~\cite{zbMATH00681374}, Rivero~\cite{zbMATH02209766, zbMATH05232980}, Fitzsimmons~\cite{zbMATH05070581}, Caballero and Chaumont~\cite{zbMATH05043274},  Barczy and D\"oring~\cite{zbMATH06245578}, as well as Baguley et al.~\cite{arXiv:2410.07664} and D\"oring and Kyprianou~\cite{zbMATH07226359}. 

We also note that the Laplace duality framework bears some resemblance to the class of generalized Wright–Fisher processes with frequency-dependent selection studied by González-Casanova and Spanò~\cite{zbMATH06873684}. In that setting, a \textit{moment duality} with certain fragmentation–~coalescence exchangeable processes arises. Ancestral selection graphs (see \cite{zbMATH06873684} and Etheridge~\cite{zbMATH05819412} for background) encode selection at the level of individuals. For CBDIs, the question of understanding L\'evy–Khintchine interactions at the individual level, and their emergence from discrete models, is left for future investigation.
\smallskip

Let us describe in more detail the framework and the main results. For the class of $\mathrm{CBDI}$ processes considered here, the infinitesimal generator takes the form
\begin{equation}\label{eqgenintro}
	\mathcal{X}f(x) := \mathscr{L}^{\Psi}f(x) - \hat{\Psi}(x) f'(x),
	\quad x \in [0,\infty)
\end{equation}
where  $\mathscr{L}^{\Psi}$ denotes the generator of a $\mathrm{CB}$ process with branching mechanism $\Psi$ (which encodes both jumps and continuous dynamics), and $\hat{\Psi}$ is a L\'evy–Khintchine function of spectrally positive type, hence another branching mechanism.
\smallskip

Background on $\mathscr{L}^{\Psi}$ and L\'evy-Khintchine functions will be provided in Section~\ref{sec:preliminaries}.  Let us recall that they may tend either to $+\infty$ or to $-\infty$, and may or may not change sign. For instance, when~$\hat{\Psi}$ in~\eqref{eqgenintro}  changes sign, this reflects the fact that the density dependence is negative (competition) for sufficiently large population sizes but becomes positive (cooperation) when the population size is low; see Figure~\ref{fig:drift}.

\begin{figure}[h!]
    \centering
\includegraphics[width=0.52\linewidth]{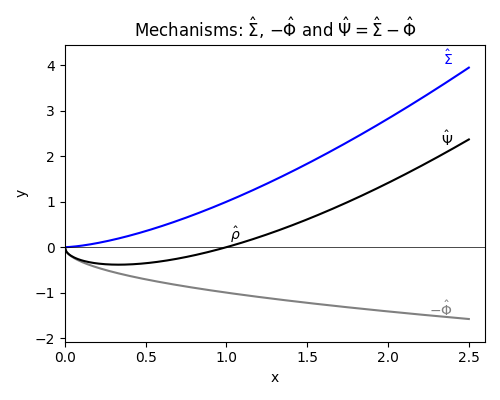}
    \caption{Three drift-interaction mechanisms: $\hat{\Sigma}$ is pure competition, $-\hat{\Phi}$ pure cooperation, $\hat{\Psi}=\hat{\Sigma}-\hat{\Phi}$ is a mixture and $\hat{\rho}$ is its largest zero. The behavior of $\hat{\Sigma}$ near $\infty$ reflects the competition pressure at high population sizes. The behavior of $\hat{\Phi}$ near $0$ reflects the strength of cooperation when the population size becomes very small.} 
    \label{fig:drift}
\end{figure}

Operators of the form \eqref{eqgenintro} arise as examples of generators satisfying a \textit{Laplace duality} relationship, as studied in Foucart and Vidmar \cite{foucartvidmar2025}. More precisely,  if one exchanges the roles of the mechanisms $\Psi$ and $\hat{\Psi}$ and denotes by $\mathcal{Y}$ the generator of a $\mathrm{CBDI}(\hat{\Psi},\Psi)$ process, that is, \begin{equation}\label{eq:genY}\mathcal{Y}f(y):=\mathscr{L}^{\hat{\Psi}}f(y)-\Psi(y)f'(y), \quad y\in [0,\infty),
\end{equation}
then, defining the maps $\ee^y$ and $\ee_x$ by
\[\ee^y(x):=e^{-xy}=:\ee_x(y),\ x,y\in (0, \infty),\]
it will be easily checked (see Section \ref{sec:dualitygenlevel})  that
\begin{align}
   \label{eq:gendualityintro}
\mathcal{X}\ee^y(x)&=\mathcal{Y}\ee_x(y), \ x,y \in  (0, \infty).
\end{align}
This identity is the Laplace duality relationship at the level of generators.
\smallskip

Pure competition in the operator $\mathcal{X}$, \eqref{eqgenintro}, that is to say, $\hat{\Psi}\geq 0$, corresponds then to a (sub)critical branching mechanism for the dual operator $\mathcal{Y}$, \eqref{eq:genY}, whereas, conversely, cooperation, i.e. $-\hat{\Psi}\geq 0$ in some neighbourhood of $0$, corresponds to supercriticality.
\smallskip

We shall see that the operators $\mathcal{X}$ and $\mathcal{Y}$ characterize the infinitesimal dynamics of $\mathrm{CBDI}$ processes with mechanisms
$(\Psi,\hat{\Psi})$ and $(\hat{\Psi},\Psi)$,
\textit{as long as they evolve in $(0,\infty)$}. We call\textit{ minimal} the  $\mathrm{CBDI}(\Psi,\hat{\Psi})$ process stopped upon reaching one of its boundaries, which are therefore \textit{absorbing}. Its unique existence will be established through a stochastic equation and follows mainly from the fact that any L\'evy–Khintchine function is locally Lipschitz on $(0,\infty)$. 
\smallskip

When the mechanisms $\Psi$ and $\hat{\Psi}$ are not Lipschitz on $[0,\infty)$, i.e. they satisfy $\Psi'(0+)=\hat{\Psi}'(0+)=-\infty$, there are possibly several $[0,\infty]$-valued Markov processes associated to $\mathcal{X}$ and $\mathcal{Y}$ as different behaviors at the boundaries $0$ and $\infty$ may exist. We shall see, after specifying their behaviors, how the relationship~\eqref{eq:gendualityintro} can be integrated out at the level of the semigroups. Namely, let $\big(X, (\mathbb{P}_x)_{x\in [0,\infty]}\big)$ be a $\mathrm{CBDI}(\Psi,\hat{\Psi})$ process and $\big(Y,(\mathbb{P}^{y})_{y\in [0,\infty]}\big)$ a $\mathrm{CBDI}(\hat{\Psi},\Psi)$ process. Then, under appropriate boundary conditions and conventions for $0\cdot \infty$ and $\infty\cdot 0$, one has
\begin{equation}\label{eq:dualsemigroupintro}
\mathbb{E}_x[e^{-X_t y}] = \mathbb{E}^{y}[e^{-x Y_t}], \quad x,y\in [0,\infty],\ t\in [0,\infty).
\end{equation}
The duality relationship \eqref{eq:dualsemigroupintro} serves as our principal tool for constructing Fellerian extensions of $\mathrm{CBDIs}$. 
Accordingly, we shall primarily adopt an approach based on semigroups, martingale problems, and generators, rather than relying on excursion theory or stochastic differential equations, although the latter will be invoked at certain points.
\smallskip

We first look for an extension at infinity. We start by establishing that $\mathrm{CBDIs}$ with mechanisms $\Psi,\hat{\Psi}$ such that $\Psi'(0+)\in (-\infty,\infty)$, i.e. the jumps have finite mean, and $\int_1^{\infty}\frac{\ddr u}{\hat{\Psi}(u)}<~\infty$, have always $\infty$ as an \textit{instantaneous entrance} boundary (the process can start from $\infty$, leaves it instantaneously and never returns to it). We then construct, Theorem \ref{thm2infty},  an extension $X^{\mathrm{e}\infty}$ by taking limits as $n$ goes to $\infty$ in $\mathrm{CBDI}(\Psi^{n},\hat{\Psi})$ processes $X^{\mathrm{e}\infty,(n)}$, started from $\infty$, and whose jump measure is truncated at level $n$ (and with therefore finite mean $|(\Psi^{n})'(0)|<\infty$). The following weak convergence in the Skorokhod space is then established:
\[X^{\mathrm{e}\infty,(n)}\underset{n\rightarrow \infty}{\Longrightarrow} X^{\mathrm{e}\infty} \text{ in } \mathbbm{D}_{[0,\infty]}.\]
The process $X^{\mathrm{e}\infty}$ is a $[0,\infty]$-valued Feller process whose evolution on $(0,\infty)$ is governed by $\mathcal{X}$ in \eqref{eqgenintro},  and for which the boundary point $\infty$ may be both visited and left continuously. 
\smallskip

The behavior at infinity of $X^{\mathrm{e}\infty}$ is then investigated.  It will depend solely on the large-jump behavior and the competition component, encoded respectively by a Bernstein function, i.e. the Laplace exponent of a subordinator,
$\Phi$ and a positive L\'evy-Khintchine one $\hat{\Sigma}$, such that $\Psi(y)\underset{y\to 0}{\sim} -\Phi(y)\rule{0pt}{12pt}$ and $\hat{\Psi}(x)\underset{x\to \infty}{\sim}\hat{\Sigma}(x)$.
\begin{table}[h!]
\begin{center}
\renewcommand{\arraystretch}{1.5} 
\setlength{\tabcolsep}{11pt} %
\begin{tabular}{|c|c|}
\hline
\multicolumn{1}{|c|}{Large jumps} & \multicolumn{1}{c|}{Competition} \\
\hline
$\Psi(y)\underset{y\to 0}{\sim} -\Phi(y)\rule{0pt}{12pt}$ & $\hat{\Psi}(x)\underset{x\to \infty}{\sim}\hat{\Sigma}(x)\rule{0pt}{12pt}$ \\
\hline
\end{tabular}
\caption{\footnotesize{Large jumps facing competition}}
\label{tablelargejumpscompet}
\end{center}
\end{table}
\vspace{-4mm}

We design, with the help of $\Phi$ and  $\hat{W}$ the scale function associated to $\hat{\Sigma}$ (whose definition is recalled in Section~\ref{sec:scalefunction}), the following $[0,\infty]$-valued parameters: \begin{equation}\label{eq:thetaintro}\overline{\theta}_{\,\scriptscriptstyle\Phi, \hat{\Sigma}}:=\underset{x\rightarrow \infty}{\limsup} \, x\int_0^{\infty}e^{-zx}\frac{\Phi(z)}{z}\hat{W}(z)\ddr z, \quad \underline{\theta}_{\,\scriptscriptstyle\Phi, \hat{\Sigma}}:=\underset{x\rightarrow \infty}{\liminf}\, x\int_0^{\infty}e^{-zx}\frac{\Phi(z)}{z}\hat{W}(z)\ddr z.
\end{equation}
They characterize wether the process $X^{\mathrm{e}\infty}$ can enter from $\infty$. Namely, we establish, see Theorem~\ref{thm3infty}, that if $\overline{\theta}_{\,\scriptscriptstyle\Phi, \hat{\Sigma}}<1$, then $\infty$ is non-absorbing, whereas if $\underline{\theta}_{\,\scriptscriptstyle\Phi, \hat{\Sigma}}>1$, then $\infty$ is absorbing. In smooth cases, both parameters coincide and a phase transition occurs at $1$. The question of accessibility of $\infty$ is addressed in the sequel.
 \vspace*{1mm}

We next look at an extension at zero. We first show that if $\hat{\lambda}=-\hat{\Psi}(0)>0$ and $\Psi$ has no diffusive part, then the $\mathrm{CBDI}(\Psi,\hat{\Psi})$ process has $0$ instantaneous entrance. We then construct, Theorem \ref{thm2zero}, an extension $X^{\mathrm{e}0}$ by taking the limit as $n$ goes to $\infty$ in processes $X^{\mathrm{e}0,(n)}$ started from $0$ whose drift-interaction term $\hat{\Psi}_n$ satisfies $-\hat{\Psi}_n(0)=\hat{\lambda}_n\downarrow 0$ as $n \to \infty$, and show:

\[X^{\mathrm{e}0,(n)}\underset{n\rightarrow \infty}{\Longrightarrow} X^{\mathrm{e}0} \text{ in } \mathbbm{D}_{[0,\infty]}.\]
The process $X^{\mathrm{e}0}$ is a $[0,\infty]$-valued Feller process whose evolution on $(0,\infty)$ is prescribed by $\mathcal{X}$ in \eqref{eqgenintro}. Heuristically, the extension at $0$ is constructed through processes with a constant immigration rate $\hat{\lambda}_n$ vanishing at the limit as $n$ goes to $\infty$. The limiting process this way sees no immigration but might have an infinitesimal ``reservoir" from which the population can be started or resurrected. 
\smallskip

The behavior of $X^{\mathrm{e}0}$ at $0$ is depending only on the small jumps and the cooperation component, encoded respectively by mechanisms $\Sigma$ and $\hat\Phi$, Table~\ref{tablesmalljumpscoop}.
\begin{table}[h!]
\begin{center}
\renewcommand{\arraystretch}{1.5} 
\setlength{\tabcolsep}{11pt}
\begin{tabular}{|c|c|}
\hline
\multicolumn{1}{|c|}{Small jumps} & \multicolumn{1}{c|}{Cooperation} \\
\hline
$\Psi(y)\underset{y\to\infty}{\sim}\Sigma(y)$ &  $\hat{\Psi}(x)\underset{x\to 0}{\sim}-\hat\Phi(x)$. \\
\hline
\end{tabular}
\caption{\footnotesize{Small jumps facing cooperation}}
\label{tablesmalljumpscoop}
\end{center}
\end{table}
\vspace{-3mm}

We design, with the help of $\Sigma$ and $\hat{U}$ the potential measure associated to $\hat{\Phi}$ (its definition is recalled in Section~\ref{sec:potentialmeasure}), the following $[0,\infty]$-valued parameters:
\begin{equation}\label{eq:rhointro}
\rhoup:= \underset{x\rightarrow 0}{\limsup} \, x \int_{0}^{\infty}e^{-xz}\frac{\Sigma(z)}{z}\hat{U}(\ddr z), \quad 
\rhodown:= \underset{x\rightarrow 0}{\liminf}\, x\int_{0}^{\infty}e^{-xz}\frac{\Sigma(z)}{z}\hat{U}(\ddr z).
\end{equation}
These parameters characterize the behavior at zero for the extended processes $X^{\mathrm{e}0}$: we show, Theorem~\ref{thm3zero}, that if $\overline{\rho}_{\,\scriptscriptstyle\Sigma, \hat{\Phi}}<1$, then $0$ is non-absorbing, whereas if $\underline{\rho}_{\,\scriptscriptstyle\Sigma, \hat{\Phi}}>1$, then $0$ is absorbing. Both limits coincide in some explicit cases  and another phase transition occurs.
\smallskip

As a matter of fact, the extended processes we construct, $X^{\mathrm{e}\infty}$ and $X^{\mathrm{e}0}$, are in Laplace duality, that is, they satisfy \eqref{eq:dualsemigroupintro}, with the \textit{minimal} $\mathrm{CBDI}(\hat{\Psi},\Psi)$ process $Y^{\mathrm{m}}$ (i.e. the process which, when started from any of its boundaries, remains there). 
\smallskip

The duality relation entails that any progress on the extended $\mathrm{CBDI}(\Psi,\hat\Psi)$ yields information on the minimal $\mathrm{CBDI}(\hat\Psi,\Psi)$, and vice-versa, thereby deepening our understanding of this entire class of processes.  More precisely, the identity \eqref{eq:dualsemigroupintro} will entail that non-absorption at $\infty$ (resp. $0$) for the extended process, $X^{\mathrm{e}\infty}$ (resp. $X^{\mathrm{e}0}$) corresponds to the accessibility of $0$ (resp. $\infty$) for $Y^{\mathrm{m}}$. So that in order to study the phenomenon of accessibility of $\infty$ for the $\mathrm{CBDI}(\Psi,\hat{\Psi})$, we instead analyze the dual problem of non-absorption at $0$ for the $\mathrm{CBDI}(\hat{\Psi},\Psi)$.

Specifically, the \textit{dual} parameters $\rhodowndual, \rhoupdual$  help us to complete the classification of the boundaries as follows: the minimal process hits $\infty$ with positive probability (\textit{explosion}) if $\rhoupdual<1$ and does not when $\rhodowndual>1$, see Theorem~\ref{thm:accessibilityinftybyrho}. 

Similarly for the boundary $0$, the minimal process $X^{\mathrm{m}}$ hits $0$ with positive probability (\textit{extinction}) if $\thetaupdual<1$ and does not when $\thetadowndual>1$, Theorem~\ref{cor:accessibility0bytheta}.  
\smallskip

The combination of the conditions for non-absorption and accessibility leads to the following boundary classification, Corollary~\ref{thm:classification}. 
\medskip

For the extension at $0$, whether $X^{\mathrm{e}0}$ has $0$ accessible and absorbing depends on the balance between cooperation, $\hat \Phi$, and ``natural deaths", encoded by $\Sigma$: 

\begin{itemize}
	\item If $\rhodown > 1>\thetaupdual$, then $0$ is an \emph{exit} (i.e.\ the boundary is accessible and absorbing). In this regime, cooperation is not strong enough to prevent extinction.
	
	\item If  $\rhoup < 1$ and $\thetaupdual < 1$, then $0$ is \emph{regular} (i.e.\ the boundary is accessible and non-absorbing). In this case, cooperation and natural deaths compensate each other, leading to local extinctions.
	
	\item If $\thetadowndual > 1>\rhoup$, then $0$ is an \emph{entrance} (i.e.\ the boundary is inaccessible and non-absorbing). Here, cooperation is sufficiently strong for the process to start from $0$ without hitting it thereafter.
\end{itemize}

\indent Similarly, for the extension at $\infty$, the behaviour of $X^{\mathrm{e}\infty}$ depends on the interplay between competition, $\hat\Sigma$, and large reproduction events, encoded by $\Phi$:
\begin{itemize}
	\item If $\thetadown > 1>\rhoupdual$, then $\infty$ is an \emph{exit}.
	
	\item If $\thetaup < 1$ and $\rhoupdual < 1$, then $\infty$ is \emph{regular}.
	
	\item If $\rhodowndual > 1>\thetaup$, then $\infty$ is an \emph{entrance}.
\end{itemize}

Explicit values of $\thetaup, \thetadown$ and $\rhoup, \rhodown$,  within the regularly varying setting, will be provided in Section~\ref{sec:examples}. They are covering the various regimes.  Moreover, in the regular cases, when the mechanisms $\hat{\Sigma}$ governing competition, corresponding to the natural deaths (compensated jumps) of the dual process, and $\Phi$ governing large jumps, corresponding to cooperation in the dual process, are regularly varying, we show that the extended process admits a boundary that is  non-sticky (i.e., the level set has zero Lebesgue measure) and regular for itself (that is, when started from the boundary, the process returns to it instantaneously). 
\smallskip

The one-to-one correspondence between non-absorptivity and accessibility provided by a duality relationship can be traced back at least to the work of L\'evy \cite{zbMATH03052578}, who observed a duality relationship between two Brownian motions on $[0,\infty)$, one reflected at $0$ and the other absorbed. We refer to Siegmund \cite{MR0431386} and Cox and R\"osler~\cite{MR724061}. 

In a sense, the duality allows one to transfer \textit{entrance} properties of one process (here the $\mathrm{CBDI}(\Psi,\hat{\Psi})$) into \textit{exit} properties for the dual process (here the $\mathrm{CBDI}(\hat{\Psi},\Psi)$). This correspondence has recently been exploited in various frameworks, primarily within the discrete state space setting; see, for example, Gonzalez et al.~\cite{zbMATH07458586} and Berzunza-Ojeda and Pardo~\cite{zbMATH08062150}, where processes satisfy a moment duality relationship, as well as H\'enard~\cite{zbMATH06496393} and Kukla and M\"ohle~\cite{zbMATH06837778}, in which Siegmund duality plays a central role.
For a continuous-state space framework, we also refer to Foucart and Vidmar~\cite{foucartvidmar}, where Laplace duality is used to study a class of branching processes with random collisions.
\smallskip

The techniques developed in this article draw partly on arguments introduced in three closely related contexts: the study of the logistic CB process and its extension at infinity \cite{MR3940763}; the characterization of absorption or non-absorption at infinity for exchangeable fragmentation–coalescence processes \cite{zbMATH07493833}; and the construction of extensions beyond fixation (i.e. upon hitting the boundary point~$1$) for generalized Wright–Fisher processes~\cite{zbMATH07734715}.
\smallskip

The key distinctions from the setting of logistic $\mathrm{CB}$ processes are as follows. First, logistic $\mathrm{CB}$ processes benefit from a representation as time-changed generalized Ornstein–Uhlenbeck processes \cite{MR3940763,Lambert}. Properties of the latter, such as their potential measure, were then used in \cite{MR3940763} to  study the explosion of logistic processes. No such transformation exists for a general $\mathrm{CBDI}$. Second, in the logistic case, the dual operator $\mathcal{Y}$  turns out to be the generator of a diffusive $\mathrm{CBDI}$, solution to the equation (of the form \eqref{eq:fellerdiffusion}):
$$\ddr Y_t=\sqrt{2\hat{\mathrm{a}}Y_t}\ddr B_t-\Psi(Y_t)\ddr t, \ \ Y_0=y\in [0,\infty).$$
Feller’s tests for classifying the point $0$ of  $Y$ were therefore available and of great help in the study of logistic CBs \cite[Section~5]{MR3940763}. 
\smallskip

In the present framework, both processes $X$ and $Y$ exhibit non-trivial positive jumps, their generators contain non-local parts, and no general explicit theory provides the classification of the boundaries. There is also no plain time-change relationship between $\mathrm{CBDIs}$ and well-known Markov processes outside the selfsimilar setting and the logistic case. 
\smallskip

We investigate the boundary behavior of  $\mathrm{CBDIs}$ with the help of Lyapunov functions, see e.g. Rebotier~\cite[Theorem A]{rebotier} and the references therein for background. These functions are tailored to the generator $\mathcal{X}$, \eqref{eqgenintro}, as they make essential use of the L\'evy-Khintchine structure of the drift $\hat{\Psi}$. The parameters $\thetaup$ and $\thetadown$, used to classify $\infty$ are obtained through analogous arguments to those developed in a different setting in~\cite[Sections 3.2 and 3.3]{zbMATH07493833}.
In our present framework, the absence of negative jumps, together with Laplace duality, allows us to express,  see~\eqref{eq:thetaintro}, the parameters~$\thetaup$ and~$\thetadown$ in terms of the scale function~$\hat{W}$ associated with~$\hat{\Sigma}$ and the Bernstein function~$\Phi$. This expression also sheds light on why these parameters capture information about extinction/non-extinction for the dual (minimal) $\mathrm{CBDI}(\hat{\Psi},\Psi)$-process $Y^{\mathrm{m}}$. Indeed, in heuristic terms, $\hat{W}$ characterizes the probability that the classical $\mathrm{CB}(\hat{\Sigma})$-process exits the interval $(0,z)$ through $0$, whereas $\Phi$ governs how it is deterministically pushed away from the boundary $0$. 
\smallskip

Finally, the arguments designed for defining the extension at~$0$ are close in spirit to those in~\cite{zbMATH07734715} for dealing with $\Lambda$-Wright-Fisher processes. The validity of our construction, and consequently of the inherited boundary classification, relies on the crucial assumption that the mechanism $\Psi$ contains no diffusive component. This ensures that the process dynamics are governed exclusively by jumps and drift which will simplify the study of the behavior of the process when its starting point tends to $0$. 
\medskip

The parameters~$\rhoup, \rhodown$ are first introduced in order to analyze the boundary point~$0$. They are subsequently employed to study the boundary $\infty$ of the dual $\mathrm{CBDI}$ process.  Similarly as for $\thetaup$, the expression \eqref{eq:rhointro} with the potential measure $\hat{U}$ shows that these quantities encode, in a certain sense, the interplay between the large jumps of the dual $\mathrm{CBDI}(\hat{\Psi},\Psi)$ and the competition pressure $\Sigma$. This explains intuitively the role of $\rhoup, \rhodown$ in determining the accessibility of $\infty$ of the dual process. 
\medskip

The article is organized as follows. In Section~\ref{sec:preliminaries}, we introduce the notation used throughout the text, along with the definitions and terminology for classifying boundaries. Section~\ref{sec:terminology} reviews basic properties of L\'evy–Khintchine functions and $\mathrm{CB}$ processes. The notions of scale functions and potential measures are recalled in Section~\ref{sec:potentialelements}. Minimal $\mathrm{CBDI}$ processes are introduced in Section~\ref{sec:minimal}. Sections~\ref{sec:prooftheorem2infinity} and \ref{sec:absorptionatinfinity} are devoted to the boundary at $\infty$: we first construct a Feller process extending the minimal $\mathrm{CBDI}$ process, and then investigate the behavior at this boundary with the help of $\thetaup$ and $\thetadown$. Sections~\ref{sec:extensionat0} and \ref{sec:proofofthm3zero} address the boundary at $0$ and follow a similar structure. Finally, in Section~\ref{sec:examples}, we sum up the classification obtained with  $\thetaup, \thetadown$ and $\rhoupdual,\rhodowndual$ and apply our results to regularly varying mechanisms for which these parameters are explicit.
\section{Preliminaries}\label{sec:preliminaries}
\textbf{Notation}. 
For two positive functions $f,g$, we write $f\asymp g$  when there exist $c_1,c_2\in(0,\infty)$ such that $c_1f\leq g\leq c_2f$ and $f\underset{a}{\sim} g$ if $\underset{x\rightarrow a}{\lim} \frac{f(x)}{g(x)}=1$. 
We use the classical conventions $\inf \emptyset=\infty$ and $\sup \emptyset=0$.  We also agree on $1/\infty=0$ and $1/0=\infty$. For any condition $\mathbb{H}$ we denote its negation by $\neg \mathbb{H}$. When a limit is increasing (resp. decreasing) we write $\lim \! \uparrow$ (resp. $\lim \! \downarrow$).
\smallskip

Let $[0,\infty]$ be the extended half-line. We equip it with the compact metric $d(x,y):=|e^{-x}-e^{-y}|$, with the convention $e^{-\infty}=0$. Convergence in this metric corresponds to the usual convergence in $[0,\infty)$, with $\infty$ naturally treated as a boundary point.  For any real function $f$, provided the limits exist in $[0,\infty]$, we write
\begin{center}
$f(a+)=\underset{\epsilon \to 0^+}{\lim} f(a+\epsilon)$ and $f(b-)=\underset{\epsilon \to 0^+}{\lim} f(b-\epsilon)$. 
\end{center}
\smallskip

We denote by $\mathrm{B}_{[0,\infty]}$ the Borelian functions (and sets) defined on (included in) $[0,\infty]$. Similarly, $\mathrm{B}_{[0,\infty)}$ denotes the Borel subsets of $[0,\infty)$. The space of continuous functions on $[0,\infty]$, hence with a \textit{finite} limit at $\infty$, is denoted by $\mathrm{C}([0,\infty])$. For any $f\in \mathrm{C}([0,\infty])$, we set $\|f\|_{\infty}:=\sup_{x\in [0,\infty]}|f|$. The subspace of continuous functions vanishing at $\infty$ is $\mathrm{C}_0$. The domain of definition of a function $f$ is denoted by $\mathcal{D}_f$. A function is said to be $\mathrm{C}^1$ (or $\mathrm{C}^2$), when it is (twice) continuously differentiable on its domain of definition. The space $\mathrm{C}_b^2$ is the set of functions that are bounded together with their derivatives. The space of continuously differentiable functions on $(0,\infty)$ is $\mathrm{C}^1((0,\infty))$. Similarly, $\mathrm{C}^2((0,\infty))$ gathers the twice continuously differentiable functions. Last, we denote by $\mathrm{C}_c^2((0,\infty))$ the space of $\mathrm{C}^{2}$ functions whose support is compact and included in $(0,\infty)$. 
\smallskip

The space of $[0,\infty]$-valued càdlàg paths defined on $[0,\infty)$ is $\mathbbm{D}_{[0,\infty]}$. By convention, for any $X\in \mathbbm{D}_{[0,\infty]}$, we set $X_{0-}:=X_{0}$. The convergence in Skorokhod sense is denoted by $\Longrightarrow$.  We refer e.g. to Ethier-Kurtz's book \cite[Chapter 3]{EthierKurtz} for background.
\smallskip

Last, indeterminate products of the form $0\cdot \infty$ and $\infty\cdot 0$ will occur when expressing the Laplace duality at the boundaries: 
\begin{center}
$\mathbb{E}_\infty[e^{-X_ty}]=\mathbb{E}^{y}[e^{-\infty\cdot Y_t}]$ and $\mathbb{E}_0[e^{-X_ty}]=\mathbb{E}^{y}[e^{-0\cdot Y_t}]$. 
\end{center}
In order to interpret the products in the right-hand sides of these equalities when $Y_t=0$ and $Y_t=\infty$ respectively, we use the conventions given in Table~\ref{conventiontable}, following the notations introduced in~\cite[Definition 3.10]{foucartvidmar2025}: 
\begin{table}[htpb]
\begin{center}
\begin{tabular}{|c|c|c|c|}
\hline
$0^+\cdot \infty$ &  $\infty\cdot 0^+$ & $\infty^-\cdot 0$ & $0\cdot \infty^{-}$ \\
\hline
$0\cdot \infty=\infty$ &  $\infty \cdot 0=\infty$ &   $\infty \cdot 0=0$  & $0\cdot \infty=0$
\\
\hline
\end{tabular}
\vspace*{2mm}
\caption{Conventions for $0\times \infty, \infty\times 0$}
\label{conventiontable}
\end{center}
\end{table}

\noindent  We refer to \cite[Theorem 3.13-(iv)]{foucartvidmar2025} for a full explanation of their role in the Laplace duality relationship. We shall mainly work in the article with the conventions $0^+\cdot \infty,\ \infty\cdot 0^-$. Observe that under the latter, the following continuity properties at the boundaries hold
\[\mathbb{E}_{\infty-}[e^{-X_ty}]=\mathbb{E}_{\infty}[e^{-X_ty}] \ \text{ and } \ \mathbb{E}_{0+}[e^{-X_ty}]=\mathbb{E}_{0}[e^{-X_ty}].\]
\subsection{Terminology of boundaries}\label{sec:terminology}
Let $(X, (\mathbb{P}_x)_{x\in [0,\infty]})$ be a $[0,\infty]$-valued càdlag Markov process  with no negative jumps. Let $a,b\in [0,\infty]$, we
set 
\[\sigma_a^{-}:=\inf\{t\geq 0: X_t\leq a\} \text{ and } \sigma_b^{+}:=\inf\{t\geq 0: X_t\geq b\}.\]

\begin{defstar}\label{def:accessibility}\
\begin{enumerate}
\item \textbf{Accessibility.} The boundary $\Delta\in \{0,\infty\}$ is said to be accessible if $$\forall x\in (0,\infty), \ \mathbb{P}_x(\exists t\geq 0: X_t=\Delta)>0.$$ 
\item \textbf{Absorption.} The boundary $\Delta\in \{0,\infty\}$ is absorbing if $$\mathbb{P}_{\Delta}(\exists t\geq 0: X_t\neq \Delta)=0.$$ 
\end{enumerate}
\end{defstar}
\medskip
\noindent Define the following first hitting times 
\begin{center}
	$\sigma^{+}_{\infty}:=\inf\{t\geq 0: X_{t-} \text{ or } X_t=\infty\}$  and $\sigma^{-}_{0}:=\inf\{t\geq 0: X_t=0\}.$
\end{center}
One has a.s. $\sigma_{\infty}^{+}=\underset{b\rightarrow \infty}{\lim}\!\uparrow \sigma_b^{+}$ and by the absence of negative jumps, 
$$\sigma^{-}_{0}=\underset{a\rightarrow 0}{\lim}\!\uparrow \sigma_a^{-} =\inf\{t>0: X_t\leq 0\}.$$

Notice that the boundary $\infty$ (resp. $0$) is then accessible if and only if $\sigma^{+}_\infty<\infty$ with $\mathbb{P}_x$-positive probability (resp. $\sigma_0^{-}<\infty$) for all $x\in (0,\infty)$. 
\smallskip

Furthermore, if the process $X$ is strong Markov and a boundary, say $\infty$, is accessible \textit{absorbing}, then the strong Markov property at the stopping time $\sigma_\infty^{+}$ ensures that \[\mathbb{P}_{x}\big(\exists t\geq 0: X_{t+\sigma^{+}_{\infty}}\neq \infty, \sigma^{+}_\infty<\infty\big)=0,\ \forall x\in [0,\infty].\] 
\smallskip

We shall also use the terminology of Feller for classifying the boundaries, see e.g. Durrett's book \cite[Section 6]{MR1398879}.

\begin{defstar}
	The boundary point $\Delta$ is classified as follows:
	\begin{itemize}
		\item \textbf{Entrance} if, when $X$ is started from $\Delta$, it leaves $\Delta$ and never returns almost surely. 
		Equivalently, $\Delta$ is non-absorbing and inaccessible from the interior.
		
		\item \textbf{Exit} if $X$ hits $\Delta$ with positive probability and, once there, remains at $\Delta$ forever. 
		That is, $\Delta$ is accessible and absorbing.
		
		\item \textbf{Regular} if $X$ hits $\Delta$ with positive probability and can subsequently leave it with positive probability; in other words, $\Delta$ is accessible and non-absorbing. Precisely, for any $x\in (0,\infty)$,
		\[
		\mathbb{P}_x\!\left(\exists\, 0 < s < t < \infty : X_s=\Delta \text{ and } X_t \neq \Delta \right) > 0.
		\]
		
		\item \textbf{Natural} if $\Delta$ is inaccessible and absorbing.
	\end{itemize}
\end{defstar}

In this article, we encounter only non-absorbing boundaries (entrance or regular) that are \textit{instantaneous} and \textit{continuous}. 

\smallskip

We now recall these notions, together with non-stickiness and regularity-for-itself. We refer e.g. to Bertoin's book \cite[Chapter IV]{Bertoin96}.
\vspace{2mm}

\begin{defstar}
	\label{def:boundary}
Let $\Delta \in \{0,\infty\}$.

\begin{enumerate}[(a)]
	\item
	$\Delta$ is \textbf{instantaneous} if
	\[
	\mathbb{P}_{\Delta}(T_\Delta=0)=1,
	\qquad 
	T_\Delta:=\inf\{t>0: X_t\neq \Delta\}.
	\]
	Equivalently,
	\[
	\mathbb{P}_{\Delta}\big(\forall t>0,\ \exists s\in(0,t): X_s\neq \Delta\big)=1.
	\]

\item  \label{def:continuous} 
$\Delta$ is \textbf{continuous} if, under $\mathbb{P}_\Delta$, the process is either identically equal to $\Delta$ (absorbing case) or leaves $\Delta$ without a jump, that is, the jump measure of the process vanishes at~$\Delta$.

Equivalently, in the non-absorbing case, if one denotes by $G$ the set of strictly positive left-end points of excursion intervals away from $\Delta$, then
\[\mathbb{P}_{\Delta}\left( \forall s\in G, \ X_s=\Delta \right)=1.\]
\item \label{def:nonsticky} $\Delta$ is \textbf{non-sticky} if
\begin{equation}\label{eq:nonsticky}
\mathbb{P}_x(X_t=\Delta)=0,\ \forall t\in (0,\infty), \ x\in [0,\infty].\end{equation}
Equivalently, the Lebesgue measure of the $\Delta$-level set $\{t\geq 0: X_t\in \Delta\}$ is zero. 
\item $\Delta$ is \textbf{regular-for-itself} if 
\begin{equation}\label{eq:regularfor}\mathbb{P}_{\Delta}(R^{\Delta}=0)=1, \text{ with } R^\Delta:=\inf\{t>0: X_{t-}=\Delta \text{ or } X_t=\Delta\}.
\end{equation}
In words, the process started from $\Delta$, returns immediately to $\Delta$. 
\end{enumerate}
\end{defstar}
We emphasize that when a process is stopped upon reaching a regular boundary, that boundary becomes absorbing. This contrasts with an exit boundary for which no non-trivial \textit{continuous} extension of the process beyond the hitting time is possible. The only possible way to leave an exit boundary would be via a jump from the boundary point back into the state space. We refer, for instance, to Pakes \cite{MR1241926} for such a study in the setting of explosive discrete branching processes. We do not consider such additional dynamics in this article and focus on the interplay between the jumps of the branching dynamics and the drift.

\subsection{Background on L\'evy-Khintchine functions and CB processes}
\subsubsection{L\'evy-Khintchine functions}\label{sec:LKfunction}
A branching mechanism is a function $\Psi$  of the following L\'evy-Khintchine form 
\begin{equation}\label{branchingmechanism}\Psi(x)=\mathrm{a}x^2-\gamma x-\lambda+\int_0^{\infty} \big(e^{-ux}-1+ux\mathbbm{1}_{(0,1]}(u)\big)\pi(\ddr u),\quad x\in [0,\infty),\end{equation}
with $\mathrm{a}\in [0,\infty)$, $\gamma\in \mathbb{R}$, $\lambda\in [0,\infty)$ and $\pi$ is a L\'evy measure, i.e. $\int_0^{\infty}1\wedge u^2\pi(\ddr u)<\infty$. 
\smallskip

The parameters $(\pi,\mathrm{a},\gamma,\lambda)$ are called the L\'evy-quadruplet associated to $\Psi$. 
\begin{example}
Let $\beta \in (0,1], \ C,c\in (0,\infty)$. The following maps, defined on $[0,\infty)$, are examples of L\'evy-Khintchine functions, see e.g. Kyprianou's book \cite{Kyprianoubook},
\begin{center}	$x\mapsto Cx^{1+\beta}$,\ $x\mapsto -cx^{\beta}$,\ $x\mapsto  C(1+x)(\log(1+x))^{1+\beta}$, \ $x\mapsto -c\log(1+x)^{\beta}$.
\end{center}
\end{example}
\medskip

We collect in the sequel some fundamental analytical facts about functions of the form~\eqref{branchingmechanism}. Any such function $\Psi$ is convex, continuous on $[0,\infty)$, satisfies $\Psi(0)=-\lambda\in (-\infty,0]$ and thus in particular verifies that $x\mapsto \frac{\Psi(x)}{x}$ is non-decreasing on $(0,\infty)$. 

It is also known, and easily checked, that $\Psi$ has at most quadratic growth at $\infty$, that is to say there exists $C\in (0,\infty)$ such that :
\[\forall x\in [1,\infty), \ |\Psi(x)|\leq Cx^2.\]

Any $\Psi$ admits continuous derivatives of any order on $(0,\infty)$ and $\Psi'$ has the following limits at $0$ and $\infty$:
\begin{center}
$\Psi'(0+)=-\gamma-\int_{1}^{\infty}u\pi(\ddr u)\in [-\infty,\infty)$  and $\Psi'(\infty-)=2\mathrm{a}\cdot \infty+\int_0^1 u\pi(\ddr u)-\gamma\in (-\infty,\infty]$,
\end{center}
with the convention $0\cdot \infty^{-}$, see Table~\ref{conventiontable}.
\smallskip

In addition, $\Psi$ can be decomposed as follows: \begin{equation}\label{eq:decompositionpsi} \Psi=\Sigma-\Phi, 
\end{equation} with positive functions $\Sigma$ and $\Phi$, defined on $[0,\infty)$, with the following forms
\begin{equation}\label{eq:sigmaphi}\Sigma(x)=\mathrm{a}x^2+dx+\int_0^{\infty} \big(e^{-ux}-1+ux\big)\eta(\ddr u), \quad \Phi(x)=\beta x+\int_{0}^{\infty}(1-e^{-xu}) \nu(\ddr u)+\lambda,\end{equation}
for some $d, \beta\geq 0$ and $\eta,\nu$ measures on $(0,\infty)$ such that 
$$\int_{0}^{\infty}(u\wedge u^2)\eta(\ddr u)<\infty \text{ and } \int_{0}^{\infty}(1\wedge u)\nu(\ddr u)<\infty.$$
There is not a unique couple of functions $(\Sigma,\Phi)$ providing a decomposition \eqref{eq:decompositionpsi}. We call \textit{canonical} decomposition, the one with 
\[\eta:=\pi_{|(0,1]},\ d:=\gamma^{-},\quad  \nu:=\pi_{|(1,\infty)},\ \beta:=\gamma^{+},\]
with $\gamma^{+}:=\max(\gamma,0)=\gamma-\gamma^{-}$ and where $\pi_{|A}$ denotes the measure $\pi$ restricted to $A\subset (0,\infty)$.
\smallskip

For any decomposition \eqref{eq:decompositionpsi} of $\Psi$, the functions $\Phi$ and $\Sigma$ \eqref{eq:sigmaphi} are both positive non-decreasing. We denote by $\rho$, the largest zero of $\Psi$, that is 
$$\rho:=\sup\{x\in [0,\infty): \Psi(x)\leq 0\}\in [0,\infty].$$ The following classification is standard.
\begin{itemize}
\item If $\Psi\geq 0$, then $\rho=0$, and $\Psi$ is of the form $\Sigma$ for some function as in \eqref{eq:sigmaphi}. This setting covers two cases, we say that $\Sigma$ (or $\Psi$) is subcritical if $\Sigma'(0+)>0$ and critical if $\Sigma'(0+)=0$. Notice that $\Sigma'(0+)=d\in [0,\infty)$, the drift parameter of $\Sigma$ in \eqref{eq:sigmaphi}.
\item If $\Psi'(0+)\in [-\infty,0)$, then $\rho\in (0,\infty]$ and we say that $\Psi$ is supercritical. Either $\Psi$ changes sign or not, in the latter case $\Psi\leq 0$ and $\rho=\infty$.
\item If $\Psi\leq 0$, then $\Psi$ is of the form $-\Phi$ for some function as in \eqref{eq:sigmaphi}. We say that the branching mechanism $\Psi$ is \textit{immortal}. Notice that $\Phi'(0+)\in (0,\infty]$.
\end{itemize}

The interaction mechanisms represented in Figure \ref{fig:drift} are branching mechanisms, the  classification above relates each case to a particular form of interaction: notice that  $\hat{\Sigma}$ (pure competition) is (sub)-critical, $\hat \Psi$ (mixed interaction) supercritical not immortal and $-\hat\Phi$ (pure cooperation) immortal.
\smallskip

A branching mechanism $\Psi$ is the Laplace exponent of a spectrally positive L\'evy process with Brownian part driven by the coefficient $\mathrm{a}$, drift $-\gamma$, killing $-\lambda$ and jump measure $\pi$. The writing $\Psi=\Sigma-\Phi$  decomposes the dynamics into a part gathering the large jumps, the killing and a non-negative drift, controlled by $\Phi$, and  a part for the small jumps, a non-positive drift and a diffusive component, governed by $\Sigma$. 
\smallskip

A function $\Phi$ of the form given in \eqref{eq:sigmaphi} is the Laplace exponent of an increasing L\'evy process (subordinator).  It is also called a Bernstein function, see e.g. Schilling et al. \cite{bernstein}. The map $x\mapsto \Phi(x)/x$ is decreasing towards the drift $\beta$, with a finite limit at $x=0$ if and only if $\Phi'(0+)<\infty$ (equivalently the subordinator has finite mean).
\smallskip

A function $\Sigma$ as given in \eqref{eq:sigmaphi} is the Laplace exponent of a L\'evy process with no negative jumps either oscillating (when $b=\Sigma'(0+)=0$) or drifting towards $-\infty$ (when $b>0$). Moreover, it is easily checked that the maps $x\mapsto \Sigma'(x)$ and $x\mapsto \frac{\Sigma(x)}{x}$ are of Bernstein's form, hence in particular  non-decreasing.

\begin{lem}\label{lem:equivpsi} Let $\Psi$ be of the form \eqref{branchingmechanism}.  For any decomposition $\Psi=\Sigma-\Phi$ as in \eqref{eq:decompositionpsi},
\begin{enumerate}[(1)]
\item if $\underset{u\rightarrow \infty}{\lim}\Psi(u)/u=\infty$, then \[\Psi(u)\underset{u\rightarrow \infty}{\sim} \Sigma(u),\]
\item if $\underset{u\rightarrow 0}{\lim}\Psi(u)/u=-\infty$ or $\Psi(0)=-\lambda<0$, then
\[\Psi(u)\underset{u\rightarrow 0}{\sim} -\Phi(u).\]
\end{enumerate}
\end{lem}
\begin{proof}
For (1), by assumption
\[\frac{\Psi(u)}{u}=\frac{\Sigma(u)}{u}-\frac{\Phi(u)}{u}\underset{u\rightarrow \infty}{\longrightarrow}\infty.\]
Since $u\mapsto \frac{\Phi(u)}{u}$ is bounded near $\infty$, $\frac{\Sigma(u)}{u}$ goes to $\infty$ as $u$ tends to $\infty$ and
\[\frac{\Psi(u)}{\Sigma(u)}=1-\frac{\Phi(u)}{\Sigma(u)}=1-\frac{\Phi(u)}{u}\frac{u}{\Sigma(u)}\underset{u\rightarrow \infty}{\longrightarrow} 1.\]
For showing (2), we write in a similar way,
\[-\frac{\Psi(u)}{\Phi(u)}=1-\frac{\Sigma(u)}{\Phi(u)}=1-\frac{\Sigma(u)}{u}\frac{u}{\Phi(u)}\underset{u\rightarrow 0}{\longrightarrow} 1,\]
where we use either that $\Phi(0)>0$ when $\Psi(0)<0$ and $\lim_{u\rightarrow 0}\Sigma(u)=\Sigma(0)=0$ or $\Phi'(0+)=\infty$ and $\lim_{x\rightarrow 0}\frac{\Sigma(x)}{x}=\Sigma'(0+)\geq 0$.
\end{proof}
\subsubsection{CB processes}\label{sec:CB}
For modern accounts on the theory of continuous-state branching ($\mathrm{CB}$) processes, we refer to the books of Kyprianou \cite[Chapter 12]{Kyprianoubook} and Li \cite[Chapter 9]{Li-book}. Most of the background given in this section can also be found in the fundamental article of Silverstein \cite{zbMATH03294035}. We provide here only the basic elements of the theory. 
\smallskip

A $\mathrm{CB}$ process, say $(X, (\mathbb{P}_x)_{x\in [0,\infty]})$, satisfies the Markov property and the branching property, that is  $\mathbb{P}_{x_1+x_2}=\mathbb{P}_{x_1}\star\mathbb{P}_{x_2}$ for all $x_1,x_2\in [0,\infty]$. Any such process, under mild regularity assumptions, is characterized by a branching mechanism $\Psi$, of the form \eqref{branchingmechanism}. It can be constructed for instance as the solution to a certain stochastic equation with jumps, see e.g. \cite[Chapter 11]{Li-book}, we postpone this discussion to the forthcoming Section \ref{sec:minimal} in which the broader class of $\mathrm{CBDIs}$ is addressed, and focus here on the main properties.
\smallskip

The $\mathrm{CB}(\Psi)$ has for extended\footnote{in the sense that it produces local martingale} infinitesimal generator, the following operator
\begin{equation}\label{eq:genCB} \mathscr{L}^{\Psi}f(x):=x\mathrm{L}^{\Psi}f(x),\ x\in \mathcal{D}_f,\end{equation}
where $\mathrm{L}^{\Psi}$ is the generator of a spectrally positive L\'evy process with Laplace exponent $\Psi$, that is, for any $f\in \mathrm{C}([0,\infty])\cap \mathrm{C}_b^2$, and all $x\in D_f$,
\begin{align}\label{Lpsi}&\mathrm{L}^{\Psi}f(x):=\mathrm{a}f''(x)+\gamma f'(x)+\lambda\big( f(\infty)-f(x)\big)+\int_0^{\infty}\!\!\left(f(x+u)-f(x)-uf'(x)\mathbbm{1}_{\{u<1\}}\right)\pi(\ddr u),
\end{align}
see e.g. Sato's book \cite{MR3185174} or Bertoin \cite[Page 24]{Bertoin96}. The following fundamental identity is readily checked 
\[\mathrm{L}^{\Psi}\ee^y(x)=\Psi(y)\ee^y(x), \ x\in [0,\infty),\ y\in (0,\infty).\]  One has moreover
\[\mathscr{L}^{\Psi}\mathrm{e}^y(x)=x\mathrm{L}^{\Psi}\ee^y(x)=-\Psi(y)(\ee_x)'(y), \ x\in [0,\infty),\ y\in (0,\infty).\]
The ordinary differential equation (o.d.e.) \begin{equation}\label{eq:ode}\frac{\ddr }{\ddr t}\mathrm{y}_t=-\Psi(\mathrm{y}_t),\ t\in [0,\infty), \ \mathrm{y}_0=y\in (0,\infty), 
\end{equation}
admits a unique $(0,\infty)$-valued solution,  $\big(\mathrm{y}_t(y)\big)_{t\geq 0}$. One has for all $t\geq 0$, all $x\in [0,\infty]$ and $y\in (0,\infty)$,
\begin{equation}\label{eq:dualsemigroupcb}\mathbb{E}_x[e^{-X_ty}]=e^{-x\mathrm{y}_t(y)},\ t\in[0,\infty).\end{equation}

Many properties of the branching process $X$ are encoded in the deterministic function $\mathrm{y}$ and the identity \eqref{eq:dualsemigroupcb}, which can be seen as a Laplace duality relationship, is the starting point of a deep study of $\mathrm{CB}s$. 

A first straightforward consequence is the continuity of $[0,\infty]\ni x\mapsto \mathbb{E}_x(e^{-X_ty})$. The latter, together with the density of $\{\mathrm{e}^{y}, y\in [0,\infty)\}$ in $\mathrm{C}([0,\infty])$ for the uniform norm, ensured by the Stone-Weierstrass theorem, entails the Feller property  of $X$.
\smallskip

Notice also that the identity \eqref{eq:dualsemigroupcb} together with the fact that $(\mathrm{y}_t)_{t\geq 0}$ is $(0,\infty)$-valued, forces both boundaries $0$ and $\infty$ to be absorbing, indeed, for all $y\in (0,\infty)$
\[\mathbb{E}_\infty[e^{-X_ty}]=e^{-\infty\cdot \mathrm{y}_t(y)}=0 \text{ and }\mathbb{E}_0[e^{-X_ty}]=e^{-0\cdot \mathrm{y}_t(y)}=1 \text{ for all } t\in [0,\infty).\]
In other words, in the absence  of interactions or external sources modeling for instance arrivals or departures of individuals, the population with no density-dependence can only be absorbed at $0$ or at~$\infty$.

We recall some of the most important properties of $\mathrm{CBs}$ and refer to \cite[Section 12.2]{Kyprianoubook}. The next theorem is mainly due to Grey \cite{MR0408016}.
\begin{theostar}[Longterm behaviors of $\mathrm{CBs}$] Let $X$ be a $\mathrm{CB}(\Psi)$ process. Then,
\begin{enumerate}
\item One has $\lim_{t\rightarrow \infty}X_t=0$ a.s. if and only if $\Psi$ is critical or subcritical, i.e. $\Psi=\Sigma$, in the notation of Section \ref{sec:LKfunction}. In this case, $\mathrm{y}_t\to 0$ as $t$ tends to $\infty$.
\item One has $\lim_{t\rightarrow \infty}X_t=\infty$ a.s. if and only if $\Psi$ is immortal, i.e. $\Psi=-\Phi$, in the notation of Section \ref{sec:LKfunction}. In this case, $\mathrm{y}_t\to \infty$ as $t$ tends to $\infty$.
\item The process has positive probability  to converge to $0$ and to diverge to $\infty$ if and only if $\Psi$ is supercritical not immortal, i.e. $\Psi$ changes sign. In the latter case, the probability that $X$ goes to $0$ under $\mathbb{P}_x$ is given by $e^{-x\rho}$ where $\rho$ is the largest zero of $\Psi$ (equivalently the fixed point of the o.d.e \eqref{eq:ode} and limit of $\mathrm{y}$).
\item The process is absorbed at $0$ in finite time with positive probability if and only if $\int_{x_1}^{\infty}\frac{\ddr u}{\Psi(u)}<\infty$ for some $x_1\in (0,\infty)$ (Grey's condition). The latter can only occur for non-immortal mechanism, and when the integral is finite, one has, 
\[\mathbb{P}_x(X_t=0)=e^{-x\mathrm{y}_t(\infty)}\in (0,1),\ \forall t\in [0,\infty), x\in [0,\infty).\]
\item The process is absorbed at $\infty$ in finite time with positive probability if and only if $\int_0^{x_1}\frac{\ddr u}{-\Psi(u)}<\infty$ for some $x_1\in (0,\infty)$ (Dynkin's condition\footnote{Harris \cite{zbMATH02001586} cites Dynkin for the study of explosion}). The latter can only occur for supercritical mechanism and when the integral is finite, one has
\[\mathbb{P}_x(X_t=\infty)=1-e^{-x\mathrm{y}_t(0)}\in (0,1),\ \forall t\in [0,\infty), x\in (0,\infty].\]
\end{enumerate}
\end{theostar}
The following lemma shows that accessibility of the boundaries $0$ (extinction) and $\infty$ (explosion) is determined by the $\Sigma$ and $\Phi$ components, respectively, of the branching mechanism $\Psi = \Sigma - \Phi$. This will also be useful later when analyzing the assumptions on the drift–interaction term $\hat{\Psi}$.
\begin{lem}\label{lem:equivpsisigmagrey} The following equivalences hold true:
\[\exists\, x_1\in (0,\infty); \ \int_{x_1}^{\infty}\frac{\ddr u}{\Psi(u)}<\infty \Longleftrightarrow \int_1^{\infty}\frac{\ddr u}{\Sigma(u)}<\infty. \]
and \[\exists\, x_1\in (0,\infty); \ \int_0^{x_1}\frac{\ddr u}{-\Psi(u)}<\infty \Longleftrightarrow \int_0^1\frac{\ddr u}{\Phi(u)}<\infty.\]
\end{lem}
\begin{proof}
Plainly, a necessary condition for having $\int_{x_1}^{\infty}\frac{\ddr u}{\Psi(u)}<\infty$ is that $\frac{\Psi(u)}{u}\underset{u\rightarrow \infty}{\longrightarrow} \infty$. The first equivalence is then just a straightforward consequence of Lemma \ref{lem:equivpsi}-(1). We deduce similarly  the second one by using Lemma \ref{lem:equivpsi}-(2).\end{proof}

The next lemma provides two fundamental monotonic eigenfunctions of $\mathrm{CB}$'s generator. It will play some role later when discussing necessary conditions for having non-absorption of $\mathrm{CBDIs}$ at their boundaries.
\begin{lemstar}[Eigenfunctions for $\mathscr{L}^{\Psi}$]\label{lem:eigenfunctions} Let $\Psi$ be a  branching mechanism. Recall $\rho\in [0,\infty]$ its largest zero.
\begin{enumerate}
\item Assume $\Psi$ supercritical: $\Psi'(0+)\in [-\infty,0)$ and $\rho\in (0,\infty]$.  Let $u_0\in (0, \rho)$ be fixed. Define for all $\theta\in \big(0,-\Psi'(0+)\big)$,
\begin{equation}\label{harmonicfunctioncb1}f(x):=\int_0^{\rho} (1-e^{-xu})\frac{\theta}{-\Psi(u)}e^{-\int_{u_0}^{u}\frac{\theta}{\Psi(v)}\ddr v}\ddr u,\quad x\in (0,\infty).
\end{equation}
This is a well-defined increasing function, in $\mathrm{C}^2((0,\infty))$, which satisfies $x\mathrm{L}^{\Psi}f(x)=\theta f(x)$ for all $x\in (0,\infty)$.
Moreover, 
\[f(x)\underset{x\rightarrow \infty}{\longrightarrow}\infty \Longleftrightarrow \int_0^{x_0}\frac{\ddr u}{-\Psi(u)}=\infty \text{ for some } x_0\in (0,\rho).\]
\item Assume $\Psi$ non-immortal (not necessarily supercritical): $\rho\in [0,\infty)$. Let $u_0\in (\rho,\infty)$ be fixed. Define for all $\theta\in (0,\infty)$, \[f(x):=\int_{\rho}^{\infty}e^{-xu}\frac{1}{\Psi(u)}e^{\int_{u_0}^u\frac{\theta}{\Psi(v)}\ddr v}\ddr u,\ x\in (0,\infty).\]
This is a well-defined decreasing function, in $\mathrm{C}^2((0,\infty))$, which satisfies $x\mathrm{L}^{\Psi}f(x)=\theta f(x)$ for all $x\in (0,\infty)$. Moreover, \[f(x)\underset{x\rightarrow 0}{\longrightarrow}\infty \Longleftrightarrow
\int_{x_0}^{\infty}\frac{\ddr u}{\Psi(u)}=\infty \text{ for some } x_0\in (\rho,\infty).\]
\end{enumerate}
\end{lemstar}
\begin{proof}
For Statement 1, this is \cite[Lemma 5.1]{foucartvidmar}. For Statement 2,  we apply \cite[Lemma 5]{Duhalde}, with, in the notation therein, $\mu=0, q(0)=\rho_c$ and $\Phi\equiv 0$. For the limit at $0$ of $f$, one has, by the monotone convergence theorem,  $$\underset{x\rightarrow 0}{\lim} f(x)=\int_{\rho}^{\infty}\frac{1}{\Psi(u)}e^{\int_{u_0}^u\frac{\theta}{\Psi(v)}\ddr v}\ddr u=\left[e^{\int_{u_0}^{x}\frac{\ddr v}{\Psi(v)}}\right]_{x=\rho}^{x=\infty}=\infty \text{ iff } \int_{u_0}^{\infty}\frac{\ddr u}{\Psi(u)}=\infty.$$  
\end{proof}
\subsection{Scale functions and potential measures}\label{sec:potentialelements}
We gather in this section elementary facts on the scale function of a spectrally positive L\'evy process and on the potential measure of a subordinator. We refer the reader to Bertoin \cite{Bertoin96} and Kyprianou \cite{Kyprianoubook}. 
\subsubsection{Scale functions of spectrally positive L\'evy processes}\label{sec:scalefunction}
Let $\Sigma$ be a (sub)-critical branching mechanism. Recall that this is the Laplace exponent of a spectrally positive L\'evy process which is not drifting towards $+\infty$. We focus on the case with infinite variation, namely  $\lim_{x\rightarrow \infty}\uparrow \Sigma(x)/x=\infty$.

There exists a positive strictly increasing continuous function $W$ defined on $[0,\infty)$, called scale function, such that 
\begin{equation}\label{eq:scalefunction}\frac{1}{\Sigma(x)}=\int_0^{\infty}e^{-xz}W(z)\ddr z, \quad x\in (0,\infty).
\end{equation}
Furthermore,  for any $p>0$, $e^{-px}W(x)\underset{x\rightarrow \infty}{\longrightarrow} 0$,
\begin{equation}
  W(0)=0, \ W(x)\underset{x\rightarrow \infty}{\longrightarrow} \frac{1}{\Sigma'(0+)}\in (0,\infty] \text{ and } W(x)\asymp \frac{1}{x\Sigma(1/x)}.   
\end{equation}
We refer for these facts to \cite[Chapter VII]{Bertoin96}. One easily checks, with the help of Fubini-Tonelli's theorem and the identity \eqref{eq:equivgreyscale}, the following equivalence
\begin{equation}\label{eq:equivgreyscale}\int_1^{\infty}\frac{\ddr x}{\Sigma(x)}<\infty \Longleftrightarrow \int_0^1\frac{W(z)}{z}\ddr z<\infty.
\end{equation}
The scale function $W$ occurs in particular when studying the exit problem of $\mathrm{CB}$ processes (and their parent L\'evy processes). Let $X$ be a $\mathrm{CB}(\Sigma)$ process and call $\sigma_0^{-}$ and $\sigma_{x+y}^{+}$ its first hitting time of $0$ and its first passage time above $x+y$. One has, see e.g. Bingham \cite[Proposition 3.1]{MR0410961} and \cite[Theorem 8, page 194]{Bertoin96}, for all $x,y\in (0,\infty)$, if \eqref{eq:equivgreyscale} holds (so that $\mathbb{P}_x(\sigma_0^{-}<\infty)>0$),
\[\mathbb{P}_x\big(\sigma_0^{-}<\sigma_{x+y}^{+}\big)=\frac{W(y)}{W(x+y)}.\]

In the critical stable case, namely $\Sigma(x):=Cx^{1+\beta}$ with $\beta\in (0,1]$ and $C>0$, one has $W(z)=\frac{z^{\beta}}{\Gamma(\beta+1)C},\  z\in [0,\infty)$, where  $\Gamma$ is the Gamma function. Apart from this setting, scale functions are not explicit in general or appeal special functions, see e.g. \cite[Chapters 8 and 9]{Kyprianoubook}. We can however find their asymptotics in the setting of regular variation, see Section \ref{sec:regularvarying}.

\subsubsection{Potential measure of subordinators}\label{sec:potentialmeasure}
Let $(S, \mathbb{P})$ be a subordinator started from $0$. Call its Laplace exponent $\Phi$ and define the positive measure $U$ as follows
\[U(A):=\int_{0}^{\infty}\mathbb{P}(S_t\in A)\ddr t, \ A\in \mathrm{B}_{[0,\infty)}.\] 
One has, see e.g. \cite[Chapter III]{Bertoin96}, 
\begin{equation}\label{eq:potentialmeasure}\frac{1}{\Phi(x)}=\int_0^{\infty}e^{-xz}U(\ddr z), \quad x\in [0,\infty).
\end{equation}
Similarly as for the scale function, the following equivalence is easily verified
\begin{equation}\label{eq:equivdynkinpotential}\int_0^{1}\frac{\ddr x}{\Phi(x)}<\infty \Longleftrightarrow \int_1^{\infty}\frac{U(\ddr z)}{z}\ddr z<\infty.
\end{equation}
The potential measure $U$, also called renewal measure, is related to the first passage times of $S$. Denoting by $T^{+}_x$, the first passage time above $x$ of $S$, one has \[U([0,x])=\mathbb{E}[T_x^{+}], \ x\in (0,\infty).\]
When the subordinator $S$ admits a drift, i.e. $\beta:=\lim_{y\rightarrow \infty}\Phi(y)/y>0$, the measure $U$ admits a density, that is $U(\ddr z)=u(z)\ddr z$ for some 
 positive continuous function $u$ defined on $(0,\infty)$ such that $u(0+)=1/\beta$, see \cite{Kyprianoubook}.  
 
 Many other examples of subordinators with a potential density can be found in the literature, we refer to Song and Vondra{\v{c}}ek~\cite{zbMATH05841214}. The stable case is one example, $\Phi:[0,\infty)\ni q\mapsto cq^{\alpha}$, for $\alpha\in (0,1)$, and one has $U(\ddr z)=\frac{1}{c}\frac{z^{\alpha-1}}{\Gamma(\alpha)}\ddr z$.

\subsection{Minimal $\mathrm{CBDIs}$:  martingale problem and stochastic equation}\label{sec:minimal}
Let $\hat{\Psi}$ be another branching mechanism, that is to say a function of the form \eqref{branchingmechanism}. Denote its quadruplet by $(\hat{\pi},\hat{\mathrm{a}},\hat{\gamma},\hat{\lambda})$. Recall  $\mathrm{L}^{\Psi}$ in \eqref{Lpsi} and define the operator $\mathcal{X}$ as follows
\begin{equation}\label{genX}
\begin{split}
\mathcal{X}f(x):=x\mathrm{L}^{\Psi}f(x)-\hat{\Psi}(x)f'(x)&, \ x\in \mathcal{D}_f \\
&\text{ with } f\in \mathcal{D}_{\mathcal{X}}:=\left\{f\in \mathrm{C}^{2}: \mathcal{X}f \text{ is well defined}\right\}.
\end{split}
\end{equation}
We introduce below the notion of $\mathrm{CBDIs}$, along with the concept of extensions that we will use.
\begin{defi}\label{defCBDI}\
\begin{enumerate}[(i)]
    \item \label{(i)CBDIgen} We call $\mathrm{CBDI}$ any $[0,\infty]$-valued càdlàg Markov process $\big(X, (\mathbb{P}_x)_{x\in [0,\infty]}\big)$, solution to the following martingale problem $\mathrm{MP}\big(\mathcal{X},\mathrm{C}^2_c((0,\infty))\big)$:
\[\forall x\in (0,\infty), \ f\in \mathrm{C}^2_c((0,\infty)), \quad \left(f(X_t)-\int_0^{t}\mathcal{X}f(X_s)\ddr s,\ t\geq 0 \right) \text{ is a } \mathbb{P}_x\text{ - martingale}.\]
\item \label{(ii)minCBDI} We call minimal and denote by $X^{\mathrm{m}}$, any $\mathrm{CBDI}$ process  whose boundaries $\{0,\infty\}$ are absorbing, namely: 
\[\forall x\in [0,\infty], \ \mathbb{P}_x-\text{a.s.}\  X^{\mathrm{m}}_t=\begin{cases} 0 &\text{ if } t\geq \sigma_0^{-} \text{ and }  \sigma_0^{-}<\sigma_\infty^{+}, \\
\infty  &\text{ if } t\geq \sigma_\infty^{+} \text{ and }  \sigma_\infty^{+}<\sigma_0^{-}.
\end{cases}
\]
\item \label{(iii)extCBDI} For $\Delta\in \{0,\infty\}$, we say that a process $\big(X^{\mathrm{e}\Delta},(\mathbb{P}_x)_{x\in [0,\infty]}\big)$ is a $\mathrm{CBDI}$ extended at $\Delta$, if it is a $[0,\infty]$-valued  Markov process, which once stopped at the boundary $\Delta$, has the same law as the minimal process: namely, if one sets  $\zeta^{\Delta}:=\inf\big\{t\geq 0: X_{t-} \text{ or } X_t =\Delta\big\}$, then one has
\[\forall x\in [0,\infty]\setminus \{\Delta\}, \ \  (X^{\mathrm{e}\Delta}_{t\wedge \zeta^{\Delta}},t\geq 0)\overset{\mathrm{law}}{=}(X^{\mathrm{m}}_t,t\geq 0)\ \text{ under } \mathbb{P}_x.\]
\begin{itemize}
\item[-]The process $X^{\mathrm{e}\Delta}$ is said furthermore to be a continuous extension if its boundary $\Delta$ is continuous in the sense of Definition~\ref{def:boundary}-\ref{def:continuous}. 
\item[-] It is a Fellerian extension of $X^{\mathrm{m}}$ if  its semigroup $(P^{\mathrm{e}\Delta}_t)_{t\geq 0}$ satisfies for any $f\in \mathsf{C}([0,\infty])$,
    $P^{\mathrm{e}\Delta}_tf(x) \underset{t\rightarrow 0+}{\rightarrow} f(x),\ x\in [0,\infty]$ 	and $P^{\mathrm{e}\Delta}_tf \in \mathsf{C}([0,\infty])$.
\end{itemize}
\end{enumerate}
\end{defi}
The martingale problem on $\mathrm{C}^{2}_c\big((0,\infty)\big)$ in \ref{(i)CBDIgen} only identifies the interior dynamics, not the boundary condition. With a slight abuse of terminology, we nonetheless call $\mathcal{X}$ the generator of the $\mathrm{CBDI}$. Observe from \ref{(iii)extCBDI} that by definition an extension at $\infty$, $X^{\mathrm{e}\infty}$, (resp. $X^{\mathrm{e}0}$ at $0$) has the other boundary $0$ (resp. $\infty$) absorbing (we also say that the extension is minimal at $0$ (resp. at $\infty$)). Note also that, in the definition of an extension at $\Delta$, we do not prescribe its behavior at~$\Delta$: it may, in particular, be \textit{absorbing}, as in the minimal setting. 
\medskip

Nowadays, continuous-state branching (CB) processes and their generalizations are often introduced via stochastic differential equations with jumps. 
This approach will also prove useful in the present article.
\smallskip

Let $(\Omega,\mathcal{F},(\mathcal{F})_{t\geq 0},\mathbb{P})$ 
be a filtered probability space satisfying the usual hypotheses.
Let $B$ be an $(\mathcal{F}_t)_{t\geq 0}$-Brownian motion, $\mathcal{N}(\ddr s,\ddr r, \ddr u)$  be an $(\mathcal{F}_t)_{t\geq 0}$-Poisson random measure (PRM) on $(0,\infty)\times (0,\infty)\times (0,\infty]$, with intensity measure $\ddr s \ddr r \pi(\ddr u)$. Let $\bar{\mathcal{N}}(\ddr s,\ddr r, \ddr u):=\mathcal{N}(\ddr s,\ddr r, \ddr u)-\ddr s \ddr r (\pi(\ddr u)+\lambda \delta_{\infty})$ be the compensated Poisson random measure.
\medskip

Equivalently\footnote{we refer to Jacod and Shiryaev \cite[Theorem 2.26, page 157]{zbMATH01834045}, see also Kurtz \cite{zbMATH05919793}}  to Definition~\ref{defCBDI}-\ref{(i)CBDIgen}, a $\mathrm{CBDI}(\Psi,\hat{\Psi})$ is a $[0,\infty]$-valued càdlàg Markov process weak solution to the following stochastic equation: 
\begin{align}
X_t=x+\int_0^t\!\! \sqrt{2aX_s}\ddr B_s+\gamma \int_{0}^t\!\! X_s\ddr s &+\int_0^t\!\int_{0}^{X_{s-}}\!\!\int_{(0,1]}\!\!u\bar{\mathcal{N}}(\ddr s,\ddr r, \ddr u)\nonumber \\
&+\int_0^t\!\int_{0}^{X_{s-}}\!\!\int_{(1,\infty]}\!\!u\mathcal{N}(\ddr s,\ddr r, \ddr u)-\int_0^t\hat{\Psi}(X_s)\ddr s,\quad t\in [0,\infty)
\label{SDECBDI}
\end{align}
with $B$ a Brownian motion, $\mathcal{N}(\ddr s,\ddr r, \ddr u)$  an independent Poisson random measure (PRM) with intensity $\ddr s \ddr r \pi(\ddr u)$, $\bar{\mathcal{N}}(\ddr s,\ddr r, \ddr u):=\mathcal{N}(\ddr s,\ddr r, \ddr u)-\ddr s \ddr r (\pi(\ddr u)+\lambda \delta_{\infty})$ its compensated version.
\smallskip

 In general \textit{several} solutions to the martingale problem in Definition~\ref{defCBDI}-(i) (and to the stochastic equation \eqref{SDECBDI}), might exist, as we allow for non-Lipschitz drift $\hat{\Psi}$. According to Definition~\ref{defCBDI}-\ref{(i)CBDIgen}, they are all called CBDIs.\\

We see now that minimal CBDIs exist and are unique in law. We construct them from~\eqref{SDECBDI}.

\begin{theo}\label{thm:minX} Let $\Psi, \hat{\Psi}$ be two L\'evy-Khintchine functions, i.e. of the form \eqref{branchingmechanism}. There exists a 	unique minimal $\mathrm{CBDI}(\Psi,\hat{\Psi})$ process. The latter, denoted by  $X^{\mathrm{m}}$, can be constructed as the unique strong solution to \eqref{SDECBDI} absorbed at the boundaries. \end{theo}
\begin{rem}\label{rem:nolossofgen} For any given mechanism $\Psi$ and $c\in \mathbb{R}$, define the mechanism \[\Psi_{-c}:[0,\infty)\ni x\mapsto \Psi(x)-cx.\] We see plainly from the stochastic equation \eqref{SDECBDI}, that the minimal $\mathrm{CBDI}(\Psi,\hat{\Psi})$-process has the same law as the minimal $\mathrm{CBDI}(\Psi_{-c},\hat{\Psi}_{c})$ for any $c\in \mathbb{R}$ (the drifts $cX_t$ and $-cX_t$ in the branching and the interaction parts cancel out). 
\end{rem}

\begin{proof}[Proof of Theorem~\ref{thm:minX}]
Recall that $\hat{\Psi}$ is continuous on $[0,\infty)$, locally Lipschitz on $(0,\infty)$ and $\hat{\Psi}(0)\leq 0$.  Well-posedness of $\mathrm{MP}\big(\mathcal{X},\mathrm{C}^2_c((0,\infty))\big)$, with boundaries being absorbing, that is to say existence of a unique minimal $\mathrm{CBDI}(\Psi,\hat{\Psi})$-process, can be established by applying Stroock \cite[Theorem~4.3]{zbMATH03457981} (with a localization argument). The operator $\mathcal{X}$ indeed satisfies the local boundedness and continuity assumptions required in Stroock’s theorem, see Eq. (4.1) therein. 

As we will use, later on, the minimal CBDI constructed as a \textit{strong} solution to \eqref{SDECBDI}, we provide arguments based directly on this stochastic equation. The pathwise unique existence of a solution on the random interval $[0,\zeta)$, with $\zeta:=\inf\{t\geq 0: X_{t-}\text{ or } X_t \notin (0,\infty)\}\in [0,\infty]$ is obtained by applying Dawson-Li's result \cite[Theorem 2.5]{DawsonLi} after localization, see  \cite[Proposition 1]{zbMATH06836271} and \cite[Theorem 3.1]{zbMATH07120715} for details on the latter. We extend this solution after $\zeta$ by absorption. This provides the minimal $\mathrm{CBDI}$ strong solution to \eqref{SDECBDI}.

One can also deduce again uniqueness of the minimal $\mathrm{CBDI}$, in the sense of Definition~\ref{defCBDI}-\ref{(ii)minCBDI}, as follows. Pathwise uniqueness entails that there is a unique weak solution to \eqref{SDECBDI} with both boundaries absorbing, see e.g. Barczy et al. \cite[Theorem 1]{zbMATH06573013}. By Itô's formula, this process provides a solution to the martingale problem $\mathrm{MP}\big(\mathcal{X},\mathrm{C}^2_c((0,\infty))\big)$, see e.g. \cite[Section~6.1]{rebotier} for details on the calculations. Conversely, any solution to the martingale problem  is a weak solution to \eqref{SDECBDI}. We refer e.g. to \cite[Theorem 2.3]{zbMATH05919793} for formulations adapted to the present setting. There is therefore a unique solution \textit{absorbed at the boundaries} to $\mathrm{MP}\big(\mathcal{X},\mathrm{C}^2_c((0,\infty))\big)$.
\end{proof}
We gather in the following proposition, two fundamental comparison properties, along the initial values and with respect to the drift function, fulfilled by the minimal $\mathrm{CBDI}$ process. They will be used extensively through the article when defining extended $\mathrm{CBDIs}$. 
\begin{prop}\label{prop:comparison}
\
\begin{enumerate}
\item For all $y\geq x\geq 0$, denoting by $X^{\mathrm{m}}(x)$ and $X^{\mathrm{m}}(y)$ the two minimal $\mathrm{CBDI}$ processes, strong solutions of \eqref{SDECBDI} with initial value $x$ and $y$ respectively, one has $$\mathbb{P}\big(X^{\mathrm{m}}_t(y)\geq X^{\mathrm{m}}_t(x) \text{ for all } t\in [0,\infty) \big)=1.$$
In particular, $\sigma_\infty^+(y)\leq \sigma_\infty^+(x)$ and $\sigma_0^-(y)\geq \sigma_0^-(x)$ a.s..
\item If $\hat{\Psi}_1\leq \hat{\Psi}_2$ and $X^{\mathrm{m}1}$ and $X^{\mathrm{m}2}$ are minimal $\mathrm{CBDI}(\Psi,\hat{\Psi}_i)$ with $i\in \{1,2\}$, started from the same initial value, then 
$$\mathbb{P}\big(X^{\mathrm{m}1}_t\geq X^{\mathrm{m}2}_t, \text{ for all } t\in [0,\infty)\big)=1.$$
In particular, $\sigma^{1,+}_\infty\leq \sigma^{2,+}_\infty$ and $\sigma^{1,-}_0\geq \sigma^{2,-}_0$ a.s..
\end{enumerate}
\end{prop}
\begin{proof}
Let $m\in (0,\infty)$ and
$\zeta_m(z):=\sigma_{1/m}^{-}(z)\wedge\sigma_m^+(z)$ for all $z\in (0,\infty)$. Similarly as in the proof of Theorem~\ref{thm:minX}, by localizing and then applying Dawson and Li \cite[Theorem 2.2]{DawsonLi}, see \cite{rebotier} and the references therein for details, the almost sure comparison property is satisfied until $\zeta_m(x)\wedge \zeta_m(y)$. By passing to the limit as $m$ goes to $\infty$, we see then that it holds until $\zeta(x)\wedge \zeta(y)$. By the absence of negative jumps and the càdlàg regularity of the paths, we easily check that $\sigma_0^{-}(y)\geq \sigma_0^{-}(x)$ almost surely. Similarly $\sigma_\infty^{+}(x)\geq \sigma_\infty^{+}(y)$ a.s., hence the comparison is true on $\big[0,\sigma_0^{-}(x)\wedge \sigma_\infty^{+}(y)\big)$ and the boundaries being absorbing, the order $X_t^{\mathrm{m}}(x)\leq X_t^{\mathrm{m}}(y)$ holds for all $t\geq 0$ almost surely. The argument is similar for the second point.
\end{proof}
The following proposition establishes that cooperation cannot cause explosion in a non-explosive CB process, nor can competition lead to extinction in a CB process which cannot hit $0$ in finite time. 
\begin{prop}\label{prop:suffcondinaccessibility} Let $X^{\mathrm{m}}$ be a minimal $\mathrm{CBDI}(\Psi,\hat{\Psi})$ with $\Psi=\Sigma-\Phi$.
\begin{enumerate}[(1)]
    \item If
 \[\hypertarget{H1}{\mathbb{H}_1}:\ \int_0^1\frac{\ddr u}{\Phi(u)}=\infty\]
 then $X^{\mathrm{m}}$ does not explode. 
\item If  
\[\hypertarget{H2}{\mathbb{H}_2}: \  \int_{1}^{\infty}\frac{\ddr u}{\Sigma(u)}=\infty \]
then $X^{\mathrm{m}}$ does not get extinct. 
\end{enumerate}
\end{prop}
\begin{proof}
For establishing (1), we find an increasing function $f$, tending to $+\infty$ such that $\mathcal{X}f\leq \theta f$ on a neighbourhood of $\infty$ for some constant $\theta\in (0,\infty)$. We focus on the supercritical case, $\Psi'(0+)\in [-\infty,0)$, the (sub)-critical case can be deduced by a simple comparison argument. Let $\rho>0$ be the largest zero of $\Psi$. For any $c\in (0,\infty)$, define the auxiliary supercritical branching mechanism $\Psi_{-c}(x):=\Psi(x)-cx, \ x\in [0,\infty)$. Let $\rho_c>0$ be its largest zero, $u_0\in (0,\rho_c)$.  By assumption $\Psi=\Sigma-\Phi$ and one has, for all $x_0\in (0,\rho)$,
$\int_0^{x_0}\frac{\ddr u}{-\Psi(u)}=\infty$ and since $u/\Psi(u)\underset{u\rightarrow 0}{\longrightarrow} -1/\Psi'(0+)\in (-\infty,0]$, we have that $\int_0^{x_0}\frac{\ddr u}{-\Psi_{-c}(u)}=\infty$ for $x_0\in (0,\rho_c)$. Define for all $\theta\in \big(0,c-\Psi'(0+)\big)$,
\begin{equation}\label{harmonicfunctioncb2}f(x):=\int_0^{\rho_{c}} (1-e^{-xu})\frac{\theta}{-\Psi_{-c}(u)}e^{-\int_{u_0}^{u}\frac{\theta}{\Psi_{-c}(v)}\ddr v}\ddr u,\quad x\in (0,\infty).
\end{equation}
By Lemma \ref{lem:eigenfunctions}, it satisfies $x\mathrm{L}^{\Psi_c}f(x)=\theta f(x)$ for all $x>0$. Using the decomposition $\hat{\Psi}=\hat{\Sigma}-\hat{\Phi}$ with $\hat{\Sigma}\geq 0$ and $\hat{\Phi}$ the Laplace exponent of a subordinator, we find that for all $x\ge x_0$, \[-\hat{\Psi}(x)/x\leq \hat{\Phi}(x)/x \leq \hat{\Phi}(x_0)/x_0.\]
Pick $c:=\hat{\Phi}(x_0)/x_0$, for all $x\in [x_0,\infty)$
 \begin{align*}
 \mathcal{X} f(x)=x\mathrm{L}^{\Psi}f(x)-\hat{\Psi}(x)f'(x)&\leq x\mathrm{L}^{\Psi}f(x)+\frac{\hat{\Phi}(x_0)}{x_0}xf'(x).\\
 &=x\mathrm{L}^{\Psi_{-c}}f(x)=\theta f(x).
 \end{align*}
Under the assumption $\int_0^{x_0}\frac{\ddr x}{-\Psi(x)}=\infty$, $f$ increases towards $\infty$ and a standard result ensures that the process  does not explode almost surely, see e.g. \cite[Theorem~A]{rebotier} and the  arguments we will give below.

The proof of statement (2) follows similar ideas. Under the assumption $\int^{\infty}_1\frac{\ddr u}{\Sigma(u)}=\infty$, one has $\int_{x_0}^{\infty}\frac{\ddr u}{\Psi(u)}=\infty$ for all $x_0>\rho$. This ensures that $\Psi(u)/u\underset{u\rightarrow \infty}{\longrightarrow} \infty$ and one also has
$\int_{x_0}^{\infty}\frac{\ddr u}{\Psi_{c}(u)}=\infty$ for the auxiliary mechanism $\Psi_{c}$. Let $\theta>0$ and define the function, \[f(x):=\int_{\rho_c}^{\infty}e^{-xu}\frac{1}{\Psi_{c}(u)}e^{\int_{u_0}^u\frac{\theta}{\Psi_c(v)}\ddr v}\ddr u, \quad x\in (0,\infty).\]
By Lemma \ref{lem:eigenfunctions}, one has $x\mathrm{L}^{\Psi_c}f(x)=\theta f(x)$ for all $x\in (0,\infty)$. Moreover, $f$ is decreasing and $f(x)\rightarrow \infty$ as $x$ goes to $0$. Since $\hat{\Sigma}$ is a (sub)critical mechanism, one has for all $x\in (0,\infty)$, $\hat{\Sigma}(x)/x \geq \hat{\Sigma}'(0)$ and
 \begin{align*}
 \mathcal{X} f(x)=x\mathrm{L}^{\Psi}f(x)-\hat{\Psi}(x)f'(x)&=x\mathrm{L}^{\Psi}f(x)-\hat{\Sigma}(x)f'(x)+\hat{\Phi}(x)f'(x)\\
 &\leq x\mathrm{L}^{\Psi}f(x)-\hat{\Sigma}'(0)xf'(x)\\
 &=x\mathrm{L}^{\Psi_c}f(x)=\theta f(x),
 \end{align*}
 where we used in the inequality that $f'\leq 0$ and we pick $c:=\hat{\Sigma}'(0)$. We have thus built a decreasing $\mathrm{C}^2((0,\infty))$-function with limit $\infty$ at $0$ such that $\mathcal{X}f(x)\leq \theta f(x)$ for $x\in (0,\infty)$.  
 
By Itô's lemma, the process
$\left(f(X_t)e^{-\int_{0}^{t}\frac{\mathcal{X}f(X_s)}{f(X_s)}\ddr s}, t\geq 0\right)$ is a local martingale. Let $x_0\in (0,\infty)$, by stopping the latter at $\sigma_{x_0}^{-}$, we see that, for all $x\in (x_0,\infty)$,
\begin{align*}
\mathbb{E}_x\left[f(X_{t\wedge \sigma_{x_0}^{-}})e^{-\int_{0}^{t\wedge \sigma_{x_0}^{-}}\frac{\mathcal{X}f(X_s)}{f(X_s)}\ddr s} \right]&=f(x)\geq  f(x_0)\mathbb{E}_x[e^{-\theta t\wedge \sigma_{x_0}^{-}}].
\end{align*}
Thus, $\mathbb{E}_x[e^{-\theta \sigma_{x_0}^{-}}]\leq \frac{f(x)}{f(x_0)}$ and by letting $x_0$ go to $0$, we get since $f(x_0)\rightarrow \infty$, $\mathbb{E}_x[e^{-\theta \sigma_{0}^{-}}]=0$,
so that $\sigma_0^{-}=\infty$ a.s..
\end{proof}

\subsection{Generators duality and the minimal dual process}\label{sec:dualitygenlevel}
We explain in this section the Laplace duality relationship  \eqref{eq:gendualityintro} mentioned in the introduction. Martingales that play a key role in establishing duality at the level of semigroups are also introduced.\\

Let $\Psi$ and $\hat{\Psi}$ be of the L\'evy-Khintchine form \eqref{branchingmechanism}. Recall $\mathrm{L}^{\hat{\Psi}}$, \eqref{Lpsi}. Define the following operator  \[\mathcal{Y}f(y):=y\mathrm{L}^{\hat{\Psi}}f(y)-\Psi(y)f'(y),\ y\in [0,\infty).\] 
By definition~\ref{defCBDI}-\eqref{(i)CBDIgen}, the operator $\mathcal{Y}$ is the generator of a $\mathrm{CBDI}(\hat{\Psi},\Psi)$. 
\smallskip

Recall that for any $x,y \in  (0, \infty)$,  we set 
\[\ee^y(x)=e^{-xy}=\ee_x(y).\] 
Then, by combining the following identities, Section \ref{sec:CB},
\begin{center}
$\mathrm{L}^{\Psi}\ee^y(x)=\Psi(y)\ee^y(x)$ and $-\hat{\Psi}(x)(\ee^y)'(x)=y\hat\Psi(x)\ee_{x}(y)$,
\end{center}
 we get for all $(x,y)\in [0,\infty]\times (0,\infty)\cup (0,\infty)\times [0,\infty]$,
\begin{align}\label{eq:gendual}
\mathcal{X}\ee^y(x)&=\big(x\Psi(y)+\hat{\Psi}(x)y\big)\ee^{y}(x)=\big(\hat{\Psi}(x)y+x\Psi(y)\big)\ee_{x}(y)=\mathcal{Y}\ee_x(y). 
\end{align}
\medskip

\noindent We now introduce $Y^{\mathrm{m}}$, the minimal $\mathrm{CBDI}(\hat\Psi,\Psi)$ process. Since we shall mainly work at the level of its semigroup, $Y^{\mathrm{m}}$ may be defined on a probability space distinct from that of $X^{\mathrm{m}}$ in \eqref{SDECBDI}, and may in particular be taken independent.  
\medskip

By Theorem~\ref{thm:minX}, there is a unique (strong) solution $Y^{\mathrm{m}}$,  stopped at time \[\hat{\zeta}=\inf\{t\geq 0: Y^{\mathrm{m}}_{t-}\text{ or } Y^{\mathrm{m}}_t \notin (0,\infty)\},\]  to the stochastic equation
\begin{align}
Y_t=y+\int_0^t\!\! \sqrt{2\hat{a}Y_s}\ddr \hat{B}_s+\hat{\gamma} \int_{0}^t\!\! Y_s\ddr s &+\int_0^t\!\int_{0}^{Y_{s-}}\!\!\int_{(0,1]}\!\!u\bar{\hat{\mathcal{N}}}(\ddr s,\ddr r, \ddr u)\nonumber \\
&+\int_0^t\!\int_{0}^{Y_{s-}}\!\!\int_{(1,\infty]}\!\!u\hat{\mathcal{N}}(\ddr s,\ddr r, \ddr u)-\int_0^t\Psi(Y_s)\ddr s,\quad t\in [0,\infty)
\label{SDECBDIY}
\end{align}
with $y\in (0,\infty)$, $\hat{B}$ a Brownian motion, $\hat{\mathcal{N}}(\ddr s,\ddr r, \ddr u)$  a PRM with intensity $\ddr s \ddr r \hat{\pi}(\ddr u)$, and $\bar{\hat{\mathcal{N}}}(\ddr s,\ddr r, \ddr u):=\mathcal{N}(\ddr s,\ddr r, \ddr u)-\ddr s \ddr r \hat{\pi}(\ddr u)$ its compensated version.
\medskip

For any $y\in [0,\infty]$, we denote by $\mathbb{P}^y$ the law of the process $Y^{\mathrm{m}}$ started from $y$.  We insist on the fact that the boundaries are taken as absorbing for $Y^{\mathrm{m}}$ (even in the setting $-\Psi(0)>0$), so that  $$\mathbb{P}^{\Delta}(Y^{\mathrm{m}}_t=\Delta, \ \forall t\geq 0)=1, \ \forall \Delta\in \{0,\infty\}.$$ 
\begin{lem}\label{lem;MPonexpminimalprocess}  Let $X^{\mathrm{m}}$ and $Y^{\mathrm{m}}$ be minimal $\mathrm{CBDI}s$ with mechanisms $(\Psi,\hat{\Psi})$ and $(\hat{\Psi},\Psi)$.
\begin{enumerate}[(i)]
\item \label{(i)expmartingale} Let $x,y\in (0,\infty)$.	Under $\mathbb{P}^y$ and $\mathbb{P}_x$, respectively, the processes
\begin{equation}\label{eq:martingaleexpo}
\begin{aligned}
M^{x,Y^{\mathrm{m}}}
&:= \Big(\ee_x(Y_t^{\mathrm{m}})
      - \int_0^{t}\mathcal{Y}\ee_x(Y^{\mathrm{m}}_s)\,ds\Big)_{t\ge 0}, \\
M^{y,X^{\mathrm{m}}}
&:= \Big(\ee^y(X_t^{\mathrm{m}})
      - \int_0^{t}\mathcal{X}\ee^y(X^{\mathrm{m}}_s)\,ds\Big)_{t\ge 0}.
\end{aligned}
\end{equation}
		are martingales.
		\item \label{(ii)expindomainmin} The following convergence holds
		\[\frac{1}{t}\left(\mathbb{E}^y\big[\ee_x(Y^{\mathrm{m}}_t)\big]-\ee_x(y)\right)\underset{t\rightarrow 0}{\longrightarrow} \mathcal{Y}\ee_x(y), \ \forall (x,y)\in (0,\infty)^2.\]
	\end{enumerate}
\end{lem}
\begin{rem} Note that by definition $X^{\mathrm{m}}\equiv 0$ under $\mathbb{P}_0$ and  when $\hat{\Psi}(0)=0$, the process $M^{y,X^{\mathrm{m}}}$ defined in \eqref{eq:martingaleexpo} is also a martingale under $\mathbb{P}_0$. In contrast, if $\hat{\Psi}(0)<0$, this property fails, since in that case $\mathcal{X}\ee^y(0)<0$.
\end{rem}
\begin{proof}
For \ref{(i)expmartingale}. We only need to study $Y^{\mathrm{m}}$, as the martingale for $X^{\mathrm{m}}$ is obtained symmetrically by exchanging the mechanisms. Let $x,y\in (0,\infty)$. It\^o's formula ensures that the process \begin{equation}\label{eq:exponmartingaleY}\Big(\ee_x(Y_t^{\mathrm{m}})-\int_0^{t}\mathcal{Y}\ee_x(Y^{\mathrm{m}}_s)\Big)_{t\geq 0}
\end{equation}
is a local martingale under $\mathbb{P}^{y}$. Recall that there is $C_1\in (0,\infty)$, such that $|\Psi(y)|\leq C_1y^2$ for all $y\in [1,\infty)$. One has, almost surely, for all $s\geq 0$,
\[|\Psi(Y_s^\mathrm{m})|\mathbbm{1}_{\{Y_s^{\mathrm{m}}>1\}}\leq C_1(Y_s^{\mathrm{m}})^2, \  |\Psi(Y_s^\mathrm{m})|\mathbbm{1}_{\{Y_s^{\mathrm{m}}\leq 1\}}\leq C_0:=\sup_{[0,1]}|\Psi|\text{ and } xY_s^{\mathrm{m}}e^{-xY_s^{\mathrm{m}}}\leq 1.\]
 Thus, 
\begin{align*}
	|\mathcal{Y}\ee_x(Y^{\mathrm{m}}_s)|&=|e^{-xY_s^{\mathrm{m}}}x\Psi(Y_s^{\mathrm{m}})+Y_s^{\mathrm{m}}\hat{\Psi}(x)e^{-xY_s^{\mathrm{m}}}|\\
	&\leq \big(C_1(Y_s^{\mathrm{m}}x)^2\frac{1}{x}+ C_0x\big)e^{-xY_s^{\mathrm{m}}}+Y_s^{\mathrm{m}}|\frac{\hat\Psi(x)}{x}|xe^{-xY_s^{\mathrm{m}}}\\
	&\leq \frac{2C_1}{x}+C_0x+|\frac{\hat \Psi(x)}{x}|,
\end{align*}	 
where we used in the last inequality that $(yx)^2e^{-xy}\leq 2$ for all $y\in [0,\infty]$. Both terms in \eqref{eq:exponmartingaleY}
\begin{center}
$t\to \ee_x(Y_t)$ and $t\to \int_0^t\mathcal{Y}\ee_x(Y_s)\ddr s$ 
\end{center}
are bounded on finite intervals. This ensures that the local martingale is a true martingale.
\\
The second item \ref{(ii)expindomainmin} is a consequence of the first, together with the local boundedness of $s\mapsto \mathcal{Y}\ee_x(Y_s)$ previously shown. One has indeed, by Fubini and then Lebesgue's theorem
\begin{align*}
	\frac{1}{t}\big(\mathbb{E}^{y}[\ee_x(Y_t^{\mathrm{m}})]-\ee_x(y)\big)=	\frac{1}{t}&\int_0^{t}\mathbb{E}^{y}[e^{-xY_s^{\mathrm{m}}.}\psi(x,Y_s^{\mathrm{m}})]\ddr s\\
		&\underset{t\rightarrow 0}{\longrightarrow} \ee_x(y)\psi(x,y)=\mathcal{Y}\ee_x(y).
\end{align*}
\end{proof}
\begin{theo}[Duality between non-explosive minimal processes]\label{thmmin} Assume that the minimal $\mathrm{CBDI}$ processes $X^{\mathrm{m}}$ and $Y^{\mathrm{m}}$ with mechanisms respectively $(\Psi,\hat{\Psi})$ and $(\hat{\Psi},\Psi)$ do not explode. Then, under the convention $0^+\cdot \infty,\ \infty\cdot 0^+,$
		\begin{equation}\label{eq:laplaceduality0}\mathbb{E}_{x}[e^{-X^{\mathrm{m}}_ty}]=\mathbb{E}^{y}[e^{-xY^{\mathrm{m}}_t}], \ \forall (x,y)\in [0,\infty]^2,\ t \in [0,\infty).\end{equation}
		Moreover, $X^{\mathrm{m}}$, whose boundary $0$ is absorbing, is Feller at $0$:
		\[\mathbb{E}_{0}[e^{-X^{\mathrm{m}}_ty}]=\mathbb{E}_{0+}[e^{-X^{\mathrm{m}}_ty}]=\mathbb{P}^{y}(Y^{\mathrm{m}}_t<\infty)=1, \ \forall y\in [0,\infty],\ t \in [0,\infty),\]
	\end{theo}
\noindent Proposition \ref{prop:suffcondinaccessibility} yields  that if $\Hone$ and $\hatHone$ hold then the minimal $\mathrm{CBDIs}$ $X^{\mathrm{m}}$ and $Y^{\mathrm{m}}$, do not explode and Theorem~\ref{thmmin} applies. Duality relationships of the form \eqref{eq:laplaceduality0} will be established under weaker assumptions, later, when we consider extensions.
\begin{rem} The convention $\infty\cdot 0^+$ in \eqref{eq:laplaceduality0} is appropriate for working with the process ``minimal  at $\infty$" since it stipulates that $\mathbb{E}_{\infty}[e^{-X_t^{\mathrm{m}}y}]=\mathbb{E}^y[e^{-\infty\cdot Y_t^{\mathrm{m}}}]=0$ for all $t\in [0,\infty),\ y\in[0,\infty]$.  The convention $0^+\cdot \infty$ is related to the continuity at $0$ of the semigroup of $X^{\mathrm{m}}$.
\end{rem}
\begin{rem}
	In the terminology of \cite{foucartvidmar2025}, the processes $X^{\mathrm{m}}$ and $Y^{\mathrm{m}}$ have \textit{Laplace symbols}, given respectively by \[[0,\infty)\times [0,\infty)\ni (x,y)\mapsto \psi(x,y):=e^{xy}\mathcal{X}\ee^y(x)=x\Psi(y)+y\hat{\Psi}(x)\]  and 
	\[ [0,\infty)\times[0,\infty)\ni (x,y)\mapsto\hat\psi(x,y)=e^{xy}\mathcal{Y}\ee^y(x)=x\hat{\Psi}(y)+y\Psi(x).\] These are Laplace \textit{dual} symbols, in the sense that $\hat\psi$ and $\psi$ match when swapping the arguments: $\hat\psi(x,y)=\psi(y,x), \ \{x,y\}\subset (0,\infty)$. They appear in  \cite[Definition 4.19 and Section~6.1.2]{foucartvidmar2025}. 
\end{rem}
\begin{proof}[Proof of Theorem \ref{thmmin}]
We apply \cite[Theorem 5.1-(IV), pages 43-44]{foucartvidmar2025}  (a refined version of Ethier-Kurtz's classical result \cite[Theorem 4.11, p. 192]{EthierKurtz}). Let $\Psi,\hat{\Psi}$ of the form \eqref{branchingmechanism} and recall 	$\psi(x,y)=x\Psi(y)+y\hat{\Psi}(x)$. Lemma~\ref{lem;MPonexpminimalprocess}-(ii) ensures that the minimal processes $X^{\mathrm{m}}$ and $Y^{\mathrm{m}}$ solve the martingale problems associated to $\mathcal{X}$ and $\mathcal{Y}$ on the exponential functions. It remains to establish \cite[Conditions (IV), p.44]{foucartvidmar2025}. We list them below for the sake of clarity. Set
\begin{center}
$\phi(x,y)=e^{-xy}\psi(x,y)$, $(x,y)\in [0,\infty)^2$.
\end{center}
The conditions for the theorem to apply, numbered as in \cite{foucartvidmar2025}, are
\begin{align}
	&\phi(x,\cdot)\text{ is continuous at zero for all }x\in [0,\infty);\label{5.2}\tag{5.2}\\
	&\sup_{(x,y)\in [0,m]^2}\vert\phi(x,y)\vert<\infty\text{ for all }m\in [0,\infty);\label{5.4}\tag{5.4}\\
	&\phi(0,\cdot)\text{ vanishes on }[0,\infty],\label{5.6}\tag{5.6}
\end{align}
 and $\widehat{(5.2)}$ and $\widehat{(5.6)}$,  the conditions \eqref{5.2} and \eqref{5.6} with $\hat\psi(x,y):=x\hat{\Psi}(y)+x\Psi(y)$ instead of $\psi$. Moreover
\begin{equation}\label{5.7}
	\int_0^T\int_0^{T}\mathbb{E}[|\phi(X^{\mathrm{m}}_{s\land \sigma_\delta^{-}},Y^{\mathrm{m}}_{t\land\tau_\epsilon^{-}})|]\mathbbm{1}_{[0,T]}(s+t)\ddr s\ddr t<\infty,\quad T\in [0,\infty), \tag{5.7}
\end{equation}
 and
\begin{equation}\label{5.8}
	\mathbb{E}[\vert \phi(X^{\mathrm{m}}_{\sigma_{\delta}^{-}},Y^{\mathrm{m}}_{\tau_{\epsilon}^{-}}) \vert;\sigma_{\delta}^{-}+\tau_\epsilon^{-}\leq t]<\infty\text{ for a.e. $t\in [0,\infty)$}.\tag{5.8}
\end{equation}
with $\mathbb{E}$ the expectation  associated to $\mathbb{P}:=\mathbb{P}_{x}\times \mathbb{P}^{y}$. 

 Conditions \eqref{5.2}, $\widehat{\eqref{5.2}}$ and \eqref{5.4} are plainly fulfilled for any mechanisms. They follow from simple inspection. Conditions \eqref{5.6}, $\widehat{\eqref{5.6}}$ are true since by assumptions there is no killing term in $\Psi$  and $\hat{\Psi}$.

We now check the conditions \eqref{5.7} and \eqref{5.8} with $\sigma=\sigma_{\delta}^{-}, \ \tau=\sigma_{\epsilon}^{-}$. 
For any mechanism $\Psi$, one has, when $x\in (0,\infty)$, $e^{-xy}x\Psi(y)\underset{y\rightarrow \infty}{\longrightarrow}0$ as $\Psi$ has at most quadratic growth, see Section~\ref{sec:LKfunction}. This polynomial boundedness of L\'evy-Khintchine functions imply that for all $(\delta,\epsilon)\in (0,\infty)^2$
\[\underset{(x,y)\in[\delta,\infty)\times[\epsilon,\infty)}{\sup} \, e^{-xy}\vert x\Psi(y)+y\hat{\Psi}(x)\vert<\infty. \]
Then, since for all $s,t\geq 0$, $X_{s\wedge\sigma^{-}_{\delta}}\geq \delta$ and $Y_{s\wedge \tau_{\epsilon}^{-}}\geq \epsilon$, by the absence of negative jumps,  we get \eqref{5.7} and \eqref{5.8}. Finally, the assumptions of non-explosion of $X^{\mathrm{m}}$ and $Y^{\mathrm{m}}$ in Theorem~\ref{thm1infty} ensure that both processes have $\infty$ non-sticky. Theorem 5.1-(IV) in \cite{foucartvidmar2025} applies.
\end{proof}
\section{Extension at infinity}\label{sec:prooftheorem2infinity}
Our main results concerning extensions at $\infty$ are presented in Theorems~\ref{thm2infty} and~\ref{thm3infty}. We emphasize that we retain the notation $(\mathbb{P}_x)_{x\in [0,\infty]}$ to denote the law of the extended process. This will not lead to any ambiguity, as the underlying process will always be explicitly specified in the context.
\subsection{Duality, entrance law at $\infty$ and prelimiting processes}
We first observe that monotonicity with respect to the initial value, together with the Laplace duality, allows us to construct a first Fellerian extension at 
$\infty$ when this boundary is inaccessible.\begin{lem}\label{thm1infty} The minimal $\mathrm{CBDI}(\Psi,\hat{\Psi})$ process $X^\mathrm{m}$ admits a Markovian extension $X^{\mathrm{e}\infty}$ verifying almost surely,  for all $t\in [0,\infty)$
\begin{equation}\label{def:Xeinftyviamonotonicity}
 X^{\mathrm{e}\infty}_t(x)=X^{\mathrm{m}}_t(x) \text{ for all } x\in [0,\infty) \text{ and }
 X^{\mathrm{e}\infty}_t(\infty)=\underset{x\rightarrow \infty}{\lim}\!\uparrow  X^{\mathrm{m}}_t(x) \in [0,\infty].
\end{equation}
If moreover $X^{\mathrm{m}}$ and the minimal $\mathrm{CBDI}(\hat\Psi,\Psi)$, $Y^{\mathrm{m}}$, do not explode. Then, under  
$0^+\cdot \infty,\ \infty^-\cdot 0,$
\begin{equation}\label{eq:laplaceduality1}\mathbb{E}_{x}[e^{-X^{\mathrm{e}\infty}_ty}]=\mathbb{E}^{y}[e^{-xY^{\mathrm{m}}_t}], \ \forall (x,y)\in [0,\infty]^2,\ t \in [0,\infty).\end{equation}
In particular, the law of the process $X^{\mathrm{e}\infty}$ started  from $\infty$ at time $t$, which can be degenerated into $\delta_\infty$, has a Laplace transform characterized by 
\begin{equation}\label{eq:laplacedualityinfty} \mathbb{E}_{\infty}[e^{-X^{\mathrm{e}\infty}_ty}]=\mathbb{P}^{y}(Y^{\mathrm{m}}_t=0)=\mathbb{P}^{y}(\tau_0^{-}\leq t), \ \forall y\in [0,\infty],\ t \in [0,\infty).\end{equation}
The process $X^{\mathrm{e}\infty}$ has thus $\infty$ non-absorbing, and, since we work under the assumption of non-explosion, as an entrance, if and only if 
$$\mathbb{P}^{y}(\tau^{-}_0\leq t)>0 \text{ for some } t>0,$$
i.e. $0$ is accessible for the minimal $\mathrm{CBDI}(\hat{\Psi},\Psi)$, $Y^{\mathrm{m}}$.
\begin{table}[htpb]
\begin{center}
\begin{tabular}{|c|c|}
\hline
$\mathrm{CBDI}(\Psi,\hat{\Psi})$ &  $\mathrm{CBDI}(\hat{\Psi},\Psi)$ \\
\hline
$\infty$  non-absorbing  &  $0$ accessible \\
\hline
\end{tabular}
\vspace*{4mm}
\caption{Non-absorption at $\infty$ for $X^{\mathrm{e}\infty}$ and Accessibility of $0$ for $Y^{\mathrm{m}}$}
\label{correspondance}
\end{center}
\end{table}
\end{lem}

\begin{proof}
The comparison property along the initial values, Proposition~\ref{prop:comparison}-(i), ensures that one can extend the process at $\infty$ along the increasing limit \eqref{def:Xeinftyviamonotonicity}. Indeed, the collection of random variables $\big(X^{\mathrm{m}}_t(q), t\in [0,\infty), q\in \mathbb{Q}_+\big)$,  defined on the same probability space as $X^{\mathrm{m}}$, is non-decreasing along the rationals $q$ and thus admits an unique $[0,\infty]$-valued limit. Call it  $(X^{\mathrm{e}\infty}_t,\ t\geq 0)$ and denote its one-dimensional laws, on $[0,\infty]$, by $\big(P_t^{\mathrm{e}\infty}(\infty, \cdot), t\in [0,\infty)\big)$. 

We check the Markov property. Let $f\in \mathrm{B}_{[0,\infty]}$ be bounded increasing. Denote by $(P^{\mathrm{m}}_t)$ the semigroup of $X^{\mathrm{m}}$ and notice that the comparison property, Proposition~\ref{prop:comparison}-(i), ensures that $[0,\infty]\ni y\mapsto P^{\mathrm{m}}_tf(y)$ is bounded and increasing. Let $t\in [0,\infty)$, by the monotone convergence theorem
\begin{align*}P^{\mathrm{e}\infty}_{t+s}f(\infty)= \mathbb{E}[f(X^{\mathrm{e}\infty}_{t+s}(\infty))]&=\underset{x\rightarrow \infty}{\lim}\! \uparrow \mathbb{E}[f(X^{\mathrm{m}}_{t+s}(x))]\\
	&=\underset{x\rightarrow \infty}{\lim}\! \uparrow \mathbb{E}[P_s^{\mathrm{m}}f(X^{\mathrm{m}}_{t}(x))]\\
	&=\mathbb{E}[\underset{x\rightarrow \infty}{\lim} \! \uparrow P_s^{\mathrm{m}}f(X^{\mathrm{m}}_{t}(x))]\\
&=\mathbb{E}[P_s^{\mathrm{e}\infty}f(X^{\mathrm{e}\infty}_{t}(\infty))\mathbbm{1}_{\{X_t^{\mathrm{e}\infty}=\infty\}}]+\mathbb{E}[P_s^{\mathrm{m}}f(X^{\mathrm{e}\infty}_{t}(\infty))\mathbbm{1}_{\{X_t^{\mathrm{e}\infty}<\infty\}}]\\	&=\mathbb{E}[P_s^{\mathrm{e}\infty}f(X^{\mathrm{e}\infty}_{t}(\infty))]=P^{\mathrm{e}\infty}_{t}P^{\mathrm{e}\infty}_{s}f(\infty).\end{align*}
	The class of bounded increasing measurable functions being law-determining, they contain the class  $\{\mathbbm{1}_{(a,\infty]}, \ a\in [0,\infty]\}$ which generates $\mathrm{B}_{[0,\infty]}$, one sees that $\big(P^{\mathrm{e}\infty}_t, t\in [0,\infty)\big)$  is an \textit{entrance law}, that is to say for all $f\in \mathrm{B}_{[0,\infty]}$,
		\[\int_{[0,\infty]}P_{t+s}^{\mathrm{e}\infty}(\infty, \ddr x)f(x)=\int_{[0,\infty]}P_{t}^{\mathrm{e}\infty}(\infty, \ddr x )\int_{[0,\infty]}f(z)P_{s}^{\mathrm{e}\infty}(x, \ddr z).\]
By definition of $X^{\mathrm{e}\infty}$ and the duality relationship \eqref{eq:laplaceduality0}: \begin{equation} \label{eq:predualfinal} \mathbb{E}_{x}[e^{-X^{\mathrm{e}\infty}_ty}]=\mathbb{E}^{y}[e^{-xY^{\mathrm{m}}_t}], \ \forall (x,y)\in [0,\infty)\times [0,\infty],\ t \in [0,\infty).
\end{equation}
This holds for $x\in[0,\infty)$ and $y=\infty$ because of the convention $0^{+}\cdot \infty$. For $x=\infty$, by definition
\[ \mathbb{E}_{\infty}[e^{-X^{\mathrm{e}\infty}_ty}]=\underset{x\rightarrow\infty}{\lim} \mathbb{E}_{x}[e^{-X^{\mathrm{e}\infty}_ty}]=\underset{x\rightarrow\infty}{\lim}\mathbb{E}^{y}[e^{-xY^{\mathrm{m}}_t}]=\mathbb{P}^{y}(Y^{\mathrm{m}}_t=0), \ \forall t \in [0,\infty).\]
This latter identity matches with the duality \eqref{eq:predualfinal} for $x=\infty$ under the convention $\infty^{-}\cdot 0$. The Feller property is established in \cite[Proposition 3.23-(iv)]{foucartvidmar2025}. It comes from the fact that the map $$[0,\infty]\ni x\mapsto \mathbb{E}_x[e^{-X_t^{\mathrm{e}\infty}y}]$$ is continuous, together with Stone-Weierstrass theorem, guaranteeing the density of the linear span of $\{\ee^y: y\in (0,\infty)\}\cup\{1\}$ in $\big(\mathrm{C}([0,\infty]), \|\cdot\|_{\infty}\big)$.
\end{proof}
\begin{rem}
	The fact that $\big(P_t^{\mathrm{e}\infty}(\infty, \cdot), t\in [0,\infty)\big)$  forms an entrance law
	can also be established from the Laplace duality relationship \eqref{eq:dualityXeinfty}, see \cite[Proposition 3.20]{foucartvidmar2025}. 
\end{rem}

Let $\Psi$ and $\hat{\Psi}$ be two mechanisms and fix some decompositions $$\Psi=\Sigma-\Phi,\ \hat{\Psi}=\hat{\Sigma}-\hat{\Phi}.$$

We now observe, under the assumptions of Theorem~\ref{thm1infty}, that the following condition over $\hat{\Sigma}$ part of the drift $\hat{\Psi}$, $\int_1^\infty\frac{\ddr u}{\hat{\Sigma}(u)}<\infty$, is actually necessary for the boundary $\infty$ to be non-absorbing for $X^{\mathrm{e}\infty}$.
\begin{prop}\label{prop:nocdi} Assume that $X^{\mathrm{m}}$ and $Y^{\mathrm{m}}$ do not explode and
 $$\hatHtwo:\ \int_1^{\infty}\frac{\ddr u}{\hat{\Sigma}(u)}=\infty,$$ then the process $X^{\mathrm{e}\infty}$ has $\infty$ inaccessible absorbing (i.e. $\infty$ is a natural boundary).
\end{prop}
\begin{proof} The inaccessibility of $\infty$ follows from the assumption. By Proposition \ref{prop:suffcondinaccessibility}, when $\int^{\infty}_1\frac{\ddr u}{\hat\Sigma(u)}=\infty$, the dual process $Y^{\mathrm{m}}$ does not hit $0$. We see by \eqref{eq:laplacedualityinfty} in Lemma~\ref{thm1infty} plainly that $X^{\mathrm{e}\infty}$ has $\infty$ absorbing.
\end{proof}

The following proposition gives explicit conditions under which Lemma~\ref{thm1infty} applies, and thus guarantees the existence of the extended process $X^{\mathrm{e}\infty}$. It also provides a criterion, taken from \cite{rebotier}, for $\infty$ to be an entrance boundary. 
\medskip

Recall $\Hone$ and introduce the dual condition
\[\hypertarget{Hhat1}{\hat{\mathbb{H}}_1}:\
\int_0^1\frac{\ddr u}{\hat\Phi(u)}=\infty.\]
\begin{prop}\label{cor:explicitcondforXe}
Let $\Psi$ and $\hat{\Psi}$ be mechanisms, and denote by $\pi$ the L\'evy measure associated with $\Psi$. Assume that $\hat{\Psi} = \hat{\Sigma} - \hat{\Phi}$ with $\hat{\Sigma} \not\equiv 0$. Then the following statements hold:
\begin{enumerate}[1)]
\item Assume $\Hone$ and $\hatHone$. Then the process $X^{\mathrm{m}}$ admits a $[0,\infty]$-valued Feller extension at $\infty$, denoted by $X^{\mathrm{e}\infty}$, which satisfies \eqref{def:Xeinftyviamonotonicity} and the $(0^{+}\cdot \infty, \infty^-\cdot 0)$ Laplace duality relation \eqref{eq:laplaceduality1}. Moreover, $0$ is absorbing for $X^{\mathrm{e}\infty}$, and $\infty$ is inaccessible for $Y^{\mathrm{m}}$.

\item If $\hat{\Phi}'(0+) \in (0,\infty)$ (which in particular implies $\hatHone$) and
\begin{equation}\label{eq:JI}
\exists \,\kappa \in (0,\infty) \text{ such that } \int_{\kappa}^{\infty} \frac{1 + u \bar{\pi}(u)}{\hat{\Psi}(u)} \mathrm{d}u < \infty,
\end{equation}
then the extension $X^{\mathrm{e}\infty}$ exists and satisfies all the properties listed in (1). Moreover, $\infty$ is an instantaneous entrance boundary and $0$ is accessible for $Y^{\mathrm{m}}$.
\item If $\hat{\Phi}'(0+) \in (0,\infty)$, $\Phi'(0+) \in [0,\infty)$, and $\neg \hatHtwo$, that is,
$\int_1^{\infty} \frac{\mathrm{d}u}{\hat{\Sigma}(u)} < \infty$,
then \eqref{eq:JI} holds. Consequently, the extension $X^{\mathrm{e}\infty}$ exists and satisfies all the properties listed in (1) and (2).
\end{enumerate}
\end{prop}
\begin{rem}\label{rem:nonexplosionofYmforduality}
\begin{enumerate}
\item The condition $\Phi'(0+)\in [0,\infty)$ is equivalent to $\int_1^\infty \bar{\pi}(u)\ddr u<\infty$. By Lemma~\ref{lem:equivpsi}, the condition \eqref{eq:JI} is equivalent to $\int_{1}^{\infty}\frac{1+u\bar{\pi}(u)}{\hat{\Sigma}(u)}\ddr u<\infty$.
\item The condition \hyperlink{Hhat1}{$\hat{\mathbb{H}}_1$} in Proposition~\ref{cor:explicitcondforXe}-1) can be replaced by any other condition ensuring that $Y^{\mathrm{m}}$ does not explode. For instance, \cite[Theorem 3.1-(i)]{rebotier} shows that if $$\mathcal{I}:\ \int_{1}^{\infty}\frac{u\bar{\hat{\pi}}(u)}{\Sigma(u)}\ddr u<\infty$$ then $Y^{\mathrm{m}}$ has $\infty$ as an inaccessible boundary. By Lemma~\ref{thm1infty}, $X^{\mathrm{e}\infty}$ is then Feller at $0$ and has $0$ absorbing. Moreover, recalling the scale function $W$, see \eqref{eq:scalefunction}, one can verify, by Fubini's theorem, that condition $\mathcal{I}$ is equivalent to $\int_0^\infty\big(\hat{\Phi}(z)/z\big)'W(z)\ddr z<\infty$.
\item The entrance property at $\infty$ in the finite mean case plays a crucial role in the next and an other proof of Proposition~\ref{cor:explicitcondforXe}-3), not relying on \cite{rebotier}, will be provided in forthcoming Lemma \ref{lem:firstmean}.
\end{enumerate}
\end{rem}
 \begin{proof} 
Assume $\Hone$ and $\hatHone$. By Proposition~\ref{prop:suffcondinaccessibility} the minimal CBDIs $X^{\mathrm m}$ and $Y^{\mathrm m}$ with mechanisms $(\Psi,\hat\Psi)$ and $(\hat\Psi,\Psi)$ do not explode, so the first claim follows from Lemma~\ref{thm1infty}. For the second claim, $\hat\Phi'(0+)\in(0,\infty)$ implies $\hatHone$. We show that \eqref{eq:JI} yields non‑explosion of $X^{\mathrm m}$ by applying \cite[Theorem 3.1-(i)]{rebotier}. We verify the condition for the latter to apply called (B1) in \cite{rebotier}. By using the properties collected in Section~\ref{sec:LKfunction}, we have for some $\kappa>0$:
\begin{itemize}
  \item $\hat\Psi\in C^1$ and $\hat\Psi>0$ on $[\kappa,\infty)$;
  \item $z\mapsto\hat\Psi(z)/z$ is nondecreasing on $[\kappa,\infty)$ and, by \eqref{eq:JI}, necessarily tends to $+\infty$;
  \item $z\mapsto\hat\Psi'(z)/z$ is bounded on $[\kappa,\infty)$ (easily checked with $\hat\Psi'=\hat\Sigma'-\hat\Phi'$);
  \item $\displaystyle\int_{\kappa}^{\infty}\frac{u}{\hat\Psi(u)}\,\mathrm d u=\infty$, since $\hat\Psi(u)\le C u^2$ for $u$ large enough.
\end{itemize}
Hence, when \eqref{eq:JI} holds,  $X^{\mathrm m}$ does not explode and Lemma~\ref{thm1infty} applies. To see that $\infty$ is an instantaneous entrance, apply \cite[Theorem 3.5]{rebotier}, which require (B1) (checked above) and the one‑sided Lipschitz condition (B2)
\[
\exists\,b>0:\qquad \hat\Psi(y+z)-\hat\Psi(y)\ge -bz,\qquad y,z\ge0.
\]
By the convexity of L\'evy-Khintchine functions, we have $\hat\Psi(y+z)-\hat\Psi(y)\ge\hat\Psi'(y)z\ge\hat\Psi'(0+)z$, and since $-\hat\Psi'(0+)\le\hat\Phi'(0+)<\infty$ we may take $b=\hat\Phi'(0+)$. Finally, the duality \eqref{eq:laplaceduality1} for the extended process $X^{\mathrm{e}\infty}$ implies that $Y^{\mathrm m}$ hits $0$ with positive probability.
 \end{proof}
\subsection{Construction of the extension}
Let $X^{\mathrm{m}}$ be a minimal $\mathrm{CBDI}(\Psi,\hat{\Psi})$ process, and let $\Psi=\Sigma-\Phi$ denote its canonical decomposition.  We construct now through a limiting procedure a Fellerian extension at $\infty$ of $X^{\mathrm{m}}$ with  $\infty$ possibly accessible.

Let $\Psi^n$ be the mechanism obtained from $\Psi$ by truncating all jumps bigger than $n$, including the killing term, at level $n$, that is to say 
\begin{equation}\label{eq:psin}
\Psi^{n}:=\Sigma-\Phi^{n} \text{ with }
\Phi^{n}(y):=\gamma^+y+\int_{1}^{\infty}\big(1-e^{-yu}\big)\pi_n(\ddr u),\ y\in [0,\infty)
\end{equation}
with $\pi_n(\ddr u):=\pi(\ddr u)\mathbbm{1}_{[1,n)}(u)+(\bar{\pi}(n)+\lambda)\delta_n$ and $\bar{\pi}(n)=\pi([n,\infty))$.
\medskip

Plainly $(\Phi^{n})'(0)<\infty$, and the mechanism $\Phi^{n}$  satisfies $\Hone$. Proposition~\ref{cor:explicitcondforXe}-2 provides
a sequence $\big(X^{\mathrm{e}\infty, (n)}\big)_{n\geq 1}$ of extended $\mathrm{CBDI}(\Psi^{n},\hat{\Psi})$ processes, all with $\infty$ instantaneous entrance. 

\begin{theo}  \label{thm2infty} 
Assume $\neg \hatHtwo:\ \int_1^{\infty}\frac{\ddr u}{\hat{\Sigma}(u)}<\infty$ and $\hatHone:\ \int_0^1\frac{\ddr u}{\hat{\Phi}(u)}=\infty$. There exists a $[0,\infty]$-valued càdlàg Markov process  $X^{\mathrm{e}\infty}$, such that
 \[\forall x\in [0,\infty],\ \mathbb{P}_x-\text{a.s.} \ \underset {n\rightarrow \infty}{\lim}\! \uparrow X^{\mathrm{e}\infty, (n)}_t =X^{\mathrm{e}\infty}_t,\ \forall t\in [0,\infty) \text{ and } X^{\mathrm{e}\infty, (n)} \underset{n\rightarrow \infty}{\Longrightarrow} X^{\mathrm{e}\infty} \text{ in } \mathbbm{D}_{[0,\infty]}.\]
The process $X^{\mathrm{e}\infty}$ is a Fellerian continuous extension of $X^{\mathrm{m}}$ at $\infty$, absorbed at $0$. It satisfies, under $0^{+}\cdot \infty, \infty^- \cdot 0$,
\begin{equation}\label{eq:dualityXeinfty}\mathbb{E}_x[e^{-X_t^{\mathrm{e}\infty}y}]=\mathbb{E}^y[e^{-xY_t^{\mathrm{m}}}], \ \forall (x,y)\in [0,\infty]^2,\ t\in [0,\infty).\end{equation}
 \end{theo}
The process $X^{\mathrm{e}\infty}$ provides a natural candidate for extending $X^{\mathrm{m}}$ out from $\infty$ and past explosion if the latter occurs.  We  address the question whether $\infty$ is non-absorbing  or not for the process $X^{\mathrm{e}\infty}$ in the next section (Theorem~\ref{thm3infty}). 
\begin{proof}[Proof of Theorem \ref{thm2infty}] We divide it in  four steps. We explain them briefly. In Step \ref{step1}, we establish that the sequence of processes $X^{\mathrm{e}\infty, (n)}$ is non-decreasing. This yields existence of a process $X^{\mathrm{e}\infty}$ verifying for all $t\in [0,\infty)$, $\underset {n\rightarrow \infty}{\lim}\! \uparrow X^{\mathrm{e}\infty, (n)}_t =X^{\mathrm{e}\infty}_t$ almost surely. We show then that the latter, once stopped upon reaching $\infty$, has the same law as the minimal $\mathrm{CBDI}$. In Step 2, we analyze further the dual processes $Y^{\mathrm{m},(n)}$ and their limits as $n$ goes to $\infty$. This study is key in Step 3, where the Laplace duality is used to show that $X^{\mathrm{e}\infty}$ is a Feller Markov process and that $X^{\mathrm{e}\infty,(n)}$ 
converges in $\mathbbm{D}_{[0,\infty]}$ towards $X^{\mathrm{e}\infty}$. In Step 4, we study the generator of $X^{\mathrm{e}\infty}$ on the exponential functions. This last step allows us to study the dynamics at the boundary point $\infty$ and to conclude that the latter is continuous in the sense of Definition~\ref{def:continuous}.\\

Recall $\Psi^n=\Sigma-\Phi^n$ in \eqref{eq:psin} and the measure $\pi_n$.
Let $X^{\mathrm{m},(n)}$ be the minimal CBDI($\Psi^n,\hat\Psi$) process. In particular, $X^{\mathrm{m},(n)}$ is non-explosive as its jump measure has finite mean.  The process $X^{\mathrm{m},(n)}$ can be constructed as the unique strong solution to \eqref{SDECBDI} in which jumps have been truncated : any jump larger than $n$ is replaced by a jump of size $n$. In other words, $X^{\mathrm{m},(n)}$ satisfies the SDE \eqref{SDECBDI} in which the large jump term is replaced by
\begin{equation}\label{eq:truncationinSDE}
\int_0^t \int_0^{X^{\mathrm{m},(n)}_{s-}} \int_1^\infty (u\wedge n)\,\mathcal{N}(\mathrm{d}s,\mathrm{d}r,\mathrm{d}u).
\end{equation}

By Proposition~\ref{cor:explicitcondforXe}-2, under the assumptions $\neg \hatHtwo$ and $\hatHone$, the minimal process $X^{\mathrm{m},(n)}$ can be extended at $\infty$ into a Feller process by setting for all $t\in [0,\infty)$ 
\[X_t^{\mathrm{e}\infty,(n)}(x)=X_t^{\mathrm{m},(n)}(x) \ \ \text{for all } x\in [0,\infty) \  \text{and } X_t^{\mathrm{e}\infty,(n)}(\infty)=\underset{x\rightarrow\infty}{\lim} \uparrow X_t^{\mathrm{m},(n)}(x)\in[0,\infty]. \]
\begin{enumerate}[1]
\item \textbf{Existence of a pointwise limiting process of $\big(X^{\mathrm{e}\infty,(n)}\big)_{n\geq 1}$ as $n$ goes to $\infty$.} \label{step1}

\begin{lem}\label{lem:limitXn}
For all $x\in [0,\infty]$, all $n\geq 1$,
\[
X^{\mathrm{e}\infty,(n)}_t \leq X^{\mathrm{e}\infty,(n+1)}_t,\ \forall t\in [0,\infty)\quad \mathbb{P}_x\text{- a.s.}
\]
In particular, the sequence $\big(X^{\mathrm{e}\infty,(n)}\big)_{n\geq 1}$ admits an almost sure pointwise limit, and the following $[0,\infty]$-valued process is well-defined
\[X^{\mathrm{e}\infty}=\underset{n\rightarrow \infty}{\lim} \!\uparrow X^{\mathrm{e}\infty,(n)}.\]
\end{lem}

\begin{proof}
We proceed through an interlacing argument. Set $\Delta X^{\mathrm{m}}_t:=X^{\mathrm{m}}_t-X^{\mathrm{m}}_{t-}$ for all $t\geq 0$ and denote the first time at which $X^{\mathrm{m}}$ makes a jump larger than $n$ by
\[
J_n:=\inf\{t\geq 0:\ \Delta X^{\mathrm{m}}_t\geq n\}.
\]
Then, for all $n\geq 0$, we have, if $X^{\mathrm{m}}_0 = X^{\mathrm{m},(n)}_0 = X^{\mathrm{m},(n+1)}_0=x$,
\begin{equation}\label{eq:sautstronqu\'es}
\forall t\in[0,J_n), \qquad X^{\mathrm{m}}_t = X^{\mathrm{m},(n)}_t = X^{\mathrm{m},(n+1)}_t \text{ a.s.}.
\end{equation}

By \eqref{eq:sautstronqu\'es}, the processes $X^{\mathrm{m},(n)}$ and $X^{\mathrm{m},(n+1)}$ coincide on $[0,J_n)$. In view of the large-jumps dynamic \eqref{eq:truncationinSDE}, at time $J_n$, the jump size of $X^{\mathrm{m},(n+1)}$ is larger than that of $X^{\mathrm{m},(n)}$, therefore
$X^{\mathrm{m},(n)}_{J_n}\leq X^{\mathrm{m},(n+1)}_{J_n}$ a.s. 

Moreover, \[\big(X^{\mathrm{m},(n)}_{t+J_n}\big)_{t\geq 0}=\big(\tilde{X}^{\mathrm{m},(n)}_{t}(X^{\mathrm{m},(n)}_{J_n})\big)_{t\geq 0} \text{ and }
 \big(X^{\mathrm{m},(n+1)}_{t+J_n}\big)_{t\geq 0}=\big(\tilde{X}^{\mathrm{m},(n+1)}_{t}(X^{\mathrm{m},(n+1)}_{J_n})\big)_{t\geq 0}\] where $\tilde{X}^{\mathrm{m},(n)}$ and $\tilde{X}^{\mathrm{m},(n+1)}$ solve their respective shifted truncated stochastic equations~\eqref{SDECBDI}. Every jump of $X^{\mathrm{m},(n)}$ is also a jump of $X^{\mathrm{m},(n+1)}$, 
while additional (possibly compensated) positive jumps may occur for $X^{\mathrm{m},(n+1)}$ at atoms 
$(s,r,u)$ with $r\in (X_{s-}^{\mathrm{m},(n)},X_{s-}^{\mathrm{m},(n+1)}]$. Jumps larger than $1$ plainly preserve the order. The process $X^{\mathrm{m},(n+1)}$ (which, we recall, does not have negative jumps) may meet $X^{\mathrm{m},(n)}$ again through the diffusion or through the negative drift due to compensation. However, pathwise uniqueness, see Theorem~\ref{thm:minX}, for the stochastic equation truncated at level $n$, then ensures that the two processes coincide until the next jump of $X^{\mathrm{m},(n+1)}$ larger than $n$.  By iterating the argument, we obtain that 
\[X^{\mathrm{m},(n)}_t(x)\leq X^{\mathrm{m},(n+1)}_t(x), \qquad \forall t\geq J_n \quad \mathbb{P}_x\text{-a.s.}.\]
The order holds then on $[0,\infty)$. By definition, one has $X^{\mathrm{e}\infty,(n)}(x)\leq X^{\mathrm{e}\infty,(n+1)}(x)$ a.s. for all $x\in [0,\infty)$ and by letting $x$ go to $\infty$, we see that this also holds for $x=\infty$.
\end{proof}

\begin{rem}
Since for every $t\geq 0$, the two sequences $(X^{\mathrm{m},(n)}_t(x))_{n\geq 1}$ and $(X^{\mathrm{m},(n)}_t(x))_{x>0}$ are a.s. positive and non-decreasing, the two successive limits used to construct the process $X^{\mathrm{e}\infty}(\infty)$ 
can be exchanged :
\[X_t^{\mathrm{e}\infty}(\infty)=\underset{n\rightarrow\infty}{\lim} \, X_t^{\mathrm{e}\infty,(n)}(\infty)=\underset{n\rightarrow\infty}{\lim} \, \underset{x\rightarrow\infty}{\lim} \, X_t^{\mathrm{m},(n)}(x)=\underset{x\rightarrow\infty}{\lim} \, \underset{n\rightarrow\infty}{\lim} \, X_t^{\mathrm{m},(n)}(x).\]  
\end{rem}

  We show below that the process $X^{\mathrm{e}\infty}$ once stopped at its first hitting of $\infty$ coincides in law with a minimal $\mathrm{CBDI}$. The argument relies on the martingale problem formulation, see Definition~\ref{defCBDI}(i). To this end, we introduce a general convergence result for generators.

\begin{lem}\label{lem:cvPM}
Let $(X^n)_{n\geq 1}$ be a sequence of $\mathrm{CBDI}$ processes converging $\mathbb{P}_x$-a.s. pointwise towards a process $X$ for every $x\in[0,\infty]$. Let $\mathcal{X}^n$ denote the generator of $X^n$ and $\mathcal{X}$ be an operator such that for every $f\in \mathrm{C}_c^2((0,\infty))$, $\mathcal{X}f$ is continuous and
\[
\|\mathcal{X}^n f-\mathcal{X}f\|_\infty \xrightarrow[n\to\infty]{} 0.
\]
Then, $X$ solves the martingale problem 
\[\forall x\in(0,\infty), \ f\in \mathrm{C}^2_c((0,\infty)), \quad \left(f(X_t)-\int_0^t\mathcal{X}f(X_s)\ddr s, \ t\geq 0\right) \ \text{is a } \mathbb{P}_x\text{- martingale}.\]
\end{lem}

\begin{proof}
For every $t\in [0,\infty)$, $f\in \mathrm{C}_c^2((0,\infty))$, $x\in (0,\infty)$,  we have $\mathbb{P}_x$-a.s. 
    \begin{align*}
        \lvert \mathcal{X}f(X_t)-\mathcal{X}^nf(X^n_t)\rvert &\leq \lvert \mathcal{X}f(X_t)-\mathcal{X}f(X^n_t)\rvert +\lvert \mathcal{X}f(X^n_t)-\mathcal{X}^nf(X_t^n)\rvert\\
        &\leq \lvert \mathcal{X}f(X_t)-\mathcal{X}f(X_t^n)\rvert +\lVert \mathcal{X}f-\mathcal{X}^nf\rVert_\infty.
    \end{align*}
By continuity of $\mathcal{X}$, $\lvert \mathcal{X}f(X_t)-\mathcal{X}f(X_t^n)\rvert \rightarrow 0$, $\mathbb{P}_x$-a.s. as $n\rightarrow \infty$. Thus \[\mathcal{X}^nf(X_t^n)\underset{n\rightarrow \infty}{\longrightarrow} \mathcal{X}f(X_t) \quad \mathbb{P}_x\text{-a.s.}.\] 
Since $f \in \mathrm{C}_c^2((0,\infty))$, $\mathcal{X}f$ has a compact support and $\lVert \mathcal{X}f\rVert_\infty<\infty$. Therefore we have the bound for $n$ large enough $$\lvert \mathcal{X}^nf(X_s^n)\rvert\leq \lVert \mathcal{X}^nf\rVert_\infty\leq 1+\lVert \mathcal{X}f\rVert_\infty<\infty,$$ which is integrable over $[0,t]$. Thus, by Lebesgue's theorem :
\[f(X_t^n)-\int_0^t\mathcal{X}^nf(X_s^n)\ddr s \underset{n\rightarrow \infty}{\longrightarrow}f(X_t)-\int_0^t\mathcal{X}f(X_s)\ddr s, \qquad \mathbb{P}_x\text{-a.s.} \ \forall t\geq 0, \forall x\in(0,\infty).\]
We show that $\left(f(X_t)-\int_0^t\mathcal{X}f(X_s)\ddr s\right)_{t\geq 0}$ is a $\mathbb{P}_x$-martingale for every $x\in [0,\infty]$. Let $0\leq t_1\leq \dots\leq t_m\leq s<t$ and $f_1, \, \dots \,, f_m$ continuous and bounded functions. By dominated convergence, we have for every $x\in(0,\infty)$ :
\begin{align*}
   & \mathbb{E}_x\left[\left(f(X_t)-f(X_s)-\int_s^t\mathcal{X}f(X_r)\ddr r\right)\overset{m}{\underset{i=1}{\prod}}f_i(X_{t_i})\right]\\
   &=\underset{n\rightarrow \infty}{\lim} \, \mathbb{E}_x\left[\left(f(X^n_t)-f(X^n_s)-\int_s^t\mathcal{X}^nf(X^n_r)\ddr r\right)\overset{m}{\underset{i=1}{\prod}}f_i(X^n_{t_i})\right] =0.
\end{align*}
Set $M_t:=f(X_t)-\int_0^t\mathcal{X}f(X_s)\ddr s$ for all $t\in [0,\infty)$,  and let $\mathcal{F}_s$ be the $\sigma$-algebra generated by $\left(X_u, \ u\leq s\right)$. The previous computation implies in particular that for any random variable $F_s$ of the form $\overset{m}{\underset{i=1}{\prod}}f_i(X_{t_i})$,
\[\mathbb{E}_x\left[\mathbb{E}_x[M_t|\mathcal{F}_s]F_s\right]=\mathbb{E}_x\left[M_sF_s\right].\]
By a functional monotone class argument, we get
\[\mathbb{E}_x[M_t|\mathcal{F}_s]=M_s, \quad \mathbb{P}_x\text{-a.s.}, \quad \forall x\in(0,\infty).\]
Thus $M$ is a $\mathbb{P}_x$-martingale
\end{proof}
\begin{lem}\label{lem:extinfty}
We have for all $x\in[0,\infty]$, under $\mathbb{P}_x$,
\[
\left(X_{t\wedge \sigma_\infty^{\mathrm{e}\infty,+}}^{\mathrm{e}\infty}, \ t\geq 0\right)
=
\left(X_t^{\mathrm{m}}, \ t\geq 0\right)
\quad \text{in law},
\]
where $\sigma_\infty^{\mathrm{e}\infty,+}$ denotes the first hitting time of $\infty$ by $X^{\mathrm{e}\infty}$. 
\end{lem}

\begin{proof}
Let $\mathcal{X}^{(n)}$ denote the generator of $X^{\mathrm{e}\infty,(n)}$. Recall that $\mathcal{X}$ is the generator of $X^\mathrm{m}$.
Take $f\in \mathrm{C}_c^2((0,\infty))$ and choose $k\geq 1$ such that $\mathrm{supp}(f)\subset[0,k]$.
For every $x\in [0,\infty)$ :
\begin{align*}
    \mathcal{X}f(x)-\mathcal{X}^{(n)}f(x) &= x  \big(\mathrm{L}^{-\Phi}f(x)-\mathrm{L}^{-\Phi^n}f(x)\big) \\
    &= x  \int_1^\infty(f(x+h)-f(x))(\pi(\ddr h)-\pi_n(\ddr h))+x\lambda(f(\infty)-f(x))\\
    &=x\int_n^\infty(f(x+h)-f(x))\pi(\ddr h)\\
    &\qquad \qquad -x(1-e^{-xn})\big(\bar{\pi}(n)+\lambda)(f(x+n)-f(x)\big)  -\lambda xf(x).
\end{align*}
For $n>k$, we have :
\begin{align*}
    \mathcal{X}f(x)-\mathcal{X}^{(n)}f(x) &=-xf(x)\bar{\pi}(n)+x(1-e^{-xn})(\bar{\pi}(n)+\lambda)f(x)-\lambda xf(x)\\
    &=-xf(x)e^{-xn}(\bar{\pi}(n)+\lambda).
\end{align*}
Hence, since the function $x\mapsto xe^{-xn}$ reaches its maximum at $1/n$:
\[
\big|\mathcal{X}f(x)-\mathcal{X}^{(n)}f(x)\big|
\leq \frac{1}{ne}\|f\|_\infty(\bar{\pi}(n)+\lambda)
\xrightarrow[n\to\infty]{} 0,
\]
uniformly in $x\in[0,\infty]$. Thus \[\lVert \mathcal{X}f-\mathcal{X}f\rVert_\infty \underset{n\rightarrow \infty}{\longrightarrow} 0.\]
Therefore, by Lemma~\ref{lem:cvPM}, the process $X^{\mathrm{e}\infty}$ solves the following martingale problem :
\[\forall x\in(0,\infty), \ f\in \mathrm{C}^2_c((0,\infty)), \quad \left(f(X^{\mathrm{e}\infty}_t)-\int_0^t\mathcal{X}f(X^{\mathrm{e}\infty}_s)\ddr s, \ t\geq 0\right) \ \text{is a } \mathbb{P}_x\text{-martingale}.\]
Stopping the latter yields that 
\[\left(f(X^{\mathrm{e}\infty}_{t\wedge\sigma_\infty^{\mathrm{e}\infty,+}})-\int_0^{t\wedge\sigma_\infty^{\mathrm{e}\infty,+}}\!\!\!\!\mathcal{X}f(X^{\mathrm{e}\infty}_{s\wedge\sigma_\infty^{\mathrm{e}\infty,+}})\ddr s, \ t\geq 0\right) \ \text{is a } \mathbb{P}_x\text{-martingale}.\]
Since $f\in \mathrm{C}^2_c((0,\infty))$, $\mathcal{X}f(\infty)=0$,
and we have that
\[\int_{t\wedge\sigma_\infty^{\mathrm{e}\infty,+}}^t\mathcal{X}f(X^{\mathrm{e}\infty}_{s\wedge\sigma_\infty^{\mathrm{e}\infty,+}})\ddr s=0, \ \mathbb{P}_x\text{-a.s..}\]
Therefore, the stopped process $\big(X^{\mathrm{e}\infty}_{t\wedge\sigma_\infty^{\mathrm{e}\infty,+}}, \ t\geq 0\big)$ satisfies the same martingale problem as $X^{\mathrm{m}}$ and by Theorem~\ref{thm:minX}, is a minimal $\mathrm{CBDI}(\Psi,\hat\Psi)$.
\end{proof}

At this stage, the process $X^{\mathrm{e}\infty}$ is not yet known to be a Markov Feller process. The main tool used to prove that the process $X^{\mathrm{e}\infty}$ has the Markov property and satisfies the Feller property is Laplace duality. We investigate it in Step \ref{sec:ext2}.

\item \textbf{Study of the dual processes of $X^{\mathrm{e}\infty,(n)}$ and their limit as $n$ goes to $\infty$.}\label{sec:ext2}

Let $Y^{\mathrm{m},(n)}$ be the $\mathrm{CBDI}(\hat\Psi,\Psi^n)$. By Proposition~\ref{cor:explicitcondforXe}, under the convention $0^+\!\cdot\!\infty$, $\infty^-\!\cdot\!0$, the processes $X^{\mathrm{e}\infty,(n)}$ and $Y^{\mathrm{m},(n)}$ are in Laplace duality, that is,
\begin{equation}\label{eq:duality(n)}
\mathbb{E}_x\!\left[e^{- X^{\mathrm{e}\infty,(n)}_t y}\right]
=
\mathbb{E}^y\!\left[e^{-x Y^{\mathrm{m},(n)}_t}\right],
\qquad \forall t\geq 0,\ \forall x,y\in[0,\infty].
\end{equation}

In order to study the process $X^{\mathrm{e}\infty}$ through its semigroup, we investigate the sequence $(Y^{\mathrm{m},(n)})_{n\geq 1}$. Lemma~\ref{lem:cvYn} shows that its limit is a minimal $\mathrm{CBDI}(\hat\Psi,\Psi)$, $Y^{\mathrm{m}}$. 

Let us start by studying the drift  functions of $Y^{\mathrm{m},(n)}$.

\begin{lem}\label{lem:phin}
The sequence of functions $(\Phi^n)_{n\geq 1}$ is non-decreasing and for all $n\ge 1$,
$$0\leq \Phi(x)-\Phi^n(x)\leq \bar{\pi}(n)+\lambda e^{-xn}, \ \ \forall x\in (0,\infty).$$
In particular $(\Phi^n)_{n\geq 1}$ converges locally uniformly towards $\Phi$.
\end{lem}

\begin{proof}
Fix $n\geq1$. A direct computation gives for all $x\in [0,\infty)$
\begin{align*}
\Phi^{n+1}&(x)-\Phi^n(x)\\
            &=\int_n^{n+1}(1-e^{-xu})\pi(du)+(e^{-xn}-1)(\bar{\pi}(n)+\lambda)-(e^{-x(n+1)}-1)(\bar{\pi}(n+1)+\lambda)\\
            &\geq (1-e^{-x(n+1)})(\bar{\pi}(n)-\bar{\pi}(n+1))+(e^{-xn}-1)(\bar{\pi}(n)+\lambda)
\\          &\qquad \qquad\qquad\qquad\quad\qquad \qquad\qquad -(e^{-x(n+1)}-1)(\bar{\pi}(n+1)+\lambda)\\
            &=(e^{-xn}-1)(\bar{\pi}(n)+\lambda)-(e^{-x(n+1)}-1)(\bar{\pi}(n)+\lambda)\\
            &=(e^{-xn}-e^{-x(n+1)})(\bar{\pi}(n)+\lambda)\geq 0,
\end{align*}
which proves the first claim. For the second one,  for all $n\geq 1$ and $x\in [0,\infty)$,
\begin{align*}
\Phi(x)-\Phi^n(x)&=\int_n^{\infty}(1-e^{-xu})\pi(du)+(e^{-xn}-1)(\bar{\pi}(n)+\lambda)+\lambda\\
            &\leq \bar{\pi}(n)+\lambda e^{-xn}.
\end{align*}
\end{proof}

Next, we show that the monotonicity property along the drift, combined with comparison with respect to the initial value, ensures that the associated processes are ordered.

\begin{lem}\label{lem:cvYn}
    The sequence of processes $\big(Y^{\mathrm{m},(n)}_t, \ t\geq 0\big)_{n\geq 1}$ satisfies for all $y\in [0,\infty]$ and $n\geq 1$,
    \[Y^{\mathrm{m},(n+1)}_t\geq Y^{\mathrm{m},(n)}_t, \ \forall t\in [0,\infty),\  \mathbb{P}^{y}\text{-a.s..}\]
    Moreover $(Y^{\mathrm{m},(n)})_{n\geq 1}$
    converges pointwise almost surely, as $n\to\infty$, to a minimal $\mathrm{CBDI}(\hat\Psi,\Psi)$, denoted by $Y^{\mathrm{m}}$.
\end{lem}

\begin{proof}
    Since $(\Psi^n)_{n\geq 1}$ is non-increasing,  Proposition~\ref{prop:comparison} implies that the sequence of processes $(Y^{\mathrm{m},(n)})_{n\geq 1}$ is almost surely non-decreasing. Hence it admits an almost sure pointwise limit, denoted by $Y^{(\infty)}$. We show that the latter  satisfies the same martingale problem as $Y^{\mathrm{m}}$. The computation is similar as the one made in the proof of \cite[Lemma~7.1 page 22]{MR3940763}, we provide details for the sake of clarity. Let $\mathcal{Y}^{(n)}$ and $\mathcal{Y}$ be the generators of $Y^{\mathrm{m},(n)}$ and $Y^{\mathrm{m}}$. Apart from their drift parts, the operators $\mathcal{Y}^{(n)}$ and $\mathcal{Y}$ are matching, one has therefore by Lemma~\ref{lem:phin}, for all $f\in \mathrm{C}_c^2((0,\infty))$, 
    \begin{align*}
        \lVert \mathcal{Y}^{(n)}f-\mathcal{Y}f\rVert_\infty&=\underset{x\in(0,\infty)}{\sup} \, |f'(x)|(\Phi(x)-\Phi_n(x))\leq  \underset{x\in(0,\infty)}{\sup} \, |f'(x)| \big(\bar{\pi}(n)+\lambda e^{-xn}\big).
    \end{align*}
    Since $f'$ has compact support in $(0,\infty)$, $\bar{\pi}(n)\underset{n\rightarrow \infty}{\longrightarrow} 0$ and $e^{-xn} \rightarrow 0$ as $n\rightarrow \infty$ for every $x\in(0,\infty)$, we have $\underset{x\in(0,\infty)}{\sup} \, |f'(x)| \big(\bar{\pi}(n)+\lambda e^{-xn}\big) \rightarrow 0$ as $n \rightarrow \infty$. Thus,  $ \lVert \mathcal{Y}^{(n)}f-\mathcal{Y}f\rVert_\infty \rightarrow 0 $ as $n\rightarrow \infty$. By Lemma \ref{lem:cvPM}, $Y^{(\infty)}$ satisfies the same martingale problem as $Y^{\mathrm{m}}$. 
    
    It only remains to argue that $Y^{(\infty)}$ is absorbed at the boundaries to conclude that this is a minimal $\mathrm{CBDI}$. This is readily checked, since each $Y^{\mathrm{m},(n)}$ is minimal at $0$ and $\infty$, we have 
    \begin{align*}
        &\forall t\geq 0, \ Y_t^{(\infty)}=\lim_{n\rightarrow\infty} Y_t^{\mathrm{m},(n)}=\infty, \ \mathbb{P}_\infty\text{-a.s.} \ \text{ and } \forall t\geq 0, \ Y_t^{(\infty)}=\lim_{n\rightarrow\infty} Y_t^{\mathrm{m},(n)}=0, \ \mathbb{P}_0\text{-a.s.}.
    \end{align*}
     We conclude that for all $x\in[0,\infty]$, $\big(Y_t^{(\infty)}, t\geq 0\big)=\big(Y_t^{\mathrm{m}}, t\geq 0\big)$ in law under $\mathbb{P}_x$.
\end{proof}

The following duality identity is an immediate consequence of the previous lemma.

\begin{lem}\label{lem:dualityext}
    Under the convention $0^+\cdot \infty, \infty^{-}\cdot 0$, we have for all $x, y\in [0,\infty]$,
    \begin{equation}\label{eq:dualityextension}
        \mathbb{E}_x\left[e^{-X_t^{\mathrm{e}\infty}y}\right]=\mathbb{E}^{y}\left[e^{-xY_t^{\mathrm{m}}}\right], \qquad \forall t\geq 0.
    \end{equation}
\end{lem}

\begin{proof}
    Passing to the limit $n\to\infty$ in \eqref{eq:duality(n)} yields the result by dominated convergence.
\end{proof}
\item \textbf{Study of the semigroup of $X^{\mathrm{e}\infty}$.}\label{step3}
We now turn to the problem of showing that $X^{\mathrm{e}\infty}$ is Markovian and admits a càdlàg version. 

Define the operator $(P^{\mathrm{e}\infty}_t)_{t\in [0,\infty)}$ on $\mathrm{C}([0,\infty])$ by
\begin{equation}
    \label{eq:semigroupextensioninfty}
    P_t^{\mathrm{e}\infty}\ee^y(x):=\mathbb{E}_x\left[e^{-X_t^{\mathrm{e}\infty}y}\right]=\mathbb{E}^y\left[e^{-xY_t^{\mathrm{m}}}\right], \qquad x\in [0,\infty],
\end{equation}
with the convention $0^+\cdot \infty,\infty^-\cdot 0$, see Table~\ref{conventiontable}, and where $\ee^y:x\mapsto e^{-xy}$ for $y\in(0,\infty)$ and $\ee^0=1$.
\begin{lem}\label{lem:extsemigroup}
    The operator $(P^{\mathrm{e}\infty}_t)_{t\in [0,\infty)}$ is a Feller semigroup.
\end{lem}
\begin{proof}
We first check the semigroup property.     Take $s,t\geq 0$ and $y\in [0,\infty]$. Denote by $(Q_t)_{t\in [0,\infty)}$ the semigroup of $Y^{\mathrm{m}}$ and recall $\ee_x:y\mapsto e^{-xy}$ for every $x\in [0,\infty]$. Take $x\in[0,\infty], y\in[0,\infty).$ By the duality relation \eqref{eq:dualityextension} and the semigroup property of $(Q_t)_{t\in [0,\infty)}$, we have :
\begin{align*}
     P_{t+s}^{\mathrm{e}\infty}\ee^{y}(x)&=Q_{t+s}\ee_{x}(y)=Q_t\left(Q_s\ee_x\right)(y)=\mathbb{E}^y\left[\mathbb{E}^{Y_t^{\mathrm{m}}}\left[e^{-xY_s^{\mathrm{m}}}\right]\right].
\end{align*}
The duality relation \eqref{eq:dualityextension} gives 
\[\mathbb{E}^{Y_t^{\mathrm{m}}}\left[e^{-xY_s^{\mathrm{m}}}\right]=\mathbb{E}_x\left[e^{-X^{\mathrm{e}\infty}_s\tilde{Y}_t^\mathrm{m}}\right],\]
with $\tilde{Y}^{\mathrm{m}}$ is a copy of $Y^{\mathrm{m}}$ and $\tilde{\mathbb{P}}_y$ its law when it has initial value $\tilde{Y}^{\mathrm{m}}_0=y$. Thus by Fubini-Tonelli
\begin{align*}
    P_{t+s}^{\mathrm{e}\infty}\ee^{y}(x)&=\tilde{\mathbb{E}}^y\left[\mathbb{E}_x\left[e^{-X^{\mathrm{e}\infty}_s\tilde{Y}_t^\mathrm{m}}\right]\right]=\mathbb{E}_x\left[\tilde{\mathbb{E}}^y\left[e^{-X^{\mathrm{e}\infty}_s\tilde{Y}_t^\mathrm{m}}\right]\right].
\end{align*}
By using \eqref{eq:dualityextension} again, we obtain :
\[\tilde{\mathbb{E}}^y\left[e^{-X^{\mathrm{e}\infty}_s\tilde{Y}_t^\mathrm{m}}\right]=\mathbb{E}_{X^{\mathrm{e}\infty}_s}\left[e^{-X^{\mathrm{e}\infty}_ty}\right],\]
and then
\begin{align*}
      P_{t+s}^{\mathrm{e}\infty}\ee^{y}(x)&=\mathbb{E}_x\left[\mathbb{E}_{X^{\mathrm{e}\infty}_s}\left[e^{-X^{\mathrm{e}\infty}_ty}\right]\right]=P^{\mathrm{e}\infty}_s(P^{\mathrm{e}\infty}_t \ee^y)(x).
\end{align*}
The semigroup property is therefore established. The Feller property is an application of \cite[Proposition 3.23-(iv)]{foucartvidmar2025}.
\end{proof}
We now address the weak convergence in  $\mathbb{D}_{[0,\infty]}$ of $(X^{\mathrm{e}\infty,(n)},n\geq 1)$ towards $X^{\mathrm{e}\infty}$.

\begin{lem}\label{prop:cvskorokhod}
    One has
    \begin{equation*}
        X^{\mathrm{e}\infty,(n)} \underset{n\rightarrow\infty}{\Longrightarrow}X^{\mathrm{e}\infty}.
    \end{equation*}
    \end{lem}
    \textit{Proof.}        Denote by $\left((P^{(n)}_t)_{t\in [0,\infty)},n\geq 1\right)$ the semigroups of $(X^{\mathrm{e}\infty,(n)}, n\geq 1)$. We show that they converge uniformly towards that of $X^{\mathrm{e}\infty}$, that is to say, for every $f\in C([0,\infty])$ and all $t\in [0,\infty)$:
\[\lVert P_t^{(n)}f-P_t^{\mathrm{e}\infty}f\rVert_\infty:=\sup_{x\in [0,\infty]}|P_t^{(n)}f(x)-P_t^{\mathrm{e}\infty}f(x)|\underset{n\rightarrow\infty}{\longrightarrow} 0.\]
        We will then conclude by applying \cite[Theorem~2.5, page 167]{EthierKurtz} that the sequence of processes $(X^{\mathrm{e}\infty,(n)})_{n\geq 1}$ converges weakly in $\mathbb{D}_{[0,\infty]}$ towards $X^{\mathrm{e}\infty}$. Fix $t\in [0,\infty)$. Recall that $\ee^y(x)=e^{-xy}, \ x\in [0,\infty]$ and that we work with the convention $\infty^{-}\cdot 0$ (i.e. $\infty\cdot 0=0)$. 

By Stone-Weierstrass theorem, it is sufficient to establish that, for every $y\in [0,\infty)$,  
        $\lVert P_t^{(n)}\ee^y-P_t^{\mathrm{e}\infty}\ee^y\rVert_\infty \rightarrow 0$ as $n \rightarrow\infty$. Let $n\geq 1$.  Since, by Lemma~\ref{lem:cvYn}, $Y_t^{\mathrm{m}} \geq Y_t^{\mathrm{m},(n)}$ almost surely, we have $\mathbb{P}^{y}(Y^{\mathrm{m}}_t=0, \ Y^{\mathrm{m},(n)}_t>0)=0$ and
        \begin{align*}
\lVert P^{(n)}_t\ee^y-P_t^{\mathrm{e}\infty}\ee^y\rVert_\infty&\leq \mathbb{E}^y\left[\underset{x\in[0,\infty]}{\sup} \, \lvert e^{-xY_t^{\mathrm{m},(n)}}-e^{-xY_t^{\mathrm{m}}}\rvert \right] \\
&=\mathbb{E}^y\left[\underset{x\in[0,\infty]}{\sup}\left(e^{-xY_t^{\mathrm{m},(n)}}-e^{-xY_t^{\mathrm{m}}}\right)\mathbbm{1}_{\{Y_t^{\mathrm{m}}>0\}}\right]\\
& \quad +  \mathbb{E}^y\left[\underset{x\in[0,\infty]}{\sup}\left(1-e^{-xY^{\mathrm{m},(n)}_t}\right)\mathbbm{1}_{\{Y^{\mathrm{m}}_t=0,\ Y^{\mathrm{m},(n)}_t>0\}}\right]\\
& =\mathbb{E}^y\left[\underset{x\in[0,\infty]}{\sup}\left(e^{-xY_t^{\mathrm{m},(n)}}-e^{-xY^{\mathrm{m}}_t}\right)\mathbbm{1}_{\{Y^{\mathrm{m}}_t>0\}}\right].
\end{align*}
On the event $\{Y_t^{\mathrm{m}}>0\}$, the monotone convergence $Y^{\mathrm{m},(n)}_t \uparrow Y_t^{\mathrm{m}}$ a.s. ensures also that for large enough $n$, $Y_t^{\mathrm{m},(n)}>0$ almost surely. On the other hand, one can verify that the function $$[0,\infty]\ni x\mapsto e^{-xY_t^{\mathrm{m},(n)}}-e^{-xY_t^{\mathrm{m}}}$$ reaches its maximum at the value $x_{t,n}:=\frac{\log Y_t^{\mathrm{m}}-\log Y_t^{\mathrm{m},(n)}}{Y_t^{\mathrm{m}}-Y_t^{\mathrm{m},(n)}}$ almost surely. Moreover, as $n\rightarrow\infty$, we have $x_{t,n} \rightarrow \frac{1}{Y_t^{\mathrm{m}}}$ almost surely. Thus,  
\[\underset{x\in[0,\infty]}{\sup}\left(e^{-xY_t^{\mathrm{m},(n)}}-e^{-xY^{\mathrm{m}}_t}\right)\mathbbm{1}_{\{Y^{\mathrm{m}}_t>0\}}=\left(e^{-x_{t,n}Y_t^{\mathrm{m},(n)}}-e^{-x_{t,n}Y^{\mathrm{m}}_t}\right)\mathbbm{1}_{\{Y^{\mathrm{m}}_t>0\}} \underset{n\rightarrow\infty}{\longrightarrow} 0 \quad \mathbb{P}^y\text{-a.s.}.\]
By dominated convergence, it follows that 
\begin{equation*}
\mathbb{E}^y\Big[\sup_{x\in[0,\infty]}\big(e^{-xY_t^{\mathrm{m},(n)}}-e^{-xY^{\mathrm{m}}_t}\big)\mathbf{1}_{\{Y^{\mathrm{m}}_t>0\}}\Big]
\longrightarrow 0 \quad \text{as } n \to \infty. \tag*{\qed}
\end{equation*}
\item \label{stepcontinuityinfinity} \textbf{Infinitesimal generator of $X^{\mathrm{e}\infty}$ and continuity of the boundary $\infty$}. \label{step4} 
We study in this step the \textit{pointwise} infinitesimal generator of the Feller process $X^{\mathrm{e}\infty}$, that is to say, the operator $\big(\mathcal{X}^{\mathrm{e}\infty}, \mathcal{D}^{\mathrm{p}}_{X^{\mathrm{e}\infty}}\big)$,  satisfying
    \begin{equation}\label{eq:pointwisedomain}\frac{1}{t}\left(P_t^{\mathrm{e}\infty}f-f\right)\underset{t\rightarrow 0+}{\longrightarrow} \mathcal{X}^{\mathrm{e}\infty}f, \ f\in  \mathcal{D}^{\mathrm{p}}_{X^{\mathrm{e}\infty}}
    \end{equation}
    where $\mathcal{D}^{\mathrm{p}}_{X^{\mathrm{e}\infty}}$ is the set of functions for which the convergence above holds pointwise.
	\begin{rem} For a Feller process, the strong and the pointwise infinitesimal generators match, B\"ottcher et al. \cite[Theorem 1.33]{zbMATH06256582}.
		\end{rem}
		
\begin{lem}\label{lem:genXeinfinity} For all $y\in (0,\infty)$, $\mathrm{e}^{y}\in \mathcal{D}^{\mathrm{p}}_{X^{\mathrm{e}\infty}}$, $$\mathcal{X}^{\mathrm{e}\infty}\mathrm{e}^{y}(x)=\mathcal{X}\mathrm{e}^{y}(x), \ x\in [0,\infty), \text{ and } \mathcal{X}^{\mathrm{e}\infty}\mathrm{e}^{y}(\infty)=0.$$
In particular, the process $X^{\mathrm{e}\infty}$ does not jump from $\infty$.\end{lem}
The boundary $\infty$ is therefore either absorbing or \textit{continuous} non-absorbing.
\begin{rem}\label{rem:noinstantaneousjump} Since $\ee^y$ is in the domain of the generator of $X^{\mathrm{e}\infty}$, for all $y\in (0,\infty)$, Dynkin's formula ensures that the process $M^{y,X^{\mathrm{e}\infty}}$ defined in \eqref{eq:martingaleexpo} is a martingale under $\mathbb{P}_x$, for all $x\in [0,\infty)$. 
The process $X^{\mathrm{e}\infty}$  satisfies therefore the martingale problem, $\mathrm{MP}(\mathcal{X},\mathcal{D})$,  with
$\mathcal{D}:= \mathrm{span}\{\ee^y:y\in (0,\infty)\}\subset \mathcal{D}^{\mathrm{p}}_{X^{\mathrm{e}\infty}}$. It is worth recalling that $X^{\mathrm{m}}$ solves also this martingale problem, see Lemma~\ref{lem;MPonexpminimalprocess}-(i). Hence, if $X^{\mathrm{e}\infty}\not\equiv X^{\mathrm{m}}$, then $\mathrm{MP}(\mathcal{X},\mathcal{D})$ is not well-posed. Such a situation will be encountered later, see the next section.
\end{rem}

\begin{proof}[Proof of Lemma~\ref{lem:genXeinfinity}]
Let $y\in (0,\infty)$. Recall the duality relationship \eqref{eq:dualityXeinfty}, 
\[\mathbb{E}_x[e^{-X_t^{\mathrm{e}\infty}y}]=\mathbb{E}^y[e^{-xY_t^{\mathrm{m}}}],\ x\in [0,\infty], \ t\in [0,\infty).\]
We start by establishing that $\mathrm{e}^{y}$ is in the domain of $X^{\mathrm{e}\infty}$.  Let $x\in [0,\infty)$, then
\[\frac{1}{t}\left(\mathbb{E}_x[e^{-X_t^{\mathrm{e}\infty}}]-e^{-xy}\right)=\frac{1}{t}\left(\mathbb{E}^y[e^{-xY_t^{\mathrm{m}}}]-e^{-xy}\right)\underset{t\rightarrow 0}{\longrightarrow} \mathcal{Y}\mathrm{e}_x(y)=\mathcal{X}\mathrm{e}^y(x),\]
where the limit is deduced from  Lemma~\ref{lem;MPonexpminimalprocess}-(ii) and the last equality is the duality \eqref{eq:gendual}. 

Last, consider the starting point $x=\infty$. By the duality relationship \eqref{eq:dualityXeinfty},
\[\frac{1}{t}\mathbb{E}_\infty[e^{-X_t^{\mathrm{e}\infty}y}]=\frac{1}{t}\mathbb{P}^{y}(Y^{\mathrm{m}}_t=0).\]
We show that $\frac{1}{t}\mathbb{P}^{y}(Y^{\mathrm{m}}_t=0)\underset{t\rightarrow 0^{+}}{\longrightarrow} 0.$
This can be easily seen by noticing that $Y^{\mathrm{m}}$ has no negative jumps, and in particular, there is no single jump from $y$ to $0$, see the stochastic equation \eqref{SDECBDIY}. We establish it rigorously. Let $f$ be a positive $\mathrm{C}^2$-function with compact support included in $[0,y/2]$ and such that $f_{\lvert [0,\epsilon]}\geq 1$ for some $\epsilon<y/2$. Then, 
\begin{equation}\label{eq:nonnegativejumps} f(y)+\mathbb{E}^{y}\left[\int_0^{t}\mathcal{Y}f(Y_s^{\mathrm{m}})\ddr s\right]=\mathbb{E}^{y}[f(Y_t^{\mathrm{m}})]\geq \mathbb{P}^{y}\left(Y_t^{\mathrm{m}}\leq \epsilon\right),\end{equation}
and since $f(y)=0$, $f'(y)=f''(y)=0$ and $f(y+h)=0$ for all $h>0$,
$$\frac{1}{t}\int_0^{t}\mathbb{E}^{y}[\mathcal{Y}f(Y_s^{\mathrm{m}})]\ddr s \underset{t\rightarrow 0}{\longrightarrow} \mathcal{Y}f(y)=0.$$ 
Then, dividing by $t$ each side of \eqref{eq:nonnegativejumps}, we get that $\frac{1}{t}\mathbb{P}^{y}\left(Y_t^{\mathrm{m}}\leq \epsilon\right)\underset{t\rightarrow 0+}{\longrightarrow} 0.$ 
	It remains to explain that there is no jump from $\infty$.  By \cite[Theorem 4.2]{foucartvidmar2025}, Courr\`ege form of the generator $\mathcal{X}^{\mathrm{e}\infty}$, the dynamics from $\infty$ are encoded through a finite jump measure $\nu_\infty$ on $[0,\infty)$ as follows:
	\begin{equation}\label{eq:generatoratinfinity}\mathcal{X}^{\mathrm{e}\infty}\mathrm{e}^{y}(\infty)=\int_{[0,\infty)}e^{-uy}\nu_\infty(\ddr u),\end{equation}
 see \cite[Equation (4.5)]{foucartvidmar2025} (with in the notation therein $k_\infty=\nu_\infty(\{0\})$). The latter quantity being identically zero,  $\nu_\infty\equiv 0$ and there is no jump from $\infty$.
 \end{proof}
 \begin{rem}\label{rem:nonegativejump} The fact that $X^{\mathrm{e}\infty}$ has a continuous boundary at $\infty$ can also be seen through the Skorokhod weak convergence, established in Lemma~\ref{prop:cvskorokhod}. Indeed, since there is no negative jump in the prelimiting processes $(X^{\mathrm{m},(n)}, n\geq 1)$, there is also none in $X^{\mathrm{e}\infty}$, see e.g. \cite[Proposition 2.1, p.337]{zbMATH01834045}.
\end{rem}

\end{enumerate}
The proof of Theorem~\ref{thm2infty} is achieved.
\end{proof}

 \begin{rem}\label{remark:infinicontinueparcvskorokhod} Any positive Markov process whose domain contains $\ee^y$ (\cite[Theorem~4.2]{foucartvidmar2025}) has a generator satisfying \eqref{eq:generatoratinfinity} at $\infty$ for some measure $\nu_\infty$. Consequently, either the generator vanishes at $\infty$ (as in our setting), or the dynamics at $\infty$ correspond to a holding point with parameter $\nu_\infty([0,\infty))$, followed by a return to $[0,\infty)$ according to the law $\frac{\nu_\infty(\cdot)}{\nu_\infty([0,\infty))}.$
 \end{rem}
\section{Absorption and non-absorption at $\infty$}\label{sec:absorptionatinfinity}

The primary aim of this section is to study whether the extended $\mathrm{CBDI}(\Psi,\hat\Psi)$ process $X^{\mathrm{e}\infty}$ has its boundary $\infty$ absorbing or not (Theorem~\ref{thm3infty}). The parameters $\thetaup$ and $\thetadown$, introduced in the introduction, will play a major role here. A consequence of the duality is that they  are also involved into the classification for the accessibility of $0$ (Theorem~\ref{cor:accessibility0bytheta}). Recall the decompositions, see Section~\ref{sec:LKfunction} \[\Psi=\Sigma-\Phi \text{ and }\hat{\Psi}=\hat\Sigma-\hat\Phi.\]

\subsection{The parameters $\thetaup$ and $\thetadown$}
Let $\hat{W}$ be the scale function of $\hat{\Sigma}$, see Section~\ref{sec:scalefunction},
and  define the $[0,\infty]$-valued parameters
\begin{equation}\label{eq:theta}\overline{\theta}_{\,\scriptscriptstyle\Phi, \hat{\Sigma}}:=\underset{x\rightarrow \infty}{\limsup} \, x\int_0^{\infty}\frac{\Phi(z)\hat{W}(z)}{z}e^{-zx}\ddr z,, \quad \underline{\theta}_{\,\scriptscriptstyle\Phi, \hat{\Sigma}}:=\underset{x\rightarrow \infty}{\liminf}\, x\int_0^{\infty}\frac{\Phi(z)\hat{W}(z)}{z}e^{-zx}\ddr z.
\end{equation}

\begin{theo}[Absorption/non-absorption of $X^{\mathrm{e}\infty}$ at $\infty$]\label{thm3infty} \

Assume 
$$\neg \hatHtwo:\ \int_1^{\infty}\frac{\ddr u}{\hat{\Sigma}(u)}<\infty  \text{ and } 	 \hatHone:\ \int_0^{1}\frac{\ddr u}{\hat\Phi(u)}=\infty.$$ 
\begin{enumerate}
    \item[i)] If $\overline{\theta}_{\,\scriptscriptstyle\Phi, \hat{\Sigma}}<1$ then $X^{\mathrm{e}\infty}$ has $\infty$ non-absorbing and instantaneous.
\item[ii)] If $\underline{\theta}_{\,\scriptscriptstyle\Phi, \hat{\Sigma}}>1$ then $X^{\mathrm{m}}$ admits no non-trivial continuous Feller extension at $\infty$, that is, any such extension must be absorbed at $\infty$ after its first explosion time. In particular, $X^{\mathrm{e}\infty}$ has $\infty$ absorbing.
\end{enumerate}
\end{theo}
Theorem~\ref{thm3infty} is established in Section~\ref{sec:prooftheorem3infinity}. We shall find conditions for $\infty$ to be (in)-accessible later on, see Theorem~\ref{thm:accessibilityinftybyrho}.
\medskip

A consequence of Theorem~\ref{thm3infty} together with the Laplace duality \eqref{eq:dualityXeinfty} is the following sufficient conditions for (in)-accessibility of the boundary $0$ for the minimal $\mathrm{CBDI}(\Psi,\hat{\Psi})$-process $X^{\mathrm{m}}$. Notice that the conditions below involve the \textit{dual} parameters $\thetadowndual$ and $\thetaupdual$ with in \eqref{eq:theta}, $\hat\Phi$ instead of $\Phi$ and $W$ the scale function of $\Sigma$ instead of $\hat{W}$.
\begin{theo}[Accessibility/inaccessibility of $0$ for  $X^{\mathrm{m}}$]\label{cor:accessibility0bytheta} Assume 
$$\neg \Htwo:\ \int_1^{\infty}\frac{\ddr u}{\Sigma(u)}<\infty  \text{ and } 	 \Hone: \ \int_0^1\frac{\ddr u}{\Phi(u)}=\infty.$$ 
	\begin{enumerate}
		\item[i)] If $\overline{\theta}_{\,\scriptscriptstyle \hat{\Phi},\Sigma}<1$, then $X^{\mathrm{m}}$ has $0$ accessible.
		\item[ii)] If $\underline{\theta}_{\scriptscriptstyle \,\hat{\Phi},\Sigma}>1$ then $X^{\mathrm{m}}$ has $0$ inaccessible. 
	\end{enumerate}
		\end{theo}
\begin{proof}
	The assumptions allow us to apply Theorem \ref{thm2infty} to the dual minimal $\mathrm{CBDI}(\hat{\Psi},\Psi)$ process, $Y^{\mathrm{m}}$, and we get the duality relationship, under the convention $(0\cdot \infty^-,\infty\cdot 0^+)$  
	$$\mathbb{E}^{y}[e^{-xY^{\mathrm{e}\infty}_t}]=\mathbb{E}_{x}[e^{-X_t^{\mathrm{m}}y}], \  x,y\in [0,\infty],\ t\in [0,\infty).$$
	In particular, for all $x\in (0,\infty)$, with the convention $0\cdot \infty^{-}$:
	\begin{equation}\label{eq:keyforaccessibility0}\mathbb{E}^{\infty}[e^{-xY_t^{\mathrm{e}\infty} }]=\mathbb{P}_x(X_t^{\mathrm{m}}=0).\end{equation}
	By Theorem \ref{thm3infty}, if $\thetaupdual<1$ then $\infty$ is non-absorbing for $Y^{\mathrm{e}\infty}$ and we see therefore from \eqref{eq:keyforaccessibility0}, that $0$ is accessible for $X^{\mathrm{m}}$. If $\thetadowndual>1$, then $\infty$ is absorbing and $0$ is inaccessible. 
\end{proof}
We summarize the correspondence between accessibility at $0$ and non-absorption at $\infty$ in Table~\ref{correspondance2}. 

\begin{table}[htpb]
    \begin{center}
		\begin{tabular}{|c|c|}
			\hline
			$\mathrm{CBDI}(\Psi,\hat{\Psi})$ &  $\mathrm{CBDI}(\hat{\Psi},\Psi)$ \\
			\hline
			$0$ accessible & $\infty$  non-absorbing   \\
			\hline
		\end{tabular}
		\vspace*{1mm}
		\caption{Accessibility of $0$ for $X^{\mathrm{m}}$ and non-absorption at $\infty$ for $Y^{\mathrm{e}\infty}$}
		\label{correspondance2}
	\end{center}
   
\end{table}

The next lemma provides basic analytical properties of the parameters $\thetaup$ and $\thetadown$ which turn to be useful when handling examples. 
\medskip

Recall $\pi$ the L\'evy measure of $\Psi$, $\bar\pi$ its tail and $\lambda=-\Psi(0)=\Phi(0)$.

\begin{lem}\label{lem:boundstheta} Assume $\neg \Htwo:\ \int_1^{\infty}\frac{\ddr u}{\Sigma(u)}<\infty$.
\begin{enumerate}
\item Let $\Phi$ be the Bernstein function with triplet $(\gamma^+, \pi_{|[1,\infty)}, \lambda)$. One has \begin{equation}\label{eq:thetawithpi} \underline{\theta}_{\,\scriptscriptstyle\Phi, \hat{\Sigma}}=\underset{x\rightarrow \infty}{\liminf}\, x\int_1^{\infty}\frac{\bar{\pi}(h)+\lambda}{\hat{\Sigma}(x+h)}\ddr h,\end{equation}
and similarly for $\overline{\theta}_{\,\scriptscriptstyle\Phi, \hat{\Sigma}}$ replacing $\liminf$ by $\limsup$.
\item
If $\Phi\sim \Phi_1$ at $0$, then $$\overline{\theta}_{\,\scriptscriptstyle\Phi, \hat{\Sigma}}=\overline{\theta}_{\,\scriptscriptstyle\Phi_1, \hat{\Sigma}},\text{ and } \underline{\theta}_{\,\scriptscriptstyle\Phi, \hat{\Sigma}}=\underline{\theta}_{\,\scriptscriptstyle\Phi_1, \hat{\Sigma}}.$$ 
\item One has
\[\thetadown=\underset{x\rightarrow \infty}{\liminf}\,\mathbb{E}[A(\mathbbm{e}_x)], \quad \thetaup=\underset{x\rightarrow \infty}{\limsup}\,\mathbb{E}[A(\mathbbm{e}_x)],\]
where $$A:(0,\infty)\ni z\mapsto\frac{\Phi(z)\hat{W}(z)}{z} \text{ and } \mathbbm{e}_x \text{ is an exponential r.v. with parameter }x.$$ The following hold: $A(z)\asymp \frac{\Phi(z)}{z^2\hat{\Sigma}(1/z)}$ and
\[\liminf_{z\rightarrow 0}A(z)\leq \underline{\theta}_{\,\scriptscriptstyle\Phi, \hat{\Sigma}}\leq \overline{\theta}_{\,\scriptscriptstyle\Phi, \hat{\Sigma}}\leq \limsup_{z\rightarrow 0} A(z).\]    
In particular if $\theta:=\underset{z\rightarrow 0}{\lim}\frac{\Phi(z)\hat{W}(z)}{z}\text{ exists in } [0,\infty]$, then $$\theta=\underline{\theta}_{\,\scriptscriptstyle\Phi, \hat{\Sigma}}=\overline{\theta}_{\,\scriptscriptstyle\Phi, \hat{\Sigma}}.$$
\end{enumerate}
\end{lem}
As explained in the introduction, the parameters $\thetaup, \thetadown$, written under the form \eqref{eq:thetawithpi}, also arise in a distinct context \cite[Theorem 1.1]{zbMATH07493833}. The proof of Lemma~\ref{lem:boundstheta} follows from simple arguments and is deffered to Section~\ref{sec:appendix2}.
\medskip

The next proposition sheds light on cases where the process can jump to $\infty$, $\lambda>0$ and has a quadratic competition term $\hat{a}>0$ in its drift. This generalizes the logistic case \cite{MR3940763} for which $\hat{\Psi}(x):=\hat{\mathrm{a}}x^2, \ x\in [0,\infty)$ and will play a significant role later on when constructing the extension at $0$. 
\begin{prop}\label{prop:simplecasetheta} In each of the following cases, one has $\overline{\theta}_{\,\scriptscriptstyle\Phi, \hat{\Sigma}}=\underline{\theta}_{\,\scriptscriptstyle\Phi, \hat{\Sigma}}$, and we denote by $\theta$ this common value:
    \begin{enumerate}
        \item If $\Phi(0)=\lambda>0$ then $\theta=\begin{cases} &\!\!\!\!\!\frac{\lambda}{\hat a} \text{ if } \hat{a}>0\\
        &\!\!\!\!\!\infty\ \text{ if } \hat{a}=0. \end{cases}  $
        \item If $\Phi(0)=0$ and $\hat{a}>0$ then $\theta=0$.
    \end{enumerate}
\end{prop}

\begin{rem} 
	The regimes $\theta>1, \theta<1$ match with those of discrete analogue exchangeable fragmentation-coalescence processes studied in Kyprianou et al. \cite{kyprianou2017}, see also \cite[Corollaries 1.2, 1.4]{zbMATH07493833}.
\end{rem}
\begin{proof}
By Lemma~\ref{lem:boundstheta}, one has
\begin{equation}\label{eq:decompo}\underline{\theta}_{\,\scriptscriptstyle\Phi, \hat{\Sigma}}=\lambda \underset{x\rightarrow \infty}{\liminf}\, x\int_1^\infty\frac{\ddr h}{\hat{\Sigma}(x+h)} +\underset{x\rightarrow \infty}{\liminf}\,x\int_1^{\infty}\frac{\bar{\pi}(h)}{\hat{\Sigma}(x+h)}\ddr h.\end{equation}
Assume $\hat{\mathrm{a}}>0$. For the first integral term, using that $\hat{\Sigma}(x)\underset{x\rightarrow \infty}{\sim} \hat{\mathrm{a}}x^2$, we see that \[\int_1^\infty\frac{\ddr h}{\hat{\Sigma}(x+h)}=\int_{x+1}^\infty\frac{\ddr h}{\hat{\Sigma}(h)}\underset{x\rightarrow \infty}{\sim }\frac{1}{\hat{\mathrm{a}}x}.\]
The first limit is therefore $\lambda/\hat{\mathrm{a}}$. For the second limit, we write
\begin{align*}x\int_1^{\infty}\frac{\bar{\pi}(h)}{\hat{\Sigma}(x+h)}\ddr h&=x\int_{1}^{x}\frac{\bar{\pi}(h)}{\hat{\Sigma}(x+h)}\ddr h+x\int_{x}^{\infty}\frac{\bar{\pi}(h)}{\hat{\Sigma}(x+h)}\ddr h\\
&\leq \frac{x}{\hat{\Sigma}(x)}\int_{1}^{x}\bar{\pi}(h)\ddr h+x\int_{x}^{\infty}\frac{\bar{\pi}(h)}{\hat{\Sigma}(h)}\ddr h.
\end{align*}
Notice that $\bar{\pi}(h) \rightarrow 0$ as $h\rightarrow \infty$ and $\hat{\Sigma}(x)\geq \hat{\mathrm{a}}x^2$ for all $x\in [0,\infty)$. Thus,  
$$\frac{x}{\hat{\Sigma}(x)}\int_{1}^{x}\bar{\pi}(h)\ddr h\leq \frac{1}{\hat{\mathrm{a}}x}\int_1^x\bar{\pi}(h)\ddr h \underset{x\rightarrow \infty}{\rightarrow} 0.$$
For the second term, we have
$$x\int_{x}^{\infty}\frac{\bar{\pi}(h)}{\hat{\Sigma}(h)}\ddr h\leq x\bar{\pi}(x)\int_x^\infty\frac{\ddr h}{\hat{\Sigma}(h)}\leq \frac{1}{\hat{\mathrm{a}}}\bar{\pi}(x) \underset{x\rightarrow \infty}{\rightarrow} 0.$$
Thus $\thetadown=\lambda/\hat{\mathrm{a}}$. The $\limsup$ is treated similarly and yields also $\thetaup=\lambda/\hat{\mathrm{a}}$.
\smallskip

In the case $\hat{\mathrm{a}}=0$, $\hat{\Sigma}(x)/x^2\rightarrow 0$ as $x\to \infty$ and when $\lambda>0$, one has $\lambda x\int_{x+1}^\infty\frac{\ddr h}{\hat{\Sigma}(h)} \underset{x\rightarrow \infty}{\longrightarrow}\infty$, entailing that the first term in \eqref{eq:decompo} is infinite and therefore $\thetadown=\infty$. The second statement with $\lambda=0$ has been shown while dealing with the second term of \eqref{eq:decompo}.
\end{proof}
\subsection{Proof of Theorem \ref{thm3infty}}\label{sec:prooftheorem3infinity}
Let $\Psi, \hat{\Psi}$ be two L\'evy-Khintchine functions of the form \eqref{branchingmechanism}. We decompose them as follows $\Psi=\Sigma-\Phi$ and $\hat{\Psi}=\hat{\Sigma}-\hat{\Phi}$. Consider a minimal $\mathrm{CBDI}(\Psi,\hat\Psi)$, $X^{\mathrm{m}}$. Recall $\mathcal{X}$ its generator:
\begin{equation}\label{eq:genCBDIX}\mathcal{X}f(x)=x\mathrm{L}^{\Psi}f(x)-\hat{\Psi}(x)f'(x),\ x\in \mathcal{D}_f.\end{equation}
We work in all this section under the assumption $\int_1^\infty \frac{\ddr u}{\hat{\Psi}(u)}<\infty$. Recall that this is is equivalent to 
$\hat{\Sigma}\not\equiv 0\text{ and }\int_1^{\infty}\frac{\ddr u}{\hat{\Sigma}(u)}<\infty$, see Lemma~\ref{lem:equivpsisigmagrey}. It is also equivalent to $\int_0^1\frac{\hat{W}(x)}{x}\ddr x<\infty$, see Section~\ref{sec:scalefunction}. The following function is thus well-defined
\begin{equation}\label{def:F}
	F(x):=\int_0^{\infty}\frac{\Phi(z)\hat{W}(z)}{z}e^{-zx}\ddr z, \quad x\in (0,\infty).
\end{equation}
Recall \eqref{eq:theta}. We have
\begin{equation}\label{eq:thetaproof}\overline{\theta}_{\,\scriptscriptstyle\Phi, \hat{\Sigma}}=\underset{x\rightarrow \infty}{\limsup} \, xF(x), \quad \underline{\theta}_{\,\scriptscriptstyle\Phi, \hat{\Sigma}}=\underset{x\rightarrow \infty}{\liminf}\, xF(x).
\end{equation}
The proof of Theorem \ref{thm3infty} is made along the following steps. In Section \ref{subsec:Lyapunovf}, we build a certain Lyapunov function $f$ for the generator $\mathcal{X}$.  We then use it in Section \ref{subsec:nonabsorptionatinfinity} to provide some bounds on the first  entrance times of $X^{\mathrm{m}}$, see Lemma \ref{lem:firstentranceXmin}. Next we establish that the extended process $X^{\mathrm{e}\infty}$ has $\infty$ non-absorbing when $\thetaup<1$. In Section \ref{subsec:absorptionatinfinity}, we use this function in order to show that when $\thetaup>1$, $\infty$ is absorbing for $X^{\mathrm{e}\infty}$. More precisely, we establish that there exists no continuous Fellerian extension at $\infty$, with $\infty$ non-absorbing, for the minimal process, see Lemma \ref{lem:proofthm3ii}.

 We mention that the arguments in this section do not make use on the fact that the process $X^{\mathrm{e}\infty}$ has $0$ as an absorbing state. Instead, together with the Feller property, the key ingredient when dealing with $\thetaup < 1$ is the monotone convergence property $\underset{n\rightarrow \infty}{\lim} \uparrow X_t^{\mathrm{e}\infty,(n)}=X_t^{\mathrm{e}\infty},\ \forall t\geq 0 \text{ a.s.}$ For $\thetadown>1$, the continuity of the boundary will play a central role. 
 
\subsubsection{A Lyapunov function}\label{subsec:Lyapunovf}
The first lemma studies the action of $\mathcal{X}$ on a specific function $f$, which will turn to be key in the proof of Theorem \ref{thm3infty}.
\begin{lem}\label{lem:lyapunovf} Assume $\neg \hatHtwo:\ \int_1^\infty\frac{\ddr u}{\hat\Sigma(u)}<\infty$  and define the decreasing bounded $\mathrm{C}^2$-function \begin{equation}\label{def:f}f: [1,\infty)\ni x \mapsto f(x):=\int_x^{\infty}\frac{\ddr u}{\hat\Sigma(u)}.\end{equation} The function $f$  can be rewritten as $$f(x)=\int_0^{\infty}e^{-xz}\frac{\hat{W}(z)}{z}\ddr z, \ x\in [1,\infty)$$
and one has
\begin{equation}\label{eq:keyLyapunov} \mathcal{X}f(x)=1-xF(x)-\epsilon(x), \ x\in [1,\infty),
\end{equation}
where $F$ is given by \eqref{def:F} and $\epsilon$ is vanishing at $\infty$.
\end{lem}
\begin{rem} We shall see in the proof that if $\Psi=-\Phi$ and $\hat{\Psi}=\hat{\Sigma}$, that is to say when the branching is immortal and the interaction is pure competition, then $\epsilon\equiv 0$ in \eqref{eq:keyLyapunov}.
\end{rem}
\begin{proof} We start with the setting $\Psi=-\Phi$ and $\hat{\Psi}=\hat{\Sigma}$. Denote by $(\gamma^+,\nu,\lambda)$ the triplet of $\Phi$: \[\Phi(x)=\gamma^{+} x+\int_0^{\infty}(1-e^{-xh})\nu(\ddr h)+\lambda.\]
    We have 
\begin{align*}    
    \mathcal{X}f(x)&=x\mathrm{L}^{-\Phi}f(x)-\hat\Sigma(x)f'(x)\\
    &= \gamma^{+} xf'(x)+x\int_0^{\infty}(f(x+h)-f(x))\nu(\ddr h)+\lambda x\big(f(\infty)-f(x)\big)-\hat\Sigma(x)f'(x).
    \end{align*}
By the definition \eqref{def:f}, one has $-\hat\Sigma(x)f'(x)=1$ for all $x\in [1,\infty)$.   Recall from Section \ref{sec:scalefunction} that \begin{equation}\label{scalefunctionid}\frac{1}{\hat\Sigma(u)}=\int_0^{\infty}e^{-uz}\hat W(z)\ddr z, \quad u\in (0,\infty).
\end{equation}
Let $x\in[1,\infty)$ be fixed. One has
\begin{equation}\label{eq:driftpart}
\gamma^{+} f'(x)=-\frac{\gamma^{+} }{\hat{\Sigma}(x)}=-\int_{0}^{\infty}\gamma^{+}  e^{-xz}\hat{W}(z)\ddr z=-\int_{0}^{\infty}\frac{\gamma^{+}z}{z} e^{-xz}\hat{W}(z)\ddr z.
\end{equation}
For all $h\in [0,\infty)$, by  \eqref{scalefunctionid} and Fubini-Tonelli's theorem
    \begin{align}
        f(x+h)-f(x)=-\int_x^{x+h}\frac{\ddr u}{\hat\Sigma(u)}&=-\int_0^h\int_0^{\infty}e^{-z(u+x)}\hat W(z)\ddr u \ddr z \nonumber \\
        &=-\int_0^{\infty}\left(\int_0^he^{-zu}\ddr u\right)e^{-zx}\hat W(z)\ddr z \nonumber \\
        &=-\int_0^{\infty}\frac{1-e^{-zh}}{z}e^{-zx}\hat W(z)\ddr z.    \end{align}
        Thus, \begin{equation}\label{eq:jumppart}
        \int_0^{\infty}\big(f(x+h)-f(x)\big)\nu(\ddr h)=-\int_0^{\infty}\frac{1}{z}\left(\int_0^\infty(1-e^{-zh})\nu(\ddr h)\right)e^{-zx}\hat W(z)\ddr z.        \end{equation}
Similarly, one can check that 
\begin{equation}
\label{eq:killingpart}
\lambda\left(f(\infty)-f(x)\right)=-\int_0^{\infty}\lambda\frac{e^{-zx}}{z}\hat W(z)\ddr z. 
\end{equation}
By combining the three terms \eqref{eq:driftpart}, \eqref{eq:jumppart}, \eqref{eq:killingpart}, we get
\[\mathrm{L}^{-\Phi}f(x)=-x\int_{0}^{\infty}\frac{1}{z}\left(\gamma^{+} z+\int_{0}^{\infty}(1-e^{-zh})\nu(\ddr h)+\lambda \right)e^{-xz}\hat{W}(z)\ddr z=- \int_{0}^{\infty}\frac{\Phi(z)}{z}\hat{W}(z)\ddr z.\]
Hence
\begin{equation}\label{purebirthcompetitionid}x\mathrm{L}^{-\Phi}f(x)-\hat{\Sigma}(x)f'(x)=1-xF(x).
\end{equation}
We now consider the general setting for which $\Psi=\Sigma-\Phi$ and $\hat{\Psi}=\hat{\Sigma}-\hat{\Phi}$. Denote by $(\gamma^{-}, a, \eta)$ the triplet associated to $\Sigma$, see Section \ref{sec:LKfunction}. Since $\mathrm{L}^{\Psi}=\mathrm{L}^{\Sigma}-\mathrm{L}^{-\Phi}$, one gets from \eqref{purebirthcompetitionid}
 \begin{align*}
        \mathcal{X}f(x)=1-xF(x)+x\mathrm{L}^{\Sigma}f(x)+\hat\Phi(x)f'(x), \ x\in [1,\infty).
    \end{align*}
Only remains to show that $\epsilon(x):=x\mathrm{L}^{\Sigma}f(x)+\hat\Phi(x)f'(x) \rightarrow 0$ as $x\rightarrow \infty$.  One has
\[\hat\Phi(x)f'(x)=-\frac{\hat{\Phi}(x)}{\hat\Sigma(x)}=-\frac{\hat{\Phi}(x)}{x}\frac{x}{\hat\Sigma(x)}.\]
Notice that the following limits always exist 
$$\underset{x\rightarrow \infty}{\lim}\frac{\hat{\Phi}(x)}{x} \in [0,\infty) \text{ and }\underset{x\rightarrow \infty}{\lim}\frac{x}{\hat\Sigma(x)}\in [0,\infty].$$ 
By assumption, $\int^{\infty}\frac{\ddr x}{\hat{\Sigma}(x)}<\infty$ therefore, $\underset{x\rightarrow \infty}{\lim}\frac{x}{\hat\Sigma(x)}=0$ and $\hat\Phi(x)f'(x)\underset{x\rightarrow \infty}{\longrightarrow} 0.$
We study now $x\mathrm{L}^{\Sigma}f(x)$.  Recall that $\int_0^{\infty} h\wedge h^2\eta(\ddr h)<\infty$.
    \begin{align*}
        x\mathrm{L}^{\Sigma}f(x)&=\gamma^{-}xf'(x)+axf''(x)+x\int_0^\infty(f(x+h)-f(x)-hf'(x))\eta(\ddr h)\\
        &=-\gamma^{-}\frac{x}{\hat\Sigma(x)}+a\frac{\hat\Sigma'(x)}{x}\left(\frac{x}{\hat\Sigma(x)}\right)^2-x\int_0^\infty\left(\frac{1}{\hat\Sigma(x)}+\frac{1}{\hat\Sigma(u+x)}\right)\bar{\eta}(u)\ddr u,
    \end{align*}
with $\bar{\eta}(u)=\eta([u,\infty)$. As previously seen, the first term vanishes. Notice that $\hat{\Sigma}'$ is a Bernstein function, see Section \ref{sec:LKfunction}, hence $[1,\infty)\ni x\mapsto \hat{\Sigma}'(x)/x$ admits a finite limit, thus the second term also vanishes. For the integral term, we argue with Lebesgue's theorem. By the mean value theorem, for any $x\in [1,\infty)$ and $u\in [0,1)$, there exists $w\in [0,u]$ such that
\[0\leq x\left(\frac{1}{\hat\Sigma(x)}-\frac{1}{\hat\Sigma(u+x)}\right)=
 \left|\left(1/\hat{\Sigma}(x+w)\right)'\right| u.\]
For all $x\in [1,\infty)$,
\begin{align*}
x|\left(1/\hat{\Sigma}(x+w)\right)'|=\frac{x\Sigma'(x+w)}{\Sigma(x+w)^2}&=\frac{\hat\Sigma'(x+w)}{x+w}\left(\frac{x+w}{\hat\Sigma(x+w)}\right)^2\frac{x}{x+w}.
\end{align*}
The functions $x\mapsto \frac{\hat\Sigma'(x)}{x}$ and $x\mapsto \frac{x}{\hat\Sigma(x)}$ are decreasing and $\frac{x}{x+w}\in (0,1]$ thus each of the three factors above is bounded near $\infty$ and there exists $C_0<\infty$ such that
\[x\left(\frac{1}{\hat\Sigma(x)}-\frac{1}{\hat\Sigma(u+x)}\right)\leq C_0u.\]
If $u\in [1,\infty)$, then by setting $C_1=\frac{1}{\Sigma(1)}$, one has
\[0\leq \frac{1}{\hat\Sigma(x)}-\frac{1}{\hat\Sigma(u+x)}\leq C_1.\]
Therefore, for all $x\in [1,\infty)$, all $u\in [0,\infty)$, \[\left(\frac{1}{\hat\Sigma(x)}-\frac{1}{\hat\Sigma(u+x)}\right)\bar{\eta}(u)\leq \left(C_0 u \wedge C_1 \right)\bar{\eta}(u).\]
The right-hand side is integrable. The left-hand-side vanishes as $x$ goes to $\infty$. We conclude by Lebesgue's theorem that
\[x\int_0^\infty\left(\frac{1}{\hat\Sigma(x)}-\frac{1}{\hat\Sigma(u+x)}\right)\bar{\eta}(u)\ddr u \underset{x\rightarrow \infty}{\longrightarrow} 0.\]
The proof is achieved.
\end{proof}
\subsubsection{First entrance times and $\infty$ non-absorbing}\label{subsec:nonabsorptionatinfinity}
Recall the notation of the first passage times. For all $a\in [0,\infty)$, $\sigma_a^{-}=\inf\{t\geq 0: X^{\mathrm{m}}_t\leq a\}$.  Recall that $X^{\mathrm{m}}$ has no negative jump, therefore $\sigma_a^{-}=\inf\{t\geq 0: X^{\mathrm{m}}_t=a\}$ a.s. and $X^{\mathrm{m}}_{\sigma_a^{-}}=a$ a.s. on $\{\sigma_a^{-}<\infty\}$. For all $b\in [0,\infty)$, $\sigma_b^{+}=\inf\{t\geq 0: X^{\mathrm{m}}_t\geq b\}$ and $\sigma_\infty^{+}$ is the (first) explosion time of $X^{\mathrm{m}}$. Notice $\sigma_\infty^{+}=\underset{b\rightarrow \infty}{\lim} \uparrow \sigma_b^{+}$ a.s.  

We work in all the section under the assumption of Theorem~\ref{thm2infty}. Recall in particular the assumption $\neg \hatHtwo:\ \int_1^{\infty}\frac{\ddr u}{\hat{\Sigma}(u)}<\infty$.

\begin{lem}\label{lem:firstentranceXmin}
    Assume $\thetaup<1$.    Then, there exists $x_0\in (0,\infty)$ such that for any $a\in (x_0,\infty)$:  \[\underset{x\geq x_0}{\sup} \, \mathbb{E}_x[\sigma_a^{-}\wedge \sigma_{\infty}^{+}]\leq \frac{2}{1-\overline{\theta}_{\,\scriptscriptstyle\Phi, \hat{\Sigma}}}\int_{a}^{\infty}\frac{\ddr u}{\hat{\Sigma}(u)}.\]
\end{lem}
\begin{proof}
To alleviate the notation, we write here $\overline{\theta}_{\,\scriptscriptstyle\Phi, \hat{\Sigma}}=\overline{\theta}$. Recall the $\mathrm{C}^2_b$-function $f$ in \eqref{def:f}. By Lemma \ref{lem:lyapunovf} and the identity \eqref{eq:keyLyapunov}, we see that
\[\underset{x\rightarrow \infty}{\liminf}\,\mathcal{X}f(x)=1-\overline{\theta}>0.\]
Therefore, there exists $x_0>1$ such that for all $x\geq x_0$, $\mathcal{X}f(x)\geq \frac{1-\overline{\theta}}{2}$. By Theorem \ref{thm:minX} and the optional stopping theorem, the process $M=(M_t)_{t\geq 0}$ defined for all $t\geq 0$ as follows
\[M_t:=f\big(X^{\mathrm{m}}_{\sigma_a^{-}\wedge\sigma_b^{+}\wedge t}\big)-\int_{0}^{\sigma_a^{-}\wedge\sigma_b^{+}\wedge t}\mathcal{X}f(X^{\mathrm{m}}_{s})\ddr s,\]
is a local martingale. By Lemma \ref{lem:lyapunovf}, $f$ and $\mathcal{X}f$ are bounded on $[a,\infty)$, $M$ is therefore a martingale and for all $x\geq a$,
\[\mathbb{E}_x\left[f\left(X^{\mathrm{m}}_{t\wedge\sigma_a^{-}\wedge \sigma_b^{+}}\right)\right]-\mathbb{E}_x\left[\int_0^{t\wedge\sigma_a^{-}\wedge \sigma_b^{+}}\mathcal{X}f(X^{\mathrm{m}}_s)\ddr s\right]=f(x).\]
Since $\mathcal{X}f(x)\geq \frac{1-\overline{\theta}}{2}$ for any $x\geq a$ and $f$ is non-decreasing, 
\begin{align*}
f(x)\leq \mathbb{E}_x\left[f\left(X^{\mathrm{m}}_{t\wedge\sigma_a^{-}\wedge \sigma_b^{+}}\right)\right]+\frac{1-\overline{\theta}}{2}\mathbb{E}_x[t\wedge\sigma_a^{-}\wedge \sigma_b^{+}]\leq f(a)+\frac{1-\overline{\theta}}{2}\mathbb{E}_x[t\wedge\sigma_a^{-}\wedge \sigma_b^{+}].
\end{align*}
By letting $t$ go to $\infty$ and using that $f(x)\leq f(\infty)=\int_1^{\infty}\frac{\ddr u}{\hat{\Sigma}(u)}$, we obtain
\[\mathbb{E}_x\left[\sigma_a^{-}\wedge \sigma_b^{+}\right]\leq \frac{2}{1-\overline{\theta}}\big(f(a)-f(x)\big)\leq \frac{2}{1-\overline{\theta}} \int_{a}^{\infty}\frac{\ddr u}{\hat{\Sigma}(u)}.\]
Last, with $b\rightarrow\infty$, we get by monotone convergence,  \[\mathbb{E}_{x}\left[ \sigma_a^{-}\wedge \sigma_{\infty}^{+}\right]\leq \frac{2}{1-\overline{\theta}} \int_{a}^{\infty}\frac{\ddr u}{\hat{\Sigma}(u)}.\]
Observe that the right-hand side does not depend on the starting value $x$.
\end{proof}

A first consequence of Lemma \ref{lem:firstentranceXmin} is that when $\Phi'(0+)<\infty$ (i.e. $\pi$ the L\'evy measure in $\Psi$ has a finite first moment), the process $X^{\mathrm{e}\infty}$ has its boundary $\infty$ entrance. We explained this already in Proposition~\ref{cor:explicitcondforXe}, we provide here another proof.
\begin{lem}\label{lem:firstmean} Assume that $\Phi'(0+)<\infty$ then the process $X^{\mathrm{m}}$ does not explode, $\thetaup=0$ and the extended process $X^{\mathrm{e}\infty}$ defined in \eqref{def:Xeinftyviamonotonicity} has its boundary $\infty$ instantaneous entrance.
\end{lem}
\begin{rem} Lemma \ref{lem:firstmean} ensures that all processes $X^{\mathrm{e}\infty,(n)}$, whose branching measures $\pi_n$ have bounded support $(0,n]$ and thus finite mean,  have their boundary $\infty$ instantaneous entrance. 
\end{rem}
\begin{proof} The fact that $X^{\mathrm{m}}$ does not explode is a direct consequence of Proposition \ref{prop:suffcondinaccessibility}, indeed $\Phi'(0+)<\infty$ entails that $\int_{0}^1\frac{\ddr u}{\Phi(u)}=\infty$. From Lemma \ref{lem:boundstheta} and the fact that $W(0+)=W(0)=0$, see Section \ref{sec:scalefunction}, we have $\thetaup\leq \Phi'(0+)W(0)=0$. By applying Lemma \ref{lem:firstentranceXmin}, we see that $$\sup_{x\in (a,\infty)}\mathbb{E}_x[\sigma_a^{-}]\leq 2\int_{a}^{\infty}\frac{\ddr u}{\hat{\Sigma}(u)}.$$
We now apply \cite[Theorem 2.2]{zbMATH07242423}. This ensures that the Feller process $X^{\mathrm{e}\infty}$ has its boundary $\infty$ instantaneous entrance. Moreover, the comparison property stipulating that  for all $x\leq y$, one has $X_t^{\mathrm{m}}(x) \leq  X_t^{\mathrm{m}}(y)\leq X^{\mathrm{e}\infty}_t$ for all $t\in [0,\infty)$ almost surely, entails 
that
$$\sigma_a^{\mathrm{e}\infty,-}:=\inf\{t\geq 0: X^{\mathrm{e}\infty}_t\leq a\}\geq \lim_{x\rightarrow \infty}\uparrow \sigma_a^{-}(x), \ \mathbb{P}_\infty\text{-a.s.}.$$ 
\cite[Proposition~2.4-(b)]{zbMATH07242423} ensures that $$\lim_{x\rightarrow \infty}\uparrow \mathbb{E}_x[\sigma^{\mathrm{e}\infty,-}_a]=\mathbb{E}_\infty[\sigma_a^{\mathrm{e}\infty,-}]$$ which allows us to conclude that $\sigma_a^{\mathrm{e}\infty,-}=\lim_{x\rightarrow \infty}\!\uparrow \sigma_a^{-}(x), \ \mathbb{P}_\infty\text{-a.s.}.$
\end{proof}

We establish now Theorem \ref{thm3infty}-(i), namely we show that when $\thetaup<1$, the extended process $X^{\mathrm{e}\infty}$, defined as the limit of $X^{\mathrm{e}\infty,(n)}$, see Theorem \ref{thm2infty}, has its boundary $\infty$ non-absorbing. For all $a\geq 0$, let $\sigma_a^{\mathrm{e}\infty,-}:=\inf\{t>0: X_t^{\mathrm{e}\infty}\leq a\}$. We also sometimes emphasize on the initial value and write $\sigma_a^{\mathrm{e}\infty,-}(x)$ for the first passage time below $a$ of the extended process started from $x\in [0,\infty]$.
\medskip

Let us start by proving the convergence of the first entrance times of the truncated process towards that of $X^{\mathrm{e}\infty}$. For all $n\geq 1$, set $\sigma_a^{\mathrm{e},(n),-}:=\inf\{t\geq 0: \ X_t^{\mathrm{e}\infty,(n)}\leq a\}$.
\begin{lem}\label{lem:cvsigma_an}
On the event $\left\{\underset{n\rightarrow \infty}{\lim} \! \uparrow \sigma_a^{\mathrm{e},(n),-}<\infty\right\}$, it holds that \[\underset{n\rightarrow \infty}{\lim} \! \uparrow \sigma_a^{\mathrm{e},(n),-}=\sigma_a^{\mathrm{e}\infty,-}, \ \mathbb{P}_\infty\text{-a.s.}.\]
\end{lem}
\begin{proof}
        Recall that $X^{\mathrm{e}\infty,(n)}\leq X^{\mathrm{e}\infty, (n+1)}$ and $X^{\mathrm{e}\infty}:=\underset{n\rightarrow \infty}{\lim} \!\uparrow X^{\mathrm{e}\infty,(n)}$ a.s. This ensures that  \[\sigma:=\underset{n\rightarrow \infty}{\lim}\uparrow  \sigma_a^{\mathrm{e},(n),-}\leq \sigma_a^{\mathrm{e}\infty,-} \text{ a.s. }\]
         Since $X^{\mathrm{e}\infty}$ is a Feller process, it is quasi-left continuous, that is to say for any predictable time $T$, $X^{\mathrm{e}\infty}_{T-}=X^{\mathrm{e}\infty}_T$ a.s. Assume for a moment that $\sigma$ is predictable. 
        Since $X^{\mathrm{e}\infty,(n)}$ converges towards $X^{\mathrm{e}\infty}$ weakly in $\mathbbm{D}_{[0,\infty]}$, one has, see e.g. \cite[Proposition 2.1-(b.5), p.337]{zbMATH01834045},
        \[X^{\mathrm{e},(n)}_{\sigma_a^{\mathrm{e},(n),-}}\underset{n\rightarrow \infty}{\longrightarrow} X^{\mathrm{e}\infty}_{\sigma-} =X^{\mathrm{e}\infty}_{\sigma}.\]
        Plainly $X^{\mathrm{e},(n)}_{\sigma_a^{\mathrm{e},(n),-}}=a$, hence $X^{\mathrm{e}\infty}_{\sigma}=a$ a.s. which, by definition, entails $\sigma\geq \sigma_a^{\mathrm{e}\infty,-}$ a.s. 
        Only remains to establish that $\sigma$ is a predictable stopping time. Notice that the process $X^{\mathrm{e}\infty}$ is adapted to the filtration $(\mathcal{F}_t)_{t\geq 0}$ associated to the noises $B$ and $\mathcal{N}$ in the stochastic equation~\eqref{SDECBDI}. Moreover $(\sigma_a^{\mathrm{e},(n),-})_{n\geq 1}$ are $(\mathcal{F}_t)_{t\geq 0}$-stopping times. In order to show that $\sigma$ is predictable, we look for an increasing sequence of stopping times $(T_k)_{k\geq 1}$ such that 
        \[T_k<\sigma \text{ for all } k\geq 1 \text{ and } \underset{k\rightarrow \infty}{\lim}\!\!\uparrow T_k=\sigma \ \text{ a.s.} .\] 
        Plainly, by the absence of negative jumps, for all $k\leq n$,  \[\sigma^{(n,k)}:=\sigma^{\mathrm{e},(n),-}_{a+1/k}<\sigma_a^{\mathrm{e},(n),-}\leq \sigma \ \text{ and } \ \sigma^{(n,k+1)}>\sigma^{(n,k)}\ \text{ a.s.}.\]
        We deduce that for all $k\leq m$, $T_{k}:=\sup_{1\leq n\leq k}\sigma^{(n,k)}\leq T_{m}<\sigma$ and since a.s. $$\liminf_{k\to \infty}T_{k}\geq \liminf_{k\to \infty}\sigma^{\mathrm{e},(n),-}_{a+1/k}=\sigma^{\mathrm{e},(n),-}_{a},$$ 
        we have that $\underset{k\to \infty}{\liminf}\ T_{k}\geq \sigma$ a.s. which entails $\underset{k\to \infty}{\lim} T_{k}=\sigma$ almost surely.
        \end{proof}

\begin{lem}\label{lem:proofthm3i} There exists $x_0\in (0,\infty)$ such that for all $a\in [x_0,\infty)$, 
\begin{equation}\label{thebound}
\mathbb{E}_\infty[\sigma^{\mathrm{e}\infty,-}_a]\leq \frac{2}{1-\overline{\theta}_{\,\scriptscriptstyle\Phi, \hat{\Sigma}}}\int_a^{\infty}\frac{\ddr u}{\hat{\Sigma}(u)}.\end{equation}
In particular, $\infty$ is non-absorbing and instantaneous.
\end{lem}
\begin{proof}
We first study the first passage time below level $a$ for the prelimiting process $X^{\mathrm{e}\infty,(n)}$. Recall $\Psi^n=\Phi^n-\Sigma$ with $\Phi^n$ the Bernstein function with no killing and drift and jump measure given by
\begin{center}
$\gamma^{+}$ and $\pi_n(\ddr h)=\pi(\ddr h)\mathbbm{1}_{(1,n)}(h)+(\bar{\pi}(n)+\lambda)\delta_n$.
\end{center}
Recall that $\Phi^{n}\leq \Phi^{n+1}\leq \Phi$ and denote by $\mathcal{X}^{(n)}$ the generator of $X^{\mathrm{m},(n)}$. By Lemma \ref{lem:lyapunovf}, one has
\[\mathcal{X}^{(n)}f(x)=1-\int_{0}^{\infty}\frac{\Phi^n(z)\hat{W}(z)}{z}e^{-xz}\ddr z -\epsilon(x),\ x\in [1,\infty)\]
where $\epsilon$ only depends on $\Sigma$ and $\hat{\Psi}$, hence does not depend on $n$. We deduce that for all $x\in [1,\infty)$, $$\mathcal{X}^{(n)}f(x)\geq \mathcal{X}f(x).$$
Similarly as in the proof of Lemma \ref{lem:firstentranceXmin}, using that $X^{\mathrm{e}\infty,(n)}$ has the same law as $X^{\mathrm{m},(n)}$, when started from $x\in [0,\infty)$ and that it does not explode, we have for all $a\in (0,\infty), x\in [a,\infty)$ and $t\geq 0$,
\begin{align}
f(x)&=\mathbb{E}_x\left[f(X^{\mathrm{e}\infty,(n)}_{\sigma_a^{\mathrm{e},(n),-}\wedge t})-\int_{0}^{\sigma_a^{\mathrm{e},(n),-}\wedge t}\mathcal{X}^{(n)}f(X^{\mathrm{e}\infty,(n)}_s)\ddr s\right]\nonumber \\
&\leq f(a)-\mathbb{E}_x\left[\int_{0}^{\sigma_a^{\mathrm{e},(n),-}\wedge t}\mathcal{X}f(X^{\mathrm{e}\infty,(n)}_s)\ddr s\right]. \label{theboundwithn}
\end{align}
Under the assumption $\overline{\theta}_{\,\scriptscriptstyle\Phi, \hat{\Sigma}}<1$, there exists $x_0$ such that for all $x\geq x_0$, $$\mathcal{X}f(x)=1-xF(x)-\epsilon(x)\geq \frac{1-\overline{\theta}_{\Phi,\hat\Sigma}}{2}>0.$$
Therefore, by choosing $a\geq x_0$, we obtain from \eqref{theboundwithn}, 
\[\mathbb{E}_x[\sigma_a^{\mathrm{e},(n),-}\wedge t]\leq \frac{2}{1-\overline{\theta}_{\Phi,\hat\Sigma}}\big(f(a)-f(x)\big).\]
Now, letting $t$ and $x$ go to $\infty$ yields
\[\mathbb{E}_\infty[\sigma_a^{\mathrm{e},(n),-}]\leq \frac{2}{1-\overline{\theta}_{\Phi,\hat\Sigma}}\int_a^{\infty}\frac{\ddr u}{\hat{\Sigma(u)}}.\]
By letting $n$ go to $\infty$ and applying Lemma \ref{lem:cvsigma_an}, we deduce the bound \eqref{thebound}. In particular, for all $a\in [x_0,\infty)$,  $\sigma^{\mathrm{e}\infty,-}_a<\infty$, $\mathbb{P}_\infty$-a.s. so that $\infty$ is non-absorbing. The fact that it is instantaneous is established by letting $a$ go to $\infty$ similarly as in the proof of Lemma~\ref{lem:firstmean}.
\end{proof}
\subsubsection{Infinity absorbing}\label{subsec:absorptionatinfinity}
We establish now that if $\thetadown>1$ then the extended process $X^{\mathrm{e}\infty}$ has $\infty$ absorbing (accessibility will be studied later).  

We show that there exists no continuous Fellerian extension at $\infty$ of the minimal process $X^{\mathrm{m}}$. The argument will show that  no excursion measure continuous at $\infty$ exist. We design a supermartingale. Recall $f$ in \eqref{def:f}.
\begin{lem}\label{lem:supermartingale} 
Assume $\thetadown>1$, there exists $x_0\in (1,\infty)$ such that $(f(X^{\mathrm{m}}_{t\wedge \sigma_{x_0}^{-}}),t\ge 0)$ is a supermartingale.
\end{lem}
\begin{proof}
By definition, $\thetadown=\underset{x\rightarrow \infty}{\liminf }\, xF(x)\in [0,\infty]$. Thus, if $\thetadown>1$, there exists $x_0$ such that $\mathcal{X}f(x)\leq 0$ for all $x\geq x_0$.  By Itô's lemma, $$(M_t)_{t\ge 0}:=\left(f(X^{\mathrm{m}}_{t\wedge \sigma_{x_0}^{-}})-\int_{0}^{t\wedge \sigma_{x_0}^{-}}\mathcal{X}f(X^{\mathrm{m}}_s)\ddr s, \right)_{t\geq 0}$$
is a local martingale. The latter is positive, since for all $s\leq \sigma_{x_0}^{-}$, $\mathcal{X}f(X^{\mathrm{m}}_s)\geq 0$. This is therefore a supermartingale and the inequality, which holds for all $t\geq 0$, 
\[\mathbb{E}_x\big[f(X^{\mathrm{m}}_{t\wedge \sigma_{x_0}^{-}})\big]\leq \mathbb{E}_x[M_t]\leq \mathbb{E}_x[M_0]= f(x),\ x\in [x_0,\infty),\]
ensures that so is the process $\big(f(X^{\mathrm{m}}_{t\wedge \sigma_{x_0}^{-}}), t\geq 0\big)$.
\end{proof}
Let $\mathcal{E}_\infty$ be the subset of $\mathbbm{D}_{[0,\infty]}$ consisting in excursions away from $\infty$. Denote  by $X$ and $\zeta$, a generic excursion and its length. 
\begin{lem}\label{lem:noexcursion}
Let $n$ be a $\sigma$-finite positive measure on $\mathcal{E}_\infty$ such that  \begin{equation}\label{eq:markovunderexcmeas} 
\begin{aligned}
n\big(X_{t+s}\in & \, \ddr y, 0<t+s<\sigma_{x_0}^{-}<\zeta\big)\\
&=\int_{[0,\infty]}n(X_t\in \ddr x, t<\sigma_{x_0}^{-}<\zeta)\mathbb{P}_x(X_s^{\mathrm{m}}\in \ddr y, s<\sigma_{x_0}^{-}<\zeta),
\end{aligned}
\end{equation}
\begin{equation}\label{eq:concatenable}
n(1\wedge \zeta)<\infty,
\end{equation}
and 
\begin{equation}\label{eq:entrancestartfrominfinity}
n(X_t\in \ddr x, t<\zeta)\underset{t\rightarrow 0}{\rightarrow} \delta_\infty \text{ weakly},
\end{equation}

Then, if $\thetadown>1$, we have 
\[n(\exists s\in (0,\zeta): X_s<\infty)=0.\]
\end{lem}
\begin{proof}
Recall the positive function $f$. By the assumption \eqref{eq:markovunderexcmeas}, the facts that for $n$-a.e. $X$,  $t<\sigma_{x_0}^{-}$ entails $ X_t\geq x_0$ and Lemma \ref{lem:supermartingale},
\begin{align*}
n\big(f(X_{t+s}), 
t+s<\sigma_{x_0}^{-}<\zeta\big)&=\int_{[0,\infty]}n(X_t\in \ddr x, t<\sigma_{x_0}^{-}<\zeta)\mathbb{E}_x\big[f(X^{\mathrm{m}}_{s}); s<\sigma_{x_0}^{-}\big]\\
&\leq \int_{[0,\infty]}n(X_t\in \ddr x, t<\sigma_{x_0}^{-}<\zeta)\mathbb{E}_x\big[f(X^{\mathrm{m}}_{s\wedge \sigma_{x_0}^{-}})\big]\\
&\leq \int_{[0,\infty]}n(X_t\in \ddr x, t<\sigma_{x_0}^{-}<\zeta)f(x).
 \end{align*}
Notice that the measure $n(X_t\in \ddr x, t<\sigma_{x_0}^{-}<\zeta)$ is finite since only finitely many excursions may go below $x_0$, otherwise excursions would not be concatenable and \eqref{eq:concatenable} would not be satisfied. The function $f$ being in $\mathrm{C}_0$, by \eqref{eq:entrancestartfrominfinity}, we see that the right-hand side below goes to $0$ as $t$ goes to $\infty$. We therefore get, by right-continuity
\[n\big(f(X_{s}), s<\sigma_{x_0}^{-}<\zeta\big)=0,\ \forall s>0.\]
Since $f\geq 0$ and $f(x)=0$ if and only if $x=\infty$, we get
\[n\big(X_{s}<\infty, s<\sigma_{x_0}^{-}<\zeta\big)=0,\ \forall s>0.\]
In other words, under $n$, almost every excursion do not cross level $x_0$. The latter being arbitrarily large, we conclude that there is no excursion, in other words
\[n\big(X_{s}<\infty, 0<s<\zeta\big)=0,\ \forall s>0.\]
\end{proof}

We can now proceed to show Theorem \ref{thm3infty}-(ii). 
\begin{lem}\label{lem:proofthm3ii} Assume $\thetadown>1$. There exist no continuous Fellerian extension of $X^{\mathrm{m}}$ with $\infty$ non-absorbing. In particular, the extended process $X^{\mathrm{e}\infty}$, in Theorem~\ref{thm2infty}, is absorbed at $\infty$.
\end{lem}
\begin{proof}
We argue by contradiction and assume that there exists $X^{\mathrm{e}}$ a càdlàg Feller process such that $\mathbb{P}_\infty(\exists t>0: X_t\neq \infty)=1$ and $\infty$ is continuous. The Fellerian assumption ensures that there exists an excursion measure $n$ away from $\infty$ (possibly finite if $\infty$ is an entrance or is no regular for itself) which satisfies the strong Markov property, see Blumenthal \cite[Theorem~3.28, Chapter~III, pages 102-103]{MR1138461}, and therefore \eqref{eq:markovunderexcmeas}, the integrability \eqref{eq:concatenable} is also verified (excursions must be concatenable). The limit \eqref{eq:entrancestartfrominfinity}  follows from the fact that the boundary is supposed to be continuous ($n$-almost every excursions start from $\infty$). Lemma~\ref{lem:noexcursion} ensures that this is not possible. The process $X^{\mathrm{e}\infty}$ built in Theorem \ref{thm2infty} is Feller, with $\infty$ continuous, and therefore must be absorbed at $\infty$.
\end{proof}
\noindent \textbf{Proof of Theorem \ref{thm3infty}:} This is obtained by combining Lemma \ref{lem:proofthm3i} and Lemma \ref{lem:proofthm3ii}.\qed
\section{Extension at zero}\label{sec:extensionat0}
We now investigate the boundary point $0$. Our main results are presented in Theorems \ref{thm2zero} and \ref{thm3zero}.  The techniques we will employ share many aspects as our study of $\infty$. As previously, our construction of an extension out from $0$ will be done in two steps. First, with the help of Laplace duality, we study a setting with entrance at $0$ and then by taking limits we construct the general extension.

\subsection{Duality, entrance law at $0$ and prelimiting processes}
We work with two mechanisms $\Psi$ and $\hat \Psi$ and some decompositions
\begin{center} $\Psi=\Sigma-\Phi$ and $\hat{\Psi}=\hat\Sigma-\hat\Phi$.
\end{center} 
Let $X^{\mathrm{m}}$ be the minimal $\mathrm{CBDI}(\Psi,\hat \Psi)$ and $Y^{\mathrm{m}}$ the minimal $\mathrm{CBDI}(\hat \Psi,\Psi)$.  Set $Y^{\mathrm{e}\infty}$, the Fellerian extension at $\infty$ of $Y^{\mathrm{m}}$ provided by Theorem~\ref{thm2infty}. 
\medskip

In a similar way as in Section \ref{sec:prooftheorem2infinity} when dealing with $\infty$, we start by a duality statement and provide a theoretical construction of a Fellerian extension with $0$ possibly entrance. We also see that the assumption of non-extinction of the minimal dual process $Y^{\mathrm{m}}$ allows us to work with a process $X^{\mathrm{m}}$ which satisfies the Feller property at $\infty$ and whose boundary $\infty$ is absorbing. This allows us to focus on zero. Recall $\Hone$: \ $\int_0^1\frac{\ddr u}{\Phi(u)}=\infty$.
\begin{lem}\label{thm1zero} 
 The process $X^{\mathrm{m}}$ admits a Markovian extension $X^{\mathrm{e}0}$ verifying almost surely,  for all $t\in [0,\infty)$
\begin{equation}\label{def:Xezeroviamonotonicity}
 X_t^{\mathrm{e}0}(x)=X_t^{\mathrm{m}}(x) \text{ for all } x\in (0,\infty] \text{ and }
 X_t^{\mathrm{e}0}(0)=\underset{x\rightarrow 0}{\lim}\!\downarrow  X_t^{\mathrm{m}}(x) \in [0,\infty].
\end{equation}
Assume $\neg \Htwo: \int_1^{\infty}\frac{\ddr u}{\Sigma(u)}<\infty$ and $\Hone$. If $Y^{\mathrm{e}\infty}$  has its boundary $\infty$ absorbing, equivalently $Y^{\mathrm{e}\infty}=Y^{\mathrm{m}}$ a.s. and $0$ inaccessible then, under  
$0^+\cdot \infty,\ \infty^-\cdot 0,$
\begin{equation}\label{eq:laplaceduality20}\mathbb{E}_{x}[e^{-X^{\mathrm{e}0}_ty}]=\mathbb{E}^{y}[e^{-xY^{\mathrm{m}}_t}], \ \forall (x,y)\in [0,\infty]^2,\ t \in [0,\infty).\end{equation}
In particular,  $X^{\mathrm{e}0}$ is Feller and its entrance law from $0$, which can be degenerated into $\delta_0$, has a Laplace transform characterized by 
\[\mathbb{E}_{0}[e^{-X^{\mathrm{e}0}_ty}]=\mathbb{P}^{y}(Y^{\mathrm{m}}_t<\infty)=\mathbb{P}^{y}(\tau_\infty^{+}> t), \ \forall (x,y)\in [0,\infty]^2,\ t \in [0,\infty).\]
The process $X^{\mathrm{e}0}$ has thus $0$ non-absorbing (actually entrance) if and only if 
$$\forall y\in (0,\infty),\ \mathbb{P}^{y}(\tau_\infty^{+}> t)>0, \text{ for some } t\in (0,\infty),$$
i.e. the minimal $\mathrm{CBDI}(\hat{\Psi},\Psi)$, $Y^{\mathrm{m}}$, has $\infty$ accessible (and absorbing).
\end{lem}
\begin{proof} The statement \eqref{def:Xezeroviamonotonicity} follows from the comparison property and the same arguments as in the proof of Lemma~\ref{thm1infty}. The condition $\neg \Htwo$ and $\Hone$ allows us to apply Theorem~\ref{thm2infty} to $Y^{\mathrm{m}}$. Plainly, if  $Y^{\mathrm{e}\infty}=Y^{\mathrm{m}}$ a.s., then, 
\begin{equation}\label{eq:laplaceduality00}\mathbb{E}_{x}[e^{-X^{\mathrm{m}}_ty}]=\mathbb{E}^{y}[e^{-xY^{\mathrm{m}}_t}], \ \forall (x,y)\in (0,\infty)^2,\ t \in [0,\infty).\end{equation}
If moreover, $0$ is inaccessible for $Y^{\mathrm{m}}$, then  $X^{\mathrm{m}}$, whose boundary $\infty$ is absorbing, is Feller at $\infty$ (i.e. its semigroup is weakly continuous at $\infty$):
\[\mathbb{E}_{\infty}[e^{-X^{\mathrm{m}}_ty}]=\mathbb{E}_{\infty-}[e^{-X^{\mathrm{m}}_ty}]=\mathbb{P}^{y}(Y^{\mathrm{m}}_t=0)=0, \ \forall y\in [0,\infty],\ t \in [0,\infty).\]
By taking limits as $x$ goes to $0$ in \eqref{eq:laplaceduality00}, using \eqref{def:Xezeroviamonotonicity}, we end up with 
\begin{equation*}\mathbb{E}_{x}[e^{-X^{\mathrm{e}0}_ty}]=\mathbb{E}^{y}[e^{-xY^{\mathrm{m}}_t}], \ \forall (x,y)\in [0,\infty]^2,\ t \in [0,\infty),\end{equation*}
 under the convention $0^+\cdot \infty,\ \infty^-\cdot 0$. The Feller properties and the remaining claims follow then directly, similarly as in the proof of Lemma~\ref{thm1infty}.
\end{proof}

\begin{table}[htpb]
\begin{center}
\begin{tabular}{|c|c|}
\hline
$\mathrm{CBDI}(\hat{\Psi},\Psi)$ &  $\mathrm{CBDI}(\Psi,\hat{\Psi})$ \\
\hline
$0$ non-absorbing  &  $\infty$ accessible \\
\hline
\end{tabular}
\vspace*{4mm}
\caption{Non-absorption at $0$ for $X^{\mathrm{e}0}$ and Accessibility of $\infty$ for $Y^{\mathrm{m}}$}
\label{correspondance3}
\end{center}
\end{table}

A consequence of Lemma \ref{thm1zero} is that the cooperation mechanism $\hat{\Phi}$ necessarily satisfies $\int_0^{1}\frac{\ddr u}{\hat{\Phi}(u)}<\infty$ for the boundary $0$ of the $\mathrm{CBDI}(\Psi,\hat{\Psi})$ to be non-absorbing.  Introduce the condition
\[\hypertarget{hatH2}{\hat{\mathbb{H}}_2}: \ \ \int_{1}^{\infty}\frac{\ddr u}{\hat\Sigma(u)}=\infty.\]
\begin{prop}\label{prop:necessarycondfornonabsorptionat0}
Assume $\neg \Htwo: \int_1^{\infty}\frac{\ddr u}{\Sigma(u)}<\infty$, $\Hone$ and $\hatHtwo$. If the following holds $$\hatHone: \int_0^1\frac{\ddr u}{\hat\Phi(u)}=\infty$$ then, $X^{\mathrm{e0}}$ has $0$ inaccessible absorbing ( $0$ is a natural boundary).
\end{prop}
\begin{proof}
By Proposition \ref{prop:suffcondinaccessibility},  $\hatHtwo$ and $\hatHone$ entail that $Y^{\mathrm{m}}$ cannot touch $0$ and $\infty$ respectively. By Lemma~\ref{thm1zero}, we see from \eqref{eq:laplaceduality20} that $0$ is absorbing for $X^{\mathrm{e}0}$.
\end{proof}
The next corollary provides explicit conditions for $X^{\mathrm{e}0}$ to exist and have $0$ entrance.
\begin{cor}\label{cor:entrancewithlambda>0} Let $\Psi$ and $\hat{\Psi}$ be mechanisms and some decompositions $\Psi=\Sigma-\Phi$, $\hat\Psi=\hat\Sigma-\hat\Phi$. Suppose $\neg \Htwo: \ \int_1^{\infty}\frac{\ddr u}{\Sigma(u)}<\infty$ and $\Hone$.
    \begin{enumerate}[(1)]
        \item If $\underline{\theta}_{\scriptscriptstyle\, \hat{\Phi},\Sigma}>1$, then the process $X^{\mathrm{m}}$ admits a $[0,\infty]$-valued càdlàg extension $X^{\mathrm{e}0}$ at $0$, with $0$ an entrance. Moreover, it satisfies \eqref{def:Xezeroviamonotonicity} and the $(0^{+}\cdot \infty, \infty^{-}\cdot 0)$-Laplace duality relationship \eqref{eq:laplaceduality20}. The process $X^{\mathrm{e}0}$ has furthermore $\infty$ absorbing  and $Y^{\mathrm{m}}$ has $\infty$ accessible, $0$ inaccessible.
        \item If there is no diffusive part in $\Psi$, i.e. $\mathrm{a}=0$ in \eqref{branchingmechanism}, $\hat\lambda=\hat\Phi(0)>0$ then $X^{\mathrm{e}0}$ exists with all properties stated in (1).
    \end{enumerate}
\end{cor}
\begin{rem}
The assumption $\neg \Htwo: \int_1^{\infty}\frac{\ddr u}{\Sigma(u)}<\infty$  ensures that the $\mathrm{CB}(\Psi)$ process (no interaction) can hit $0$. The conditions $\Hone:=\int_0^1\frac{\ddr u}{\Phi(u)}=\infty$ and $\mathrm{a}=0$ are merely technical. Specifically,
\begin{itemize}
    \item $\Hone$ guarantees that the   $\mathrm{CB}(\Psi)$ process cannot explode, which in turn prevents explosion of the $\mathrm{CBDI}(\Psi,\hat\Psi)$, $X^{\mathrm{e}0}$. 
\item The condition $\mathrm{a}=0$ excludes the presence of a Feller diffusion component. Such a diffusive part cannot be accommodated in our argument: when it is present, the boundary $0$ is not necessarily an entrance for $X^{\mathrm{e}0,(n)}$ (equivalently, it is not always an exit for $Y^{\mathrm{e}\infty}$).
\end{itemize}
In summary, the first assumption $\neg \Htwo$ ensures extinction of the branching part, the second $\Hone$ controls the boundary at infinity (no explosion and absorption), and the last, $\mathrm{a}=0$, avoids complications from diffusion.
\end{rem}
\begin{proof} 
By Proposition~\ref{prop:suffcondinaccessibility}-(2), $\hatHtwo$ ensures that the minimal $\mathrm{CBDI}(\hat\Psi,\Psi)$ $Y^{\mathrm{m}}$ with mechanism, does not get extinct. For part (a), Condition $\Hone$ allows us to apply Theorem~\ref{cor:accessibility0bytheta} which  guarantees the inaccessibility of $0$ for $Y^{\mathrm{m}}$. The remaining part of the statement is a direct consequence of Lemma~\ref{thm1zero}. 
\end{proof}
\subsection{Construction of the extension}
 We investigate now an extension of the minimal process $X^{\mathrm{m}}$ at $0$ without requiring the inaccessibility of this boundary. The construction is made by a limiting procedure. Intuitively, the process is extended at $0$ by taking limit of processes with constant drift immigration becoming very low.
\medskip 

 Let $\hat{\Psi}=\hat{\Sigma}-\hat{\Phi}$ be a drift-interaction mechanism of the form \eqref{branchingmechanism}, with no killing term: $\hat{\Phi}(0)=0$.  Let $(\hat{\lambda}_n)_{n\ge 1}$ be a positive sequence decreasing towards $0$. Set 
\begin{equation}\label{eq:psi_n}
\hat{\Psi}_n:=\hat{\Psi}-\hat{\lambda}_n=\hat{\Sigma}-\hat{\Phi}_n \text{ with }\hat{\Phi}_n:=\hat{\Phi}+\hat{\lambda}_n, \ \ n\in \mathbb{N}.\end{equation}
Notice the \textit{killing} term in $\hat{\Psi}_n$: $$\hat{\Psi}_n(0)=-\hat{\lambda}_n.$$    The main additional assumption, compared to our study of $\infty$, is that the branching mechanism $\Psi$ has no diffusive component, i.e. $\mathrm{a}=0$ in \eqref{branchingmechanism}.
\medskip

Corollary \ref{cor:entrancewithlambda>0}-(b) provides a sequence $\big(X^{\mathrm{e}0, (n)}\big)_{n\geq 1}$  of extended $\mathrm{CBDI}(\Psi,\hat{\Psi}_n)$ processes, all with $0$ entrance.
\begin{theo}\label{thm2zero} Assume $\neg \Htwo$,\ $\Hone$ and $\mathrm{a}=0$ in \eqref{branchingmechanism} (no diffusive part in $\Sigma$). There exists a $[0,\infty]$-valued càdlàg Markov process  $X^{\mathrm{e}0}$, such that
\begin{equation*}\label{convthm2}
\forall x\in [0,\infty],\ \mathbb{P}_x-\text{a.s., } \underset {n\rightarrow \infty}{\lim}\! \downarrow X^{\mathrm{e}0, (n)}_t  =X^{\mathrm{e}0}_t,\ \forall t\in [0,\infty)\text{, and }
X^{\mathrm{e}0, (n)} \underset{n\rightarrow \infty}{\Longrightarrow} X^{\mathrm{e}0} \text{ in } \mathbb{D}_{[0,\infty]}.\end{equation*}
The process $X^{\mathrm{e}0}$ is a Fellerian continuous extension of $X^{\mathrm{m}}$ at $0$, absorbed at $\infty$. It satisfied under $0^+\cdot \infty, \infty^- \cdot 0$, 
\begin{equation}\label{eq:dualityXe0}\mathbb{E}_x[e^{-X_t^{\mathrm{e}0}y}]=\mathbb{E}^y[e^{-xY_t^{\mathrm{m}}}], \ \forall (x,y)\in [0,\infty]^2,\ t\in [0,\infty).\end{equation}
\end{theo}
\begin{rem} 
Notice that Theorem~\ref{thm2zero} allows for strong cooperation mechanism $\hat\Phi$, specifically those for which $\neg \hatHone$: $\int_{0}^{1}\frac{\ddr u}{\hat{\Phi}(u)}<\infty$ holds.
\end{rem}
\begin{proof}
 Denote by $X^{\mathrm{e}0,(n)}$, the Fellerian $\mathrm{CBDI}(\Psi,\hat{\Psi}_n)$ with $0$ entrance, provided by Corollary~\ref{cor:entrancewithlambda>0}. Recall that is arises also as the decreasing pointwise almost sure limit:$$X^{\mathrm{e}0,(n)}(x)=\underset{n\rightarrow \infty}\lim\!\downarrow X^{\mathrm{m},(n)}(x).$$ 
\begin{enumerate}[1]
    \item \textbf{Existence of a pointwise limiting process of $(X^{\mathrm{e}0,(n)})_{n\geq 1}$ as $n$ goes to $\infty$.} \label{step1zero} 
     The fact that $(\hat{\lambda}_n)_{n\geq 1}$ is decreasing, together with the comparison property, Proposition~\ref{prop:comparison},  ensure that the sequence of processes $X^{\mathrm{e}0,(n)}$ admits an almost-sure monotonic pointwise limit. Denote it by $X^{\mathrm{e}0}$. We verify that $X^{\mathrm{e}0}$ is solution to the martingale problem on $(0,\infty)$ given in Definition~\eqref{defCBDI}-\eqref{(i)CBDIgen}. Plainly, if one denotes by $\mathcal{X}^{(n)}$ the generator of $X^{\mathrm{m},(n)}$, then \[||\mathcal{X}^{(n)}f-\mathcal{X}f||_\infty=\hat{\lambda}_n||f'||_\infty\underset{n\rightarrow \infty}{\longrightarrow} 0.\]

    Note that at this stage, we do not know whether $X^{\mathrm{e}0}$ is Markovian nor if it has a càdlàg version. Nevertheless, the process, stopped at its first passage at $0$, is a minimal $\mathrm{CBDI}(\Psi,\hat{\Psi})$.
    \item \textbf{Study of the dual processes of $X^{\mathrm{e}0,(n)}$ and their limit as $n$ goes to $\infty$.}  Let $Y^{\mathrm{m}}$ be the minimal $\mathrm{CBDI}(\hat{\Psi},\Psi)$, given by the solution of \eqref{SDECBDIY}. We enlarge the probability space on which the latter is defined with an independent  exponential random variable $\mathbbm{e}$ of parameter $1$. Define, for all $n\geq 1$ 
   \begin{equation}\label{eq:defhatzeta} \hat{\zeta}^{(n)}_{\mathrm{k}}:=\inf\{t>0: \hat{\lambda}_n\int_0^t Y^{\mathrm{m}}_s\ddr s>\mathbbm{e} \}.
\end{equation}   
    Consider $(Y^{\mathrm{m},(n)})_{n\geq 1}$ the sequence of $\mathrm{CBDI}(\hat{\Psi}_n,\Psi)$ constructed as follows
    \begin{equation*}
        Y^{\mathrm{m},(n)}_t:= \begin{cases} Y^{\mathrm{m}}_{t} \,&  \ t<\hat{\zeta}^{(n)}_{\mathrm{k}}\\
        \infty\,& \ t\geq \hat{\zeta}_{\mathrm{k}}^{(n)}. 
        \end{cases}
    \end{equation*}
    By the assumption $\mathrm{a}=0$ and since $\hat{\lambda}_n>0$, one has  $\underline{\theta}_{\scriptscriptstyle{\hat\Phi_n,\Sigma}}=\infty$.  By Corollary \ref{cor:entrancewithlambda>0}, using here the assumptions $\hatHtwo$  and $\Hone$, the processes $Y^{\mathrm{m},(n)}$ are in Laplace duality with $X^{\mathrm{e}0,(n)}$ and have all their boundary $\infty$ as exit.    We check that they almost surely converge pointwise towards $Y^{\mathrm{m}}$.  By definition \[\tau^{+,(n)}_{\infty}=\inf\{t\geq 0: Y^{\mathrm{m},(n)}_{t-}=\infty \text{ or } Y_{t}^{\mathrm{m},(n)}=\infty\}=\tau_\infty^{(+)}\wedge \hat{\zeta}_{\mathrm{k}}^{(n)},\] 
    where $\tau_\infty^{(+)}$ is the first explosion time of $Y^{\mathrm{m}}$. Since $\hat{\lambda}_{n+1}\leq \hat{\lambda}_{n}$ for all $n\geq 1$, one has almost surely, $\hat{\zeta}_{\mathrm{k}}^{(n+1)}\geq \hat{\zeta}_{\mathrm{k}}^{(n)}$, $Y_{t}^{\mathrm{m},(n+1)}=Y_{t}^{\mathrm{m},(n)}$ for all $t\leq \zeta_{\mathrm{k}}^{(n)}$ and then \[Y_{t}^{\mathrm{m},(n+1)}\leq Y_{t}^{\mathrm{m},(n)} \text{ for all } t\geq 0.\]  Let $Y^{\mathrm{m}\infty}$ be the $[0,\infty]$-valued pointwise decreasing limit, $Y^{\mathrm{m}\infty}:=\underset{n\rightarrow \infty}{\lim}\!\downarrow Y^{\mathrm{m},(n)}$. Plainly, it coincides with $Y^{\mathrm{m}}$ on any intervals of the form $[0,\tau_\infty^{+,(n)}),\ n\geq 1$. Moreover, $\tau_{\infty}^{+,(n)}$ almost surely increases, as $n$ goes to $\infty$, towards a random variable $\tau_{\infty}^{+,(\infty)}$. Let us verify that this is the first explosion time of $Y^{\mathrm{m}\infty}$.
    
    Let $t\in (0,\infty)$. On the event $\{\tau_{\infty}^{+,(\infty)}\leq t\}$, one has $t\geq \tau_{\infty}^{+,(n)}$, for all $n\geq 1$, thus  $Y_{t}^{\mathrm{m},(n)}=\infty$ a.s. and $Y_{t}^{\mathrm{m}\infty}=\underset{n\rightarrow \infty}{\lim} Y_{t}^{\mathrm{m},(n)}=\infty$. Hence $Y^{\mathrm{m}\infty}$ is stuck at $\infty$ on the interval $[\tau_{\infty}^{+,(\infty)},\infty)$.  On the event $\{t<\tau_\infty^{+,(\infty)}\}$, there is $n$  large enough such that $t<\tau_\infty^{+,(n)}$ and thus $Y_t^{\mathrm{m}\infty}\leq Y_t^{\mathrm{m},(n)}<\infty$ a.s. Hence,  $\tau_{\infty}^{+,(\infty)}$ is the first explosion time of $Y^{\mathrm{m}\infty}$ and the latter is absorbed at $\infty$ at this time. Denoting by $\mathcal{Y}^{(n)}$ the generator of $Y^{\mathrm{m},(n)}$, we have plainly, for all $f\in \mathrm{C}^{2}_c((0,\infty))$, \[\|\mathcal{Y}^{(n)}f-\mathcal{Y}f\|_\infty=\hat{\lambda}_n\|f\|_\infty \underset{n\rightarrow \infty}{\longrightarrow} 0.\]
    By Lemma~\ref{lem:cvPM}, the process $Y^{\mathrm{m}\infty}$ is therefore a process, with both boundaries absorbing, solution to $\mathrm{MP}\big(\mathcal{Y}, \mathrm{C}^{2}_c((0,\infty))\big)$. We conclude by Theorem~\ref{thm:minX}, that $Y^{\mathrm{m}\infty}$ is a minimal $\mathrm{CBDI}(\Psi,\hat{\Psi})$. Since by construction, $Y^{\mathrm{m}\infty}\geq Y^{\mathrm{m}}$ a.s., they actually coincide a.s.
    \item \textbf{Study of the semigroups of $X^{\mathrm{e}0,(n)}$ and $X^{\mathrm{e}0}$.}   The duality relationship \eqref{eq:dualityXe0} follows by taking limits in
    \[\mathbb{E}_x[e^{-X_t^{\mathrm{e}0,(n)}y}]=\mathbb{E}^y[e^{-xY_t^{\mathrm{m},(n)}}].\]
    The fact that the process $X^{\mathrm{e}0}$ is Markov and Feller is a consequence of \eqref{eq:dualityXe0}. This is checked in the same way as for $X^{\mathrm{e}\infty}$, in Step~\ref{step3} and we omit the details. 
    
    Let us explain the uniform convergence of the semigroups of $X^{\mathrm{e}0,(n)}$ towards that of $X^{\mathrm{e}0}$. 
    Let $n\geq 1$.  Since $Y_t^{\mathrm{m}} \leq Y_t^{\mathrm{m},(n)}$ almost surely, we have $\mathbb{P}^{y}(Y^{\mathrm{m}}_t=\infty, \ Y^{\mathrm{m},(n)}_t<\infty)=0$ and, with the convention $0^+\cdot \infty$,
        \begin{align*}
\lVert P^{(n)}_t\ee^y-P_t^{\mathrm{e}0}\ee^y\rVert_\infty&\leq \mathbb{E}^y\left[\underset{x\in[0,\infty]}{\sup} \, \lvert e^{-xY_t^{\mathrm{m},(n)}}-e^{-xY_t^{\mathrm{m}}}\rvert \right] \\
&=\mathbb{E}^y\left[\underset{x\in[0,\infty]}{\sup}\left(e^{-xY_t^{\mathrm{m},(n)}}-e^{-xY_t^{\mathrm{m}}}\right)\mathbbm{1}_{\{Y_t^{\mathrm{m}}<\infty\}}\right]\\
& \quad +  \mathbb{E}^y\left[\underset{x\in[0,\infty]}{\sup}\left(e^{-xY^{\mathrm{m},(n)}_t}\right)\mathbbm{1}_{\{Y^{\mathrm{m}}_t=\infty,\ Y^{\mathrm{m},(n)}_t<\infty\}}\right]\\
& =\mathbb{E}^y\left[\underset{x\in[0,\infty]}{\sup}\left(e^{-xY_t^{\mathrm{m},(n)}}-e^{-xY^{\mathrm{m}}_t}\right)\mathbbm{1}_{\{Y^{\mathrm{m}}_t<\infty\}}\right].
\end{align*}
On the event $\{Y_t^{\mathrm{m}}<\infty\}$, the almost sure monotone convergence $\underset{n\rightarrow \infty}{\lim}\! \downarrow  Y^{\mathrm{m},(n)}_t =Y_t^{\mathrm{m}}$ ensures also that for large enough $n$, $Y_t^{\mathrm{m},(n)}<\infty$ almost surely. By applying verbatim the argument in the proof of Lemma~\ref{prop:cvskorokhod}, one can show that the supremum above vanishes as $n$ goes to $\infty$ almost surely. Thus, $\lVert P_t^{(n)}\ee^y-P_t^{\mathrm{e}0}\ee^y\rVert_\infty \rightarrow 0$ as $n \rightarrow\infty$ and therefore $X^{\mathrm{e}0,(n)} \underset{n\rightarrow \infty}{\Longrightarrow} X^{\mathrm{e}0}$
in $\mathbbm{D}_{[0,\infty]}$. 
\item 

\textbf{Infinitesimal generator of $X^{\mathrm{e}0}$ and continuity of the boundary $0$}.  We study in this step the pointwise infinitesimal generator, see \eqref{eq:pointwisedomain} (and replace $\infty$ by $0$ there), of the Feller process $X^{\mathrm{e}0}$. Recall $\mathcal{X}$ in \eqref{eq:genCBDIX}.
\begin{lem} For all $y\in (0,\infty)$, $\mathrm{e}^{y}\in \mathcal{D}^{\mathrm{p}}_{\mathcal{X}^{\mathrm{e}0}}$, $$\mathcal{X}^{\mathrm{e}0}\mathrm{e}^{y}(x)=\mathcal{X}\mathrm{e}^{y}(x), \ x\in (0,\infty], \text{ and } \mathcal{X}^{\mathrm{e}0}\mathrm{e}^{y}(0)=0.$$
	In particular, the process $X^{\mathrm{e}0}$ does not jump from $0$.\end{lem}
The boundary $0$ is therefore either an absorbing boundary or a non-absorbing \textit{continuous} boundary. Note also that since $\ee^y$ is in the domain, the process $M^{y,X^{\mathrm{e}0}}$ defined in \eqref{eq:martingaleexpo} is a martingale (also under $\mathbb{P}_0$). 
\begin{proof}
	The convergence, for all $y\in (0,\infty)$, \[\frac{1}{t}\left(\mathbb{E}_x[\ee^y(X_t^{\mathrm{e}0})]-\ee^y(x)\right)\underset{t\rightarrow 0}{\longrightarrow}\mathcal{X}^{\mathrm{e}0}\mathrm{e}^{y}(x)=\mathcal{X}\mathrm{e}^{y}(x), \ x\in (0,\infty]\]
    is shown along the same argument as for $X^{\mathrm{e}\infty}$ in
    Step~\ref{stepcontinuityinfinity}, see Lemma~\ref{lem:genXeinfinity}. We establish now that 
    \[\frac{1}{t}\left(\mathbb{E}_0[\ee^y(X_t^{\mathrm{e}0})]-\ee^y(x)\right)\underset{t\rightarrow 0}{\longrightarrow}\mathcal{X}^{\mathrm{e}0}\mathrm{e}^{y}(0)=0.\]
    The fact that $0$ is a continuous boundary  is explained after. From the duality relationship, we have \[\mathbb{E}_0[e^{-X_t^{\mathrm{e}0}y}]=\mathbb{P}^{y}(Y_t<\infty).\]
	So that, 
	\[-\mathcal{X}^{\mathrm{e}0}\ee^y(0)=\underset{t\to 0+}{\lim} \frac{1}{t}\mathbb{E}_0[1-e^{-X_t^{\mathrm{e}0}y}]=\underset{t\to 0+}{\lim} \frac{1}{t}\mathbb{P}^y(Y_t^{\mathrm{m}}=\infty).\]
	By assumption $\hat{\lambda}=0$, i.e. there is no killing term in $Y^{\mathrm{m}}$. We check now from the stochastic equation \eqref{SDECBDIY} that 
	$$\underset{t\to 0+}{\lim} \frac{1}{t}\mathbb{P}^y(Y_t^{\mathrm{m}}=\infty)=0.$$
	 Fix $y\in (0,\infty)$ and $n\ge 1$. Recall $\hat{\pi}$ the L\'evy measure of $\hat\Psi$ and denote the tail by $\bar{\hat{\pi}}(z):=\hat\pi([z,\infty))$. 
     Let $\varphi\in \mathrm{C}^{2}_b$ be such that 
	\[
	\mathbbm{1}_{\{z\ge y+n\}}\leq \varphi(z)\leq \mathbbm{1}_{\{z\ge y+n/2\}}.
	\]
	Thus, 
	\begin{equation}\label{ineqvarphi}\mathbb{P}^y(Y_t^{\mathrm{m}}\geq y+n)\leq \mathbb{E}^{y}[\varphi(Y_t^{\mathrm{m}})]\leq \mathbb{P}^y(Y_t^{\mathrm{m}}\geq y+n/2).\end{equation}
	Since $\varphi(y)=0$, by Itô's lemma
	\[\frac{1}{t}\mathbb{E}^{y}[\varphi(Y_t^{\mathrm{m}})]=\frac{1}{t} \mathbb{E}^{y}\left[\int_0^t\mathcal{Y}\varphi(Y_s^{\mathrm{m}})\ddr s\right].\]
	Recall the form of $\mathcal{Y}$ in \eqref{eq:genY}.  The map $\mathcal{Y}\varphi$ is locally bounded and since $Y_s^{\mathrm{m}}\underset{s\rightarrow 0+}{\longrightarrow} y$, $\mathbb{P}^y$-a.s, we see from Lebesgue theorem that
	$$\frac{1}{t}\mathbb{E}^{y}[\varphi(Y_t^{\mathrm{m}})]\underset{t\rightarrow 0+}{\longrightarrow} \mathcal{Y}\varphi(y).$$
	 	 Since $\varphi$ vanishes in a neighborhood of $y$, we have that $\varphi'(y)=\varphi''(y)=0$ and one can readily check, from \eqref{eq:genY}, that
\begin{equation}\label{ineqbarY} y\bar{\hat{\pi}}(n)\leq  \mathcal{Y}\varphi(y)\leq y\bar{\hat{\pi}}(n/2).\end{equation}	 	 
By combining \eqref{ineqvarphi} and \eqref{ineqbarY}, we see that  $$\underset{t\to 0+}{\limsup} \frac{1}{t}\mathbb{P}^y(Y_t^{\mathrm{m}}=\infty)\leq \underset{t\to 0+}{\limsup}\, \frac{1}{t}\mathbb{P}^y(Y_t^{\mathrm{m}}\geq y+n)\leq y\bar{\hat\pi}(n/2).$$ 
Plainly $\bar{\hat\pi}(n/2) \underset{n\rightarrow \infty}{\longrightarrow} 0$, therefore $\underset{t\to 0+}{\lim} \frac{1}{t}\mathbb{P}^y(Y_t^{\mathrm{m}}=\infty)=0$ and we have $\mathcal{X}^{\mathrm{e}0}(0)=0$.
\medskip

 We explain now that $0$ is continuous. The dynamics of the process when it is at $0$ are encoded along a L\'evy measure $\nu_0$ (integrating $1\wedge u$), a drift term $d_0\in [0,\infty)$ and a killing term $c_0\in [0,\infty)$ so that
	\begin{equation}\label{eq:levytripletat0}-\mathcal{X}^{\mathrm{e}0}\ee^y(0)=\int_0^{\infty}\left(1-e^{-uy}\right)\nu_0(\ddr u)+d_0 y+c_0,\end{equation}
	see \cite[Theorem 4.2 and Equation (4.3), page 31]{foucartvidmar2025}. Since $\mathcal{X}^{\mathrm{e}0}\ee^y(0)=0$ for all $y\in (0,\infty)$, we see that $c_0=d_0=0$ and $\nu_0\equiv 0$. 
    
    As a consequence, the process does not jump from $0$. In order to establish this rigorously, we may invoke semi-martingales theory, see e.g. \cite{zbMATH01834045}. Since, for all $y\in (0,\infty)$,  $\ee^y$ belongs to the domain\footnote{pointwise but also strong since $X^{\mathrm{e}0}$ is Feller}, of the Feller process $X^{\mathrm{e}0}$, the process $\big(\ee^{y}(X_t^{\mathrm{e}0}),t\geq 0\big)$ takes the form
    \[\ee^{y}(X_t^{\mathrm{e}0})=\ee^{y}(x)+M^{y,X^{\mathrm{e}0}}_t+\int_0^t \mathcal{X}^{\mathrm{e}0}\mathrm{e}^{y}(X_s^{\mathrm{e}0})\ddr s, \quad t\in [0,\infty),\]
        where $M^{y,X^{\mathrm{e}0}}$ is a martingale and $$\mathcal{X}^{\mathrm{e}0}\mathrm{e}^{y}(X_s^{\mathrm{e}0})=\ee^{y}(X^{\mathrm{e}0}_{s})\big(X^{\mathrm{e}0}_{s}\Psi(y)+\hat{\Psi}(X^{\mathrm{e}0}_{s})y\big).$$ Hence $\ee^{y}(X^{\mathrm{e}0})$ and  $X^{\mathrm{e}0}$ are semimartingales. Recall $\pi$  and $-\gamma$ the L\'evy measure and the drift parameter of $\Psi$,  respectively. Let $\mu(\ddr s,\ddr u)$ be the random measure of jumps of $X^{\mathrm{e}0}$, that is
        \[\mu([0,t]\times A):=\sum_{0<s\leq t}\mathbbm{1}_{\{\Delta X_s^{\mathrm{e}0}\in A, \Delta X_s^{\mathrm{e}0}\neq 0\}}, \ \ t\in [0,\infty),\ A\in \mathrm{B}_{[0,\infty]}.\]
     Call $\nu$ its predictable compensator. Recall that by assumption, there is no diffusive part in $X^{\mathrm{e}0}$. By It\^o's formula, see e.g. \cite[Theorem~2.42, page 86]{zbMATH01834045}, one has, for all $y\in (0,\infty)$
     \begin{align*}e^{-X_t^{\mathrm{e}0}y}=e^{-xy}
     +&\int_0^t\int_0^{\infty}e^{-X_{s-}^{\mathrm{e}0}y}\left(e^{-uy}-1\right)\big(\mu-\nu)(\ddr s, \ddr u)\\
     +&\int_0^t\int_0^{\infty}e^{-X_{s-}^{\mathrm{e}0}y}\left(e^{-uy}-1+uy\mathbbm{1}_{[0,1]}(u)\right)\nu(\ddr s, \ddr u)\\
     +& \int_0^te^{-X_{s-}^{\mathrm{e}0}y}\big(y\hat{\Psi}(X_{s-}^{\mathrm{e}0})+\gamma X_{s-}^{\mathrm{e}0}\big)\ddr s, \ \ t\in [0,\infty)
     \end{align*}
where the first line is a martingale (not only local since the integrand is bounded on finite intervals of time)  and by uniqueness of Doob-Meyer's decomposition (with the given truncature function $\mathbbm{1}_{[0,1]}(u)$) the second line has to match with the non-local part of $\mathcal{X}^{\mathrm{e}0}\mathrm{e}^{y}$. Thus, for all $y\in (0,\infty)$,
\[\int_0^\infty e^{-X^{\mathrm{e}0}_{s-}y}\big(e^{-uy}-1+uy\mathbbm{1}_{[0,1]}(u)\big)\nu(\ddr s,\ddr u)=\int_0^\infty e^{-X^{\mathrm{e}0}_{s-}y}\big(e^{-uy}-1+uy\mathbbm{1}_{[0,1]}(u)\big)\nu_{X_{s-}^{\mathrm{e}0}}(\ddr u)\]
with $\nu_x(\ddr u)=x\pi(\ddr u)$ if $x>0$ and  $\nu_x(\ddr u)=0$  if $x=0$, as previously established.
Therefore, by the uniqueness of the L\'evy measure associated with a L\'evy-Khintchine function (for the given truncation), see e.g. \cite{foucartvidmar2025}, we have
\[\nu(\ddr s,\ddr u)=\ddr s \, X^{\mathrm{e}0}_{s-}\pi(\ddr u), \text{ if } X^{\mathrm{e}0}_{s-}>0 \text{ and }\nu(\ddr s, \ddr u)=0, \text{ if }X^{\mathrm{e}0}_{s-}=0.\]
\end{proof}
\begin{rem} In the case $\hat{\Phi}(0)=\hat{\lambda}>0$, the additional positive drift term $d_0:=\hat{\lambda}$ at $0$ arises in the L\'evy triplet \eqref{eq:levytripletat0}.
\end{rem}
\begin{rem} The argument explained in Remark~\ref{rem:nonegativejump} for asserting that $X^{\mathrm{e}\infty}$ makes no jump from $\infty$, based on the convergence in $\mathbbm{D}_{[0,\infty]}$, cannot be used in the study of $0$ since the prelimiting processes here have positive jumps. The continuity of the boundary $\infty$ is easier to handle since it can only be a holding point (no instantaneous jump), see Remark \ref{rem:noinstantaneousjump}.
\end{rem}
The proof of Theorem~\ref{thm2zero} is achieved.
        \end{enumerate}
\end{proof}
\section{Absorption and non-absorption at $0$}\label{sec:proofofthm3zero}
The aim of this section is to study the behavior of $X^{\mathrm{e}0}$ at $0$. We introduce parameters that distinguish the cases in which $0$ is non-absorbing from those in which $0$ is absorbing. 

Recall $\Psi=\Sigma-\Phi$, $\hat\Psi=\hat\Sigma-\hat\Phi$, Section~\ref{sec:LKfunction}, and let $\hat U$ be the potential measure associated to $\hat\Phi$, see Section \ref{sec:potentialelements} for background.

\subsection{The parameters $\rhoup$ and $\rhodown$}
Introduce the $[0,\infty]$-valued parameters
\begin{equation}\label{eq:rho}
\rhodown:= \underset{x\rightarrow 0}{\liminf}\, x\int_{0}^{\infty}e^{-xz}\frac{\Sigma(z)}{z}\hat{U}(\ddr z), \ 
\rhoup:= \underset{x\rightarrow 0}{\limsup}\, x\int_{0}^{\infty}e^{-xz}\frac{\Sigma(z)}{z}\hat{U}(\ddr z).
\end{equation}
\begin{theo}\label{thm3zero} 
Assume  $\Hone:\ \int_0^1\frac{\ddr u}{\Phi(u)}=\infty, \mathrm{a}=0$  and
$$\neg \Htwo: \ \int_1^{\infty}\frac{\ddr u}{\Sigma(u)}<\infty \ \text{ and }\  \neg \hatHone: \ \int_0^1\frac{\ddr u}{\hat{\Phi}(u)}<\infty.$$
	\begin{enumerate}
		\item[i)] If $\rhoup
		<1$ then $X^{\mathrm{e}0}$  has $0$ instantaneous non-absorbing.
		\item[ii)] If $\rhodown>1$ then $X^{\mathrm{m}}$ admits no non-trivial Fellerian continuous extension of $X^{\mathrm{m}}$, that is any such extension must be absorbed at $0$ after its extinction time. In particular, $X^{\mathrm{e}0}$  has $0$ absorbing.
	\end{enumerate} 
\end{theo}

The proof of Theorem \ref{thm3zero} is deferred to Section \ref{sec:proofofthm3zero}. \medskip

Similarly as in Section~\ref{sec:absorptionatinfinity}, a consequence of Theorem~\ref{thm3zero} together with Theorem~\ref{thm2zero} and the Laplace duality \eqref{eq:dualityXe0}, is the following  conditions for accessibility and inaccessibility of the boundary $\infty$ for the minimal $\mathrm{CBDI}(\Psi,\hat{\Psi})$-process $X^{\mathrm{m}}$. 
Notice that the conditions involve the \textit{dual} parameters $\rhodowndual$ and $\rhoupdual$.

\begin{theo}\label{thm:accessibilityinftybyrho}Assume  $\hatHone:\ \int_0^1\frac{\ddr u}{\hat\Phi(u)}=\infty, \hat{\mathrm{a}}=0$  and
$$\neg \hatHtwo: \ \int_1^{\infty}\frac{\ddr u}{\hat\Sigma(u)}<\infty \ \text{ and }\  \neg \Hone: \ \int_0^1\frac{\ddr u}{\Phi(u)}<\infty.$$
 \begin{enumerate} \item[(i)] If $\rhoupdual<1$ then $X^{\mathrm{m}}$ has $\infty$ accessible. \item[(ii)] If  $\rhodowndual>1$ then $X^{\mathrm{m}}$ has $\infty$ inaccessible.
			
		\end{enumerate}
	\end{theo}
	\begin{proof}
		Consider the minimal $\mathrm{CBDI}(\hat{\Psi}, \Psi)$, $Y^{\mathrm{m}}$. Then, the assumptions $\Htwo$  and $\hat{\mathrm{a}}=0$  allow us to apply Theorem \ref{thm2zero} to the latter (with therein $\Phi$ and $\Sigma$ playing the roles of $\hat\Phi$ and $\hat\Sigma$, the duality relationship \eqref{eq:dualityXe0}, applied to $Y^{\mathrm{e}0}$ reads $\mathbb{E}^{y}[e^{-xY^{\mathrm{e}0}_t}]=\mathbb{E}_x[e^{-X^{\mathrm{m}}_ty}]$ for all $x,y\in (0,\infty)$. By letting $y$ go to $0$, we see that $0$ is non-absorbing for $Y^{\mathrm{e}0}$ if and only if $\infty$ is accessible for $X^{\mathrm{m}}$. Theorem \ref{thm3zero} then allows us to conclude. 
	\end{proof}

The next lemma gathers some simple analytical facts that turn to be useful in the study of the parameters $\rhoup, \rhodown$. 
\begin{lem}\label{lem:boundsrho} Assume $\mathrm{a}=0$ and $\neg \hatHone$. 
	\begin{enumerate}
		\item Let $\eta$ be the L\'evy measure associated to $\Sigma$. Then, $$\rhoup=\underset{x\rightarrow 0}{\limsup}\,x\int_0^{\infty}\bar{\eta}(u)\Big(\frac{1}{\hat\Phi(x)}-\frac{1}{\hat{\Phi}(x+u)}\Big)\ddr u,$$
        and similarly for $\rhodown$ with $\liminf$.
        \item
		If $\Sigma\sim \Sigma_1$ at $\infty$, then $$\overline{\varrho}_{\,\scriptscriptstyle\Sigma, \hat{\Phi}}=\overline{\varrho}_{\,\scriptscriptstyle \Sigma_1,\hat\Phi}\text{ and } \underline{\varrho}_{\,\scriptscriptstyle \Sigma, \hat{\Phi}}=\underline{\varrho}_{\,\scriptscriptstyle\Sigma_1, \hat{\Phi}}.$$ 
		
		\item  There is no loss of generality in assuming that the measure $\hat{U}$ admits a density $\hat u$, moreover
\[\rhodown=\underset{x\rightarrow 0}{\liminf}\,\mathbb{E}[B(\mathbbm{e}_x)], \quad \rhoup=\underset{x\rightarrow 0}{\limsup}\,\mathbb{E}[B(\mathbbm{e}_x)],\]
where 
\begin{center}
$B:(0,\infty)\ni z \mapsto \frac{\Sigma(z)\hat{u}(z)}{z}$ and $\mathbbm{e}_x$ is an exponential r.v. with parameter $x$.\end{center}
One has also the bounds
		\[\liminf_{z\rightarrow \infty}\frac{\Sigma(z)\hat{u}(z)}{z}\leq \underline{\varrho}_{\,\scriptscriptstyle\Phi, \hat{\Sigma}}\leq \overline{\varrho}_{\,\scriptscriptstyle\Phi, \hat{\Sigma}}\leq \underset{z\rightarrow \infty}\limsup\frac{\Sigma(z)\hat{u}(z)}{z}.\]    
		In particular, if $\varrho:=\underset{x\rightarrow \infty}{\lim}\frac{\Sigma(z)\hat{u}(z)}{z}\text{ exists in } [0,\infty]$, then $\underline{\varrho}_{\,\scriptscriptstyle\Sigma, \hat{\Phi}}=\overline{\varrho}_{\,\scriptscriptstyle \Sigma, \hat{\Phi}}=\varrho.$
	\end{enumerate}
\end{lem}
The proof is postponed in Section \ref{sec:appendix2}.
\begin{prop}\label{prop:suffcondrho=0} We work under the assumptions of Lemma~\ref{lem:boundsrho}. Let the tail of the L\'evy measure of $\Sigma$ be denoted by $\bar{\eta}: (0,\infty)\ni u \mapsto \eta([u,\infty))$.
\begin{enumerate}
    \item If $\underset{x\rightarrow 0}{\limsup}\,\frac{x\hat\Phi'(x)}{\hat\Phi(x)^2}<\infty$ then $\rhoup=0$. In particular, this always holds when $\hat{\lambda}:=\hat{\Phi}(0)>0$.
    \item One has $\rhodown\geq \underset{x\rightarrow 0}{\liminf}\,\frac{x\hat\Phi'(2x)}{\hat\Phi(x) \hat\Phi(2x)} \int_0^x\bar{\eta}(u)\ddr u \in [0,\infty]$. 
\end{enumerate}
\end{prop}
An explicit example is provided in the forthcoming Section~\ref{sec:examples}, see Example~\ref{example:propsuffcond}..
\begin{proof}
 We use Lemma~\ref{lem:boundsrho}-(1).   One has, for all $x,u\in (0,\infty)$
\[0<\frac{1}{\hat\Phi(x)}-\frac{1}{\hat{\Phi}(x+u)}\leq u \sup_{y\in [x,x+u]}|\big(1/\hat{\Phi}\big)'|(y)\leq u \frac{x\Phi'(x)}{\Phi(x)^2},\]
where we used the fact that $y\to \big(1/\hat{\Phi}\big)'(y)=\frac{y\hat{\Phi}'(y)}{\hat\Phi^2(y)}$ is decreasing, see \eqref{eq:potentialmeasure}.

Assume $\sup_{[0,1]}\frac{x\hat\Phi'(x)}{\hat\Phi(x)^2}<\infty$, then, recalling that $\int_0^1 u\bar{\eta}(u)\ddr u<\infty$ and that we work under the assumption $\int_0^1\frac{\ddr x}{\hat{\Phi}(x)}<\infty$, so that either $\hat{\lambda}:=\hat\Phi(0)>0$ or
$\frac{x}{\hat{\Phi}(x)} \underset{x\rightarrow 0}{\longrightarrow}\frac{1}{\hat{\Phi}'(0+)}=0$, we have by Lebesgue's theorem
\[\underset{x\rightarrow 0}{\limsup}\,x\int_0^{1}\bar{\eta}(u)\Big(\frac{1}{\hat\Phi(x)}-\frac{1}{\hat{\Phi}(x+u)}\Big)\ddr u=0. \]
The integral from $1$ to $\infty$ is handled easily, since $\int_1^\infty \bar{\eta}(u)\ddr u<\infty$, $\hat{\Phi}\geq 0$ and $\hat{\Phi}'(0+)=\infty$, we have that 
\[x\int_1^{\infty}\bar{\eta}(u)\Big(\frac{1}{\hat\Phi(x)}-\frac{1}{\hat{\Phi}(x+u)}\Big)\ddr u \leq \frac{x}{\hat{\Phi}(x)}\int_1^\infty \bar{\eta}(u)\ddr u \underset{x\rightarrow 0}{\longrightarrow} 0.\]
We now verify that if $\hat{\Phi}(0)>0$ then $\sup_{[0,1]}\frac{x\hat\Phi'(x)}{\hat\Phi(x)^2}<\infty$, so that $\rhoup=0$. Let us check that for any Bernstein function,  $$x\hat{\Phi}'(x)\underset{x\rightarrow 0}{\longrightarrow} 0.$$
Plainly the drift part of $\hat{\Phi}$ will not play a role since $\gamma^{+}x$ goes to $0$ as $x$ goes to $0$. We focus on the jump part. Let $b>0$, for all $x\in (0,\infty)$
\begin{align*}
    x\hat{\Phi}'(x)=x\int_0^{\infty}ue^{-xu}\hat{\nu}(\ddr u)
    &\leq x\int_0^bu\hat{\nu}(\ddr u)+x\int_{b}^{1/x}u\hat{\nu}(\ddr u)+\int_{1/x}^{\infty}xue^{-xu}\bar{\nu}(\ddr u)
    \\
   &\leq  x\int_0^bu\hat{\nu}(\ddr u)+\bar{\hat{\nu}}(b)+\bar{\hat{\nu}}(1/x),
\end{align*}
where we use the facts that $xu\leq 1$ for $u\leq 1/x$ and $xue^{-xu}\leq 1$ for all $x,u\in (0,\infty)$. The first and the third term both vanishes when $x$ goes to $0$, so that
$$\underset{x\rightarrow 0}{\limsup}\, x\hat{\Phi}'(x)\leq \bar{\hat{\nu}}(b).$$
Since $\bar{\hat{\nu}}(b)\underset{b\rightarrow \infty}{\longrightarrow} 0$, we finally get $\underset{x\rightarrow 0}{\lim}x\hat{\Phi}'(x)=0$. We now look at the case when $\hat{\Phi}(0)>0$. Plainly, since $\hat{\Phi}$ is continuous not vanishing at $0$, $\frac{x\hat\Phi'(x)}{\hat\Phi(x)^2}$ goes to $0$ as $x$ goes to $0$, hence $\sup_{[0,1]}\frac{x\hat\Phi'(x)}{\hat\Phi(x)^2}<\infty$ and $\rhoup=0$.

For the second statement, write $\frac{1}{\hat\Phi(x)}-\frac{1}{\hat{\Phi}(x+u)}=\frac{\hat{\Phi}(x+u)-\hat\Phi(x)}{\hat\Phi(x)\hat{\Phi}(x+u)}$. Recall that $\Phi$ is concave, so that $\Phi(x+u)-\Phi(x)\geq u\Phi'(x+u)$. If $u\leq x$, then $x+u\leq 2x$ and since $\Phi'$ is decreasing, $\Phi'(x+u)\geq \Phi'(2x)$. Thus,
\[\frac{1}{\hat\Phi(x)}-\frac{1}{\hat{\Phi}(x+u)}\geq \frac{u\Phi'(2x)}{\hat\Phi(x)\hat{\Phi}(2x)}\]
and \[\rhodown=\underset{x\rightarrow 0}{\liminf}\ x\int_1^{\infty}\bar{\eta}(u)\Big(\frac{1}{\hat\Phi(x)}-\frac{1}{\hat{\Phi}(x+u)}\Big)\ddr u \geq \underset{x\rightarrow 0}{\liminf}\ \frac{x\Phi'(2x)}{\hat\Phi(x)\hat{\Phi}(2x)} \int_1^{x}u\bar{\eta}(u)\ddr u.\]
\end{proof}
\subsection{Proof of Theorem \ref{thm3zero}}
Under the assumption $\int_0^1\frac{\ddr u}{\hat{\Phi}(u)}<\infty$, the following function is well-defined
\begin{equation}\label{eq:G}
	G(x):=\int_{0}^{\infty}e^{-xz}\frac{\Sigma(z)}{z}\hat{U}(\ddr z),\ x\in (0,\infty).
\end{equation}
Moreover, one has 
\begin{equation}
	\rhoup= \underset{x\rightarrow 0}{\limsup} \, xG(x) \text{ and }  
	\rhodown= \underset{x\rightarrow 0}{\liminf}\, xG(x).
\end{equation}

\subsubsection{Reduction to the case with mechanisms $(\Sigma,\hat{\Psi})$}
Recall $X^{\mathrm{m}}$ the solution to \eqref{SDECBDI}, Theorem~\ref{thm:minX} and the quadruplet $(\pi, a, \gamma, \lambda)$ associated to $\Psi$.
\begin{lem}\label{lem:comparison} Let $\epsilon\in (0,1)$. Call $X^{\mathrm{m},(\epsilon)}$ the unique strong solution, with $0$ and $\infty$ absorbing, to
   \begin{equation*}
   Z_t=x+\int_{0}^{t}\sqrt{2\mathrm{a}Z_s}\ddr B_s+\int_{0}^{t}\int_{0}^{Z_{s-}}\int_{(0,\epsilon]}u\bar{\mathcal{N}}(\ddr s,\ddr r, \ddr u)-\gamma_{\epsilon}\int_0^tZ_s\ddr s-\int_0^{t}\hat{\Psi}(Z_s)\ddr s,
   \end{equation*}
   with $\gamma_{\epsilon}:=|\gamma|+\int_\epsilon^{1}u\pi(\ddr u)\in (0,\infty)$ and $B$ and $\mathcal{N}$ are those of Equation \eqref{SDECBDI}. The process $X^{\mathrm{m},(\epsilon)}$ is a minimal $\mathrm{CBDI}(\Sigma^{\epsilon},\hat{\Psi})$ process with subcritical branching mechanism 
   \begin{equation}\label{eq:Sigmaepsilon}\Sigma^{\epsilon}(x):=\mathrm{a}x^2+\gamma_{\epsilon}x+\int_{(0,\epsilon]}(e^{-ux}-1+ux)\pi(\ddr u), \ \ x\in [0,\infty).\end{equation}
   Moreover, for all $x\in (0,\infty)$,
   \[X^{\mathrm{m},(\epsilon)}_t(x)\leq X^{\mathrm{m}}_t(x),\ \text{ for all } t\geq 0\ \text{ a.s..}\]
\end{lem}
\begin{proof}
    Only the comparison statement needs to be explained. Recall the stochastic equation \eqref{SDECBDI} solved by $X^{\mathrm{m}}$, Theorem~\ref{thm:minX}, and decompose the drift and compensated part as follows
    \begin{align}\gamma& \int_{0}^{t}X_s\ddr s+\int_0^{t}\int_{0}^{X_{s-}}\int_0^1 u\Big(\mathcal{N}(\ddr s,\ddr r, \ddr u)-\ddr s\ddr r\pi(\ddr u)\Big)\nonumber \\
    &=\gamma \int_{0}^{t}\!X_s\ddr s+\int_0^{t}\!\int_{0}^{X_{s-}}\!\!\int_0^{\epsilon} \!\!u\bar{\mathcal{N}}(\ddr s,\ddr r, \ddr u)-\int_{\epsilon}^{1}\!\!u\pi(\ddr u)\int_0^{t}\!X_{s-}\ddr s+\int_0^{t}\!\int_{0}^{X_{s-}}\!\int_{\epsilon}^{1}\!u\mathcal{N}(\ddr s, \ddr r, \ddr u).\label{decompositionsde}
    \end{align}
   Plainly, $$\gamma-\int_{\epsilon}^{1}u\pi(\ddr u)\geq -\left(|\gamma|+\int_{\epsilon}^{1}u\pi(\ddr u)\right)=-\gamma_{\epsilon}.$$
   Furthermore, the compensated term over $(0,\epsilon]$ in \eqref{decompositionsde} is the same as that in the equation solved by $X^{\mathrm{m},(\epsilon)}$. The comparison between $X^{\mathrm{m}}$ and $X^{\mathrm{m},(\epsilon)}$ is readily checked: the process $X^{\mathrm{m}}$ indeed has only positive jumps and  the jumps larger than $\epsilon$, which occur along a stochastic integral with finite variation, see the third term in \eqref{decompositionsde}, are not seen by $X^{\mathrm{m},(\epsilon)}$.
\end{proof}
\subsubsection{A Lyapunov function}
\begin{lem}\label{lem:lyapunovg} Consider a (sub)critical branching mechanism $\Sigma$ and a supercritical one, $\hat{\Psi}=\hat{\Sigma}-\hat{\Phi}$, such that $\int_0^1\frac{\ddr u}{\hat\Phi(u)}<\infty$. Set $g(x):=\int_0^{x}\frac{\ddr u}{\hat\Phi(u)}$ for $x\in [0,\infty)$. The function $g$ takes the following Bernstein form
	\begin{equation}\label{eq:gberstein}g(x)=\int_0^{\infty}(1-e^{-xz})\frac{\hat{U}(\ddr z)}{z}, \ x\in [0,\infty),\end{equation} and denoting by $\mathcal{X}$ the generator of the minimal $\mathrm{CBDI}(\Sigma,\hat\Psi)$, see \eqref{genX}, one has
\begin{equation}\label{eq:keyLyapunov2} \mathcal{X}g(x)=1-xG(x)-\eta(x), \ x\in [0,\infty),
\end{equation}
where $G$ is given by \eqref{eq:G} and $\eta(x):=\frac{\hat{\Sigma}(x)}{\hat{\Phi}(x)}\underset{x\rightarrow 0}{\longrightarrow} 0$.
\end{lem}
\begin{proof} Recall that for all $u\in [0,\infty)$, $\frac{1}{\hat{\Phi}(u)}=\int_{0}^{\infty}\hat{U}(\ddr z)e^{-uz}$.  Fubini-Tonelli's theorem entails plainly that, for all $x\in [0,\infty)$, $g(x)=\int_0^x\frac{\ddr u}{\hat{\Phi}(u)}$ can be rewritten as \eqref{eq:gberstein}. 

Recall $\mathcal{X}g(x)=x\mathrm{L}^{\Sigma}g(x)-\hat{\Psi}(x)g'(x), \ x\in [0,\infty)$. We start by studying the drift part. First of all, plainly, for all $x\in [0,\infty)$, \[-\hat{\Psi}(x)g'(x)=\hat{\Phi}(x)g'(x)-\hat{\Sigma}(x)g'(x)=1-\eta(x),\]
with $\eta(x)=\frac{\hat{\Sigma}(x)}{\hat{\Phi}(x)}$.
By assumption $\int_{0}^{1}\frac{\ddr u}{\hat{\Phi}(u)}<\infty$, this ensures that either $\hat{\Phi}(0)>0$ or $\hat{\Phi}'(0+)=\infty$: in the former case, $\frac{\hat{\Sigma}(x)}{\hat{\Phi}(x)}$ goes to $0$ since $\Sigma(0)=0$, in the latter, $\frac{\hat{\Sigma}(x)}{\hat{\Phi}(x)}=\frac{\hat{\Sigma}(x)}{x}\frac{x}{\hat{\Phi}(x)}\underset{x\rightarrow 0}{\longrightarrow}\frac{\hat\Sigma'(0+)}{\Phi'(0+)}=0$,  since   $\hat\Sigma'(0+)\in [0,\infty)$. 

For all $x,h\in [0,\infty)$, by Fubini-Tonelli
    \begin{align}
g(x+h)-g(x)-hg'(x)&=\int_{x}^{x+h}\left(\frac{1}{\hat\Phi(u)}-\frac{1}{\hat\Phi(x)}\right)\ddr u\nonumber\\ 
&=\int_{x}^{x+h}\int_{0}^{\infty}\hat{U}(\ddr z)\big(e^{-uz}-e^{-xz}\big)\ddr u\\
&=\int_{0}^{\infty}\hat{U}(\ddr z)\big(e^{-zx}-e^{-z(x+h)}-he^{-xz}\big)\nonumber\\
&=-\int_{0}^{\infty}e^{-xz}\frac{\hat{U}(\ddr z)}{z}\left(e^{-zh}-1+zh\right).\label{keyrelationthetaU}
\end{align}
We also have  \[g'(x)=\frac{1}{\hat{\Phi}(x)}=\int_0^{\infty}e^{-xz}\hat{U}(\ddr z) \text{ and } g''(x)=\left(\frac{1}{\hat{\Phi}(x)}\right)'=-\int_0^{\infty}ze^{-xz}\hat{U}(\ddr z),\]
so that by combining everything and applying Fubini-Tonelli
\begin{align*}
x\mathrm{L}^{\Sigma}g(x)&=axg''(x)+x\int_0^{\infty}\left(g(x+h)-g(x)-hg'(x)\right)\pi(\ddr h)\\
&=-x\int_{0}^{\infty}e^{-xz}\frac{\hat U(\ddr z)}{z}\Sigma(z)=-xG(x).
\end{align*}
Hence $\mathcal{X}g(x)=1-xG(x)-\eta(x)$ for all $x\in (0,\infty)$.
\end{proof}
\subsubsection{First entrance times and $0$ non-absorbing}\label{subsec:nonabsorptionat0}
We deal first with the convergence of the first entrance times. For all $a\geq 0$, let $\sigma_{a}^{\mathrm{e}0,+}:=\inf\{t>0: X_t^{\mathrm{e}0}\geq a\}$. We also sometimes emphasize on the initial value and write $\sigma_{a}^{\mathrm{e}0,+}(x)$ for the first passage time above $a$ of the extended process started from $x\in [0,\infty]$.
\begin{lem}\label{lem:cvsigma_a+} Under the assumptions of Theorem~\ref{thm2zero}. Let $X^{\mathrm{e}0}$ be the Fellerian extension of $X^{\mathrm{m}}$.  For any decreasing sequence  $(x_n)_{n\geq 1}$   going towards $0$,  on the event $\left\{\underset{n\rightarrow \infty}{\lim}\uparrow \sigma_a^{\mathrm{e}0,+}(x_n)<\infty\right\}$, it holds that 
\[\lim_{n\rightarrow \infty}\uparrow \sigma_a^{\mathrm{e}0,+}(x_n)=\sigma_a^{\mathrm{e}0,+}(0) \text{ a.s..}\]
\end{lem}
\begin{proof}
Denote by $\sigma:=\underset{n\rightarrow \infty}{\lim}\!\uparrow \sigma_a^{\mathrm{e}0,+}(x_n)$. Since $X^{\mathrm{e}0}(x_n)\geq X^{\mathrm{e}0}(x_{n+1})\geq X^{\mathrm{e}0}(0)$ a.s., we have \[\sigma=\lim_{n\rightarrow \infty}\uparrow \sigma_a^{\mathrm{e}0,+}(x_n)\leq \sigma_a^{\mathrm{e}0,+},\ \text{ a.s..}\]
Recall that $X^{\mathrm{e}0}$ is Feller. This guarantees that $X^{\mathrm{e}0}(x_n)\underset{n\rightarrow \infty}{\Longrightarrow} X^{\mathrm{e}0}(0)$, see e.g. \cite[Theorem~2.5]{zbMATH07242423}. The fact that $X^{\mathrm{e}0}_{\sigma_a^{\mathrm{e}0,+}}(x_n)\geq a$, together with \cite[Proposition 2.1, p.337]{zbMATH01834045}, ensures then that a.s. \[X^{\mathrm{e}0}_{\sigma-}(0)\geq a \ \text{ or }\ X^{\mathrm{e}0}_{\sigma}(0)\geq a.\]
Plainly, if $X^{\mathrm{e}0}_{\sigma}(0)\geq a$, we have $\sigma\geq \sigma_a^{\mathrm{e}0,+}(0)$. Assume now $X^{\mathrm{e}0}_{\sigma-}(0)\geq a$. If $X^{\mathrm{e}0}_{\sigma-}(0)>a$, then the process has visited $(a,\infty)$ before $\sigma$, hence $\sigma\geq \sigma_a^{\mathrm{e}0,+}(0)$. If $X^{\mathrm{e}0}_{\sigma-}(0)=a$ then $X^{\mathrm{e}0}$ cannot be in $[0,a)$ at time $\sigma$ since there is no negative jump. In any case, we get $\sigma\geq \sigma_a^{\mathrm{e}0,+}(0)$ a.s. and conclude. 
\end{proof}
The next lemma studies the times of first entrance in intervals of the form $[a,\infty)$ with $a>0$ for the process $X^{\mathrm{m},(\epsilon)}$, introduced in Lemma~\ref{lem:comparison}, whose jumps are bounded by some $\epsilon>0$. In particular, the role of the condition $\rhoup<1$ is revealed in the case of $0$ being inaccessible.
\begin{lem}
Denote by $G^\epsilon$  the  function
given by \eqref{eq:G} associated to $\Sigma^{\epsilon}$. Let $\hat\Psi=\hat\Sigma-\hat\Phi$ with $\hat{\Phi}$ a Bernstein function. Assume 
$$\int_0^1\frac{\ddr u}{\hat{\Phi}(u)}<\infty \text{ and } \rhoupepsilon= \underset{x\rightarrow 0}{\limsup }\, xG^\epsilon(x)<1.$$ 

Let $\sigma^{(\epsilon),+}_{a}$, $\sigma^{(\epsilon),-}_0$ be the first passage time above $a$ of $X^{\mathrm{m},(\epsilon)}$ and its first hitting time of $0$. Then, there exists $x_1>0$ such that for all $a\in (0,x_1]$
\begin{equation}\label{eq:uniformboundleave0}\sup_{x\in (0,a]}\mathbb{E}_x[\sigma^{(\epsilon),+}_{a}\wedge \sigma_0^{(\epsilon),-}]\leq \frac{2}{1-\overline{\rho}_{\, \scriptscriptstyle\Sigma^{\epsilon},\hat\Phi}}\int_0^{a+\epsilon}\frac{\ddr u}{\hat\Phi (u)}.
\end{equation} 
In the setting of an inaccessible boundary $0$, i.e. $\sigma_0^{(\epsilon),-}=\infty$ a.s.,  this provides
\[\underset{x\rightarrow 0}{\lim}\uparrow \mathbb{E}_x[\sigma^{(\epsilon),+}_{a}]=\mathbb{E}_{0+}[\sigma^{(\epsilon),+}_{a}]<\infty.\]
 Under the assumptions of Theorem~\ref{thm2zero}, the Fellerian extension  defined by $$X_t^{\mathrm{e}0,(\epsilon)}(x)=X_t^{\mathrm{m},(\epsilon)}(x),\ \forall x\in (0,\infty) \text{ and } X_t^{\mathrm{e}0,(\epsilon)}(0)=X_t^{\mathrm{m},(\epsilon)}(0+), \ t\in [0,\infty),$$
has $0$ as an entrance boundary.
\end{lem}
\begin{proof}
 Let $\epsilon\in (0,1)$ and $\Sigma^{\epsilon}$ be the mechanism associated to $\pi_{|(0,\epsilon)}$. We first consider the minimal $\mathrm{CBDI}$ process $X^{\mathrm{m},(\epsilon)}$, with mechanisms $(\Sigma^{\epsilon},\hat{\Psi})$. Denote by $\mathcal{X}^{(\epsilon)}$ its generator and  $G^{\epsilon}$ the associated function defined in \eqref{eq:G}. Recall $g(x)=\int_0^{x}\frac{\ddr u}{\hat\Phi(u)}$. By Lemma \ref{lem:lyapunovg}, we have
\begin{align*}
\mathcal{X}^{(\epsilon)} g(x)
&=1-xG^{\epsilon}(x)-\eta(x),\ \ x\in [0,\infty).
\end{align*}
 Recall also that $\eta(x)$ goes to $0$ as $x$ goes to $0$ and that by assumption $\overline{\rho}_{\Sigma^{\epsilon},\hat{\Phi}}<1$. There exists therefore a small enough $x_1\in (0,\infty)$ such that for all $x\in (0,x_1]$,
$$\mathcal{X}^{(\epsilon)} g(x)\geq \frac{1-\overline{\rho}_{\Sigma^{\epsilon}, \hat\Phi}}{2}=:\kappa>0.$$
Since the jumps of $X^{\mathrm{m},(\epsilon)}$ are bounded by $\epsilon$, the overshoot of the process when it enters $[a,\infty)$ is also bounded by $\epsilon$.  The function $g$ being positive continuous increasing, one has for all $a>0$, $$g(X^{\mathrm{m},(\epsilon)}_{\sigma^{+}_{a}\wedge \sigma^{-}_0\wedge t})\leq g(X^{\mathrm{m},(\epsilon)}_{\sigma^{+}_{a}})\leq g(a+\epsilon) \text{ a.s.}.$$ By Itô's formula,  
\begin{equation}\label{eq:themartingalefor0}\mathbb{E}_x\left[g\Big(X^{\mathrm{m},(\epsilon)}_{\sigma^{(\epsilon),+}_{a}\wedge \sigma^{(\epsilon),-}_0\wedge t}\Big)-\int_{0}^{\sigma^{(\epsilon),+}_{a}\wedge \sigma^{(\epsilon),-}_0\wedge t}\mathcal{X}^{(\epsilon)}g\big(X_s^{\mathrm{m},(\epsilon)}\big)\ddr s\right]=g(x)
\end{equation}
thus, for all $a\in (0,x_1]$ and $x\in (0,a]$,
\begin{align*}
\mathbb{E}_x[\sigma^{(\epsilon),+}_{a}\wedge \sigma^{(\epsilon),-}_0]&\leq \frac{1}{\kappa}\left(\mathbb{E}_x[g(X^{\mathrm{m},(\epsilon)}_{\sigma^{(\epsilon),+}_{a}\wedge \sigma^{(\epsilon),-}_0})]-g(x)\right)\leq \frac{g(a+\epsilon)}{\kappa}.\end{align*}
If $\sigma^{(\epsilon),-}_0=\infty$ a.s. then, $$\mathbb{E}_{0+}[\sigma^{(\epsilon),+}_{a}]\leq \frac{g(a+\epsilon)}{\kappa}<\infty$$
and by Lemma~\ref{lem:cvsigma_a+}, we see that $\mathbb{E}_{0}[\sigma^{\mathrm{e}0,(\epsilon),+}_{a}]=\mathbb{E}_{0+}[\sigma^{(\epsilon),+}_{a}]$, with $\sigma^{\mathrm{e}0,(\epsilon),+}_{a}$ the first passage time above $a$ of the process $X^{\mathrm{e}0,(\epsilon)}$. The latter leaves therefore its boundary $0$ ($0$ is an entrance).
\end{proof}
We establish now Theorem \ref{thm3zero}-(i), namely we show that when $\thetaup<1$, the extended process $X^{\mathrm{e}0}$, defined as the limit of $X^{\mathrm{e}0,(n)}$, see Theorem \ref{thm2zero}, has its boundary $0$ non-absorbing. 
\begin{lem}\label{lem:proofthm3zeroi} 
Let $\Psi=\Sigma-\Phi$, $\hat{\Psi}=\hat{\Sigma}-\hat{\Phi}$ be mechanisms satisfying the assumptions of Theorem~\ref{thm2zero} together with $\int_0^{1}\frac{\ddr u}{\hat{\Phi}(u)}<\infty$. If $\rhoup<1$ then there exists $x_1\in (0,\infty)$ such that for all $a\in (0,x_1)$, 
$$\mathbb{E}_0[\sigma^{\mathrm{e}0,+}_{a}]\leq \frac{2}{1-\overline{\rho}_{\Sigma,\hat\Phi}}\int_{0}^{a}\frac{\ddr u}{\hat{\Phi}(u)}<\infty.$$
In particular, for all $a\in (0,x_1)$,  $\sigma^{\mathrm{e}0,+}_{a}<\infty$, $\mathbb{P}_0$-a.s. so that the boundary $0$ is non-absorbing. It is also instantaneous, namely $\sigma_0^{\mathrm{e}0,+}=0$, $\mathbb{P}_0$-a.s.
\end{lem}
\begin{proof}

Recall that by assumption $\mathrm{a}=0$.  Recall $\hat{\Psi}_n=\hat{\Sigma}-\hat{\Phi}_n=\hat{\Sigma}-(\hat{\Phi}+\hat\lambda_n)$, with $(\hat{\lambda}_n)_{n\geq 1}$ strictly positive and decreasing towards $0$. Let $\epsilon>0$ and recall $\Sigma^{\epsilon}$ in \eqref{eq:Sigmaepsilon}. With the assumption $\mathrm{a}=0$, one has $\thetadowndualepsilonn
=\infty>1$ and we work with the sequence of processes $X^{\mathrm{e}0,(\epsilon,n)}$, extension of $\mathrm{CBDI}(\Sigma^{\epsilon},\hat{\Psi}_n)$ at $0$, minimal at $\infty$, provided by Corollary \ref{cor:entrancewithlambda>0}. None of the processes $X^{\mathrm{e}0,(\epsilon, n)}$ gets extinct and they have all their boundary $0$ entrance.  Let $\mathcal{X}^{(\epsilon)}_{n}$ be the generator of $X^{\mathrm{m},(\epsilon,n)}$. Let $g(x):=\int_0^{x}\frac{\ddr u}{\hat\Phi(u)+\hat\lambda_n}$, for $x\in [0,\infty)$. Notice that 
\begin{equation}\label{eq:ineqpotential}\hat U^{\hat\lambda_n}(\ddr z)=\int_{0}^{\infty}e^{-\hat\lambda_n t }\mathbb{P}(\hat{S}_t\in \ddr z)\ddr t\leq \int_{0}^{\infty}\mathbb{P}(\hat{S}_t\in \ddr z)\ddr t=:\hat U(\ddr z),\end{equation}
where $(\hat{S}_t,t\geq 0)$ denotes a subordinator with Laplace exponent $\hat{\Phi}$ started from $0$. Denote by $G^{\epsilon}_{\hat \lambda_n}$ the function \eqref{eq:G} associated to the mechanisms $\Sigma^{\epsilon}$ and $\hat\Phi_n$, by \eqref{eq:ineqpotential}, one has, for all $x\in (0,\infty)$, $xG^{\epsilon}_{\hat \lambda_n}(x)\leq xG^{\epsilon}(x)$ and
by Lemma \ref{lem:lyapunovg},
\begin{align*}
\mathcal{X}^{(\epsilon)}_ng(x)&=1+x\mathrm{L}^{\Sigma^{\epsilon}}g(x)-\frac{\Sigma(x)}{\hat\Phi(x)+\lambda_n}\\
&=1-xG^{\epsilon}_{\hat\lambda_n}(x)-\eta_n(x)\\
&\geq 1-xG^{\epsilon}(x)-\eta(x)\\
&=\mathcal{X}^{(\epsilon)}g(x)
\end{align*}
where for all $n\geq 1$, we set $\eta_n(x):=\frac{\Sigma^{\epsilon}(x)}{\hat\Phi(x)+\lambda_n}\leq\eta(x):=\frac{\Sigma^{\epsilon}(x)}{\hat\Phi(x)}$ and we used in the inequality above that
$$ \frac{\Sigma^{\epsilon}(x)}{\hat\Phi(x)+\lambda_1}\leq \eta_n(x)\leq \frac{\Sigma^{\epsilon}(x)}{\hat\Phi(x)}, \ x\in (0,\infty).$$ 
Both bounds do not depend on $n$ and go to $0$ as $x$ goes to $0$, hence the term $\eta_n$ goes to $0$ uniformly in $n$. 

We now appeal to the assumption $\rhoup<1$. By Lemma \ref{lem:equivpsi}, $\Sigma^{\epsilon}\underset{\infty}{\sim} \Sigma$. Lemma~\ref{lem:boundsrho}-(1) ensures that $\rhoupepsilon:=\underset{x\rightarrow 0}{\limsup}\  xG^{\epsilon}(x)=\rhoup<1$. Therefore, there is $x_1\in (0,\infty)$ such that for all $n\geq 1$, all $a<x_1$, if $x\in(0,a)$, then
\[\mathcal{X}^{(\epsilon)}_ng(x)\geq \frac{1-\rhoup}{2}>0.\]
Plugging this inequality in the identity \eqref{eq:themartingalefor0} for the process $X^{\mathrm{m},(\epsilon,n)}$, we get the following bound (which does not depend on $n$):
$$\mathbb{E}_0[\sigma^{\mathrm{e}0,(\epsilon,n),+}_{a}]\leq \frac{2}{1-\rhoup}\int_{0}^{a+\epsilon}\frac{\ddr u}{\hat{\Phi}(u)}<\infty.$$
Recall, see Step~\ref{step1zero}, 
that almost surely \begin{center} $X^{\mathrm{e}0,(n)}_t\geq X^{\mathrm{e}0,(\epsilon,n)}_t$, \ $\forall t\in [0,\infty)$.
\end{center} Therefore 
$\sigma^{\mathrm{e}0,(n),+}_{a}\leq \sigma^{\mathrm{e}0,(\epsilon,n),+}_{a}$ $\mathbb{P}_0$-a.s. and
$$\mathbb{E}_0[\sigma^{\mathrm{e}0,(n),+}_{a}]\leq \frac{2}{1-\overline{\rho}_{\Sigma,\hat\Phi}}\int_{0}^{a+\epsilon}\frac{\ddr u}{\hat{\Phi}(u)}<\infty.$$
By letting $\epsilon$ go to $0$ in the previous inequality, we get 
$$\mathbb{E}_0[\sigma^{\mathrm{e}0,(n),+}_{a}]\leq \frac{2}{1-\overline{\rho}_{\Sigma,\hat\Phi}}\int_{0}^{a}\frac{\ddr u}{\hat{\Phi}(u)}<\infty.$$
Last, we argue that $\underset{n\rightarrow \infty}{\lim} \uparrow \sigma^{\mathrm{e0},+,(n)}_{a}=\sigma^{\mathrm{e}0,+}_{a}$, $\mathbb{P}_0$-a.s.. Notice that $X^{\mathrm{e}0,(n)}\underset{n\rightarrow \infty}{\Longrightarrow} X^{\mathrm{e}0}$ and 
$\underset{n\rightarrow \infty}{\lim}\! \downarrow X^{\mathrm{e}0,(n)}_t =X^{\mathrm{e}0}_t$ for all $t$ a.s.. The same argument as for establishing Lemma~\ref{lem:cvsigma_a+} can be applied.  Therefore, 
$$\mathbb{E}_0[\sigma^{\mathrm{e}0,+}_{a}]\leq \frac{2}{1-\overline{\rho}_{\Sigma,\hat\Phi}}\int_{0}^{a}\frac{\ddr u}{\hat{\Phi}(u)}<\infty,$$
so that $0$ is non-absorbing for $X^{\mathrm{e}0}$. By letting $a$ go to $0$, we have $\underset{a\rightarrow 0}{\lim}\! \downarrow \sigma^{+}_{a}=T_0 \text{ a.s.}$
with $T_0=\inf\{t>0: X_t^{\mathrm{e0}}(0)>0\}$. We finally see that $\mathbb{E}_{0}(T_{0})=0$, thus $T_{0}=0$, $\mathbb{P}_0$-a.s. and $0$ is instantaneous. The proof is achieved.
\end{proof}
\subsubsection{Zero absorbing}\label{subsec:absorptionatzero}
We explain now that if $\rhodown>1$ then the extended process $X^{\mathrm{e}0}$ is absorbed at $0$. This follows exactly the same arguments as for the boundary $\infty$, see Section \ref{subsec:absorptionatinfinity}. 

Analogously to the case of $\infty$, we construct a supermartingale to show that no continuous extension at $0$. We leave to the reader the adaptation of Lemma \ref{lem:noexcursion} to the case of $0$. Recall $g(x):=\int_{0}^{x}\frac{\ddr u}{\hat{\Phi}(u)}$ for all $x\in (0,\infty)$. 
\begin{lem}\label{lem:supermartingale0} Assume $\int_0^{1}\frac{\ddr u}{\hat\Phi(u)}<\infty$.   Assume $\rhodown>1$, there exists $x_0\in (0,1)$ such that $(g(X^{\mathrm{m},(\epsilon)}_{t\wedge \sigma_{x_0}^{-}}),t\ge 0)$ is a supermartingale.
\end{lem}
\begin{proof}
By Lemma \ref{lem:lyapunovg}, 
$\mathcal{X}g(x)=1-xG(x)-\eta(x)$.  
By Lemma \ref{lem:boundsrho}-(1), $\rhodown=\underset{x\rightarrow 0}{\liminf}\, xG(x)$. Thus, if $\rhodown>1$, there exists $x_0$ such that $\mathcal{X}g(x)\leq 0$ for all $x\geq x_0$.  By Itô's lemma, $$(M_t)_{t\ge 0}:=\left(g(X^{\mathrm{m}}_{t\wedge \sigma_{x_0}^{+}})-\int_{0}^{t\wedge \sigma_{x_0}^{+}}\mathcal{X}g(X^{\mathrm{m}}_s)\ddr s, t\geq 0\right)$$
is a local martingale. The latter is positive and is therefore a supermartingale, as well as the process $\big(g(X^{\mathrm{m}}_{t\wedge \sigma_{x_0}^{+}}), t\geq 0\big)$.
\end{proof}
\noindent \textbf{Proof of Theorem \ref{thm3zero}:} The non-absorption property is obtained by Lemma~\ref{lem:proofthm3zeroi}. For the absorption, we use Lemma \ref{lem:supermartingale0}, together with the fact that $g(x)\underset{x\rightarrow 0}{\longrightarrow} 0$. This implies that no excursion measure can exist. Notice that we work under the conditions of Theorem~\ref{thm2zero} so that the process $X^{\mathrm{e}0}$ is Feller and has $0$ as a continuous boundary. \qed
\section{Behaviors classification and examples}\label{sec:examples}
\subsection{Classification with the crossed parameters $\thetaup$ and $\rhoupdual$}
We sum up here the conditions obtained in Section~\ref{sec:absorptionatinfinity} and Section~\ref{sec:proofofthm3zero} and show how a classification of the boundaries emerges.

We work with fixed mechanisms $\Psi=\Sigma-\Phi$ and $\hat\Psi=\hat\Sigma-\hat\Phi$. Recall \begin{equation}\label{eq:thetauprhoupdual}\thetaup=\underset{x\rightarrow \infty}{\limsup} \, x\int_0^{\infty}e^{-zx}\frac{\Phi(z)\hat{W}(z)}{z}\ddr z, \ 
\rhoupdual:= \underset{x\rightarrow 0}{\limsup} \ x\int_{0}^{\infty}e^{-xz}\frac{\hat\Sigma(z)}{z}U(\ddr z),\end{equation}
and similarly for $\thetadown$ and $\rhodowndual$,  with $\liminf$.
\begin{cor}\label{thm:classification} Assume  $\mathrm{a}=\hat{\mathrm{a}}=0$ and the following conditions
	\begin{equation*}\label{eq:condcdiexplosion}
		\hypertarget{negH1}{\neg \mathbb{H}_1}: \
		\int_{0}^{1}\frac{\ddr u}{\Phi(u)}<\infty, \ \
	\hypertarget{negHhat2}{\neg \hat{\mathbb{H}}_2}: \ \int_1^{\infty}\frac{\ddr u}{\hat{\Sigma}(u)}<\infty\ \text{ and } \ \hat{\mathbb{H}}_1: \ \int_{0}^{1}\frac{\ddr u}{\hat\Phi(u)}=\infty, \ \ \mathbb{H}_2: \ \int_{1}^{\infty}\frac{\ddr u}{\Sigma(u)}=\infty. 
	\end{equation*}
The boundary behaviors of $X^{\mathrm{e}\infty}$, the $\mathrm{CBDI}(\Psi,\hat{\Psi})$ extended at $\infty$, and $Y^{\mathrm{e}0}$, the $\mathrm{CBDI}(\hat{\Psi},\Psi)$ extended at $0$ are classified by Table \ref{classificationtable}.
\begin{table}[h!]
\centering
\setlength{\tabcolsep}{12pt}
\renewcommand{\arraystretch}{1.5}

\begin{tabular}{|c|c|c|}
\hline
\textbf{Condition} & $X^{\mathrm{e}\infty}$ & $Y^{\mathrm{e}0}$ \\
\hline
$\rhodowndual>1>\thetaup$ 
& $\infty$ entrance 
& $0$ exit \\
\hline
$\max(\rhoupdual,\thetaup)<1$ 
& $\infty$ regular 
& $0$ regular \\
\hline
$\thetadown>1>\rhoupdual$ 
& $\infty$ exit 
& $0$ entrance \\
\hline
\end{tabular}
\caption{Boundary classification for $X^{\mathrm{e}\infty}$ and $Y^{\mathrm{e}0}$.}
\label{classificationtable}
\end{table}
\end{cor}

\begin{proof}
This is obtained by combining Theorem \ref{thm3infty} with Theorem~\ref{thm:accessibilityinftybyrho} and Theorem \ref{cor:accessibility0bytheta} with Theorem \ref{thm3zero}.	
\end{proof}

Intuitively, Table \ref{classificationtable} can be understood as follows. For the process $Y^{\mathrm{e}0}$ and a fixed mechanism $\hat{\Sigma}$, we observe that when $\rhodowndual>1>\thetaup$, the cooperation, governed by $\Phi$, is not strong enough to prevent absorption at $0$. A phase transition occurs when ${\scriptstyle\Phi} \mapsto \overline{\varrho}_{\, \scriptscriptstyle\hat\Sigma,\Phi}$  becomes smaller than $1$. When furthermore $\thetaup<1$, $Y^{\mathrm{e}0}$ undergoes local extinctions (i.e. it visits $0$) but is not absorbed there. When the parameter ${\scriptstyle\Phi}\mapsto \thetadown$ becomes larger than $1$, then cooperation is sufficiently strong for the process to start from $0$ and never visit it again.
\smallskip

Symmetrically, for $X^{\mathrm{e}\infty}$, whose large jumps are governed by $\Phi$, competition, driven by $\hat{\Sigma}$, may or may not prevent explosion. A regime in which $\infty$ is regular may exist provided that there are mechanisms such that $\max(\rhoupdual,\thetaup)<1$. Concrete examples are given in the next sections.
\subsection{Asymptotically stable mechanisms}\label{sec:concreteexample}
We focus here on mechanisms which are asymptotically stable near $0$ and $\infty$. This will provide a first regime for which the extended process $X^{\mathrm{e}\Delta}$ has its boundary $\Delta$ regular, $\Delta\in \{0,\infty\}$. 
\begin{prop}\label{prop:examplestable}  Assume $\negHone$, $\negHhattwo$, $\hatHone$  and $\Htwo$. Let $X^{\mathrm{e}\infty}$ and $Y^{\mathrm{e}0}$ be the extended $\mathrm{CBDI}s$ respectively at $\infty$ and $0$ with mechanisms $(\Psi,\hat\Psi)$ and $(\hat\Psi,\Psi)$ satisfying
\[\Phi(y)\underset{y\to 0}{\sim} cy^{\alpha} \text{ and } \hat\Sigma(x)\underset{x\to \infty}{\sim} \hat Cx^{\hat\beta+1},\]
with $c\in (0,\infty), \ \alpha\in (0,1)$ and $\hat C\in (0,\infty), \ \hat\beta\in (0,1)$.
      \begin{enumerate}
            \item If $\hat\beta<1-\alpha$, then $\thetadown=\infty$, $\rhoupdual=0$, and 
        $X^{\mathrm{e}\infty}$ has $\infty$ exit, $Y^{\mathrm{e}0}$ has $0$ entrance.
            \item If $\hat\beta>1-\alpha$, then $\thetadown=0$, $\rhoupdual=\infty$ and $X^{\mathrm{e}\infty}$ has $\infty$ entrance, $Y^{\mathrm{e}0}$ has $0$ exit.
            \item If $\hat\beta=1-\alpha$, then $\thetaup=\thetadown=\dfrac{c}{\hat{C}\Gamma(2-\alpha)}$ and $\rhoupdual
=\rhodowndual=\dfrac{\hat{C}}{c\Gamma(\alpha)}$. Thus,
            \begin{itemize}
                \item if 
                $c/\hat C>\Gamma(2-\alpha)$, $X^{\mathrm{e}\infty}$ has $\infty$ exit, $Y^{\mathrm{e}0}$ has $0$ entrance,
                \item if $1/\Gamma(\alpha)<c/\hat C<\Gamma(2-\alpha)$, then $X^{\mathrm{e}\infty}$ and $Y^{\mathrm{e}0}$ have respectively $\infty$ and $0$ regular,
                \item if $c/\hat C<1/\Gamma(\alpha)$,  then $X^{\mathrm{e}\infty}$ has $\infty$ entrance, $Y^{\mathrm{e}0}$ has $0$ exit.
       
            \end{itemize}
        \end{enumerate}
\end{prop}
The results of Proposition~\ref{prop:examplestable} are reminiscent to \cite[Theorem 3.5]{zbMATH07557537}. They complete some  results previously obtained in  \cite[Examples 2.18, 2.19]{zbMATH07120715} and will be generalized in Section~\ref{sec:regularvarying}.
\begin{proof}
By Lemma \ref{lem:boundstheta} and Lemma \ref{lem:boundsrho}, in order to compute the parameters $\thetaup$ and $\rhodown$, we can consider  the settings
\[\Phi(y)=cy^{\alpha}, \ y\in [0,\infty) \text{ and } \hat\Sigma(x)=\hat Cx^{\hat\beta+1}, \ x\in [0,\infty).\]
Recall the potential measure and the scale function, Section \ref{sec:potentialelements}, $$U(\ddr z)=\frac{1}{c}\frac{z^{\alpha-1}}{\Gamma(\alpha)}\ddr z \text{ and } \hat W(z)=\frac{z^{\hat\beta}}{\Gamma(\hat\beta+1)\hat C}, \ z\in [0,\infty).$$
Thus by Lemma \ref{lem:boundsrho} and Lemma \ref{lem:boundstheta}, 
\[
B(z)=\frac{\hat{\Sigma}(z)u(z)}{z}
=
\frac{\hat{C}}{c\Gamma(\alpha)}\, z^{\hat{\beta}+\alpha-1}
\underset{z \rightarrow \infty}{\longrightarrow}
\thetaup
=\thetadown=
\rhoupdual
=\rhodowndual=
\begin{cases}
\infty, & \text{if } \hat{\beta}>1-\alpha, \\[0.3em]
0, & \text{if } \hat{\beta}<1-\alpha, \\[0.3em]
\dfrac{\hat{C}}{c\Gamma(\alpha)}, & \text{if } \hat{\beta}=1-\alpha.
\end{cases}
\]
and
\[
A(z)=\frac{\Phi(z)\hat{W}(z)}{z}
=
\frac{c}{\hat{C}\Gamma(\hat{\beta}+1)}\, z^{\hat{\beta}+\alpha-1}
\underset{z \rightarrow 0}{\longrightarrow}
\thetaup
=\thetadown=
\begin{cases}
0, & \text{if } \hat{\beta}>1-\alpha, \\[0.3em]
\infty, & \text{if } \hat{\beta}<1-\alpha, \\[0.3em]
\dfrac{c}{\hat{C}\Gamma(2-\alpha)}, & \text{if } \hat{\beta}=1-\alpha.
\end{cases}
\]
We conclude by Theorem \ref{thm:classification}. Notice in particular that $\max(\rhoupdual,\thetaup)<1$ if and only if $1/\Gamma(\alpha)<\frac{c}{\hat{C}}<\Gamma(2-\alpha)$. This case is possible since the inequality $1/\Gamma(\alpha)<\Gamma(2-\alpha)$ is true for all $\alpha\in (0,1)$. This can be readily checked using Euler's reflection formula $$\Gamma(2-\alpha)\Gamma(\alpha)=(1-\alpha)\frac{\pi}{\sin (\pi\alpha)}$$ together with the fact that $\alpha\mapsto (1-\alpha)\pi-\sin(\pi \alpha)$ is strictly decreasing on $(0,1]$.
\end{proof}
\subsection{Regularly-varying mechanisms}\label{sec:regularvarying}
The expressions for $\thetaup$ and $\rhoup$ in \eqref{eq:thetauprhoupdual}, written in terms of Laplace transforms, naturally motivate working within the framework of regularly varying mechanisms, which allows the application of Tauberian theorems. We therefore begin by recalling the relevant definitions. We refer the reader to Bingham et al's book \cite{Bingham87}.

A function $F$, defined on $(0,\infty)$, is said to be regularly varying with index $\alpha$ at $0$ (respectively at  $\infty$) when for all $\lambda\in (0,\infty)$, \[\frac{F(\lambda x)}{F(x)}\rightarrow \lambda^{\alpha}, \ \text{ as }x\rightarrow 0\ (\text{resp. } \infty).\]
When $\alpha=0$, $F$ is called slowly varying. If $F$ is regularly varying with index $\alpha$, then $F(x)=x^{\alpha}L(x)$ for some slowly varying $L$. Notice that if $\alpha>0$ and $F$ is regularly varying at $\infty$, then $\underset{x\rightarrow \infty}{\lim} F(x)=\infty$, and similarly if $\alpha<0$, $F$ is regularly varying at $0$ and $x$ goes to $0$.
\vspace*{2mm}

The following statements are well-known. 
\begin{theostar}\label{theotauberian}\
\begin{enumerate}
\item \textbf{Tauberian theorem}: Let $\ell:(0,\infty)\to (0,\infty)$ be slowly varying at $0$ (respectively, at $\infty$) and $\alpha\in[0,\infty)$ and $U$ be a locally finite measure on $[0,\infty)$. One has
\begin{align*}\int_0^{\infty}e^{-xz}U(\ddr z) \sim x^{-\alpha}\ell(1/x), \ &(x\rightarrow \infty, \text{ resp. } 0)\\ &\Longleftrightarrow
\int_0^{z}U(\ddr v)\sim z^{\alpha}\ell(z)/\Gamma(1+\alpha), \ (z\rightarrow 0, \text{resp. } \infty).
\end{align*}
\item \textbf{Monotone density theorem}: Suppose that $U(\ddr z)=u(z)\ddr z$, where $u:(0,\infty)\to (0,\infty)$ is monotone on some neighbourhood of $0+$ (respectively of $\infty$). If there exists $\alpha>0$ and a function $\ell$ slowly varying at $0$ (respectively, at $\infty$) such that 
\[ \int_0^zu(v)\ddr v\sim z^{\alpha}\ell(z) \ (z\rightarrow 0, \text{resp. }\infty), \text{ then }
u(z)\sim \alpha z^{\alpha-1}\ell(z) \ (z\rightarrow 0, \text{resp. }\infty).\]
\end{enumerate}
\end{theostar}
As a first application of Theorem~\ref{theotauberian}, we study an explicit example showing that, providing no diffusive part ($\mathrm{a}=0$), cooperation may either systematically prevail over natural deaths, or, on the contrary, deaths may be sufficiently strong to make absorption unavoidable.
\begin{example}[drift interaction equivalent to $\log(1/x)^{-\beta}$ near $0$]\label{example:propsuffcond} Assume that $\Sigma$ has no quadratic term.  Let $\beta \in (0,\infty)$. Set  \[\hat{\Phi}(x):=\int_0^{\infty}(1-e^{-xu})\nu(\ddr u), \ \ x\in (0,\infty), \ \text{ with } \nu(\ddr u):=\frac{\mathbbm{1}_{(0,1/e)}(u)}{u(\log u)^{\beta}}\ddr u.\] Then, by setting $\ell(u):=\int_u^{1/e}\frac{1}{v(\log v)}\ddr v$, and observing that $\hat{\Phi}(x)/x=\int_0^{1/e}e^{-xu}\ell(u)\ddr u$, for all $x\in (0,\infty)$, it can be checked, with the help of Theorem~\ref{theotauberian} and by derivating under the integral, that
\[\hat{\Phi}(x)\underset{x\rightarrow 0}{\sim} \frac{1}{\log(1/x)^{\beta}}, \ \text{and} \ \hat\Phi'(x)\underset{x\rightarrow 0}{\sim} \frac{\beta}{x(\log(1/x))^{\beta+1}},\]
 so that  $\frac{x\hat{\Phi}'(x)}{\hat\Phi(x)^2}\underset{x\rightarrow 0}{\sim} \beta (\log(1/x)^{\beta-1}$.
\begin{itemize}
    \item  If $\beta \leq 1$, then $\sup_{[0,1]}\frac{x\hat{\Phi}'(x)}{\hat\Phi(x)^2} <\infty$. We see by Proposition~\ref{prop:suffcondrho=0}-1) that in this case $\rhoup=0$. The boundary $0$ is thus non-absorbing for the $\mathrm{CBDI}(\Psi,\hat{\Psi})$ and $\infty$ is accessible for  the $\mathrm{CBDI}(\hat{\Psi},\Psi)$. 
 \item If $\beta> 1$, by Proposition~\ref{prop:suffcondrho=0}-2), we see that if the L\'evy measure $\eta$ of $\Sigma$, satisfies \[\underset{x\rightarrow 0}{\liminf}\log(1/x)^{\beta-1}\int_0^{x}\bar{\eta}(u)\ddr u>2/\beta,\] then $\rhoup>1$. In this case, $0$ is absorbing for the $\mathrm{CBDI}(\Psi,\hat{\Psi})$ and  $\infty$ is inaccessible for the $\mathrm{CBDI}(\hat{\Psi},\Psi)$.
\end{itemize} 
\end{example}
\smallskip

We now apply more generally Theorem~\ref{theotauberian} for expressing the parameters $\thetaup, \thetadown$ and $\rhoup, \rhodown$.
\begin{lem}\label{lem:ABregularvarying}\
\begin{enumerate}[(1)]
\item If $\hat{\Sigma}$ is regularly varying at $\infty$ with index $1+\beta, \beta\in [0,1]$, that is $\hat\Sigma(x)=x^{1+\beta}\hat{L}(x)$, $x\in [0,\infty)$, then
\[A(z)=\frac{\Phi(z)\hat{W}(z)}{z} \underset{z\rightarrow 0}{\sim}  \frac{1}{\Gamma(1+\beta)} \frac{\Phi(z)}{z^{1-\beta}\hat{L}(1/z)}=\frac{1}{\Gamma(1+\beta)} \frac{\Phi(z)}{z^{2}\Sigma(1/z)}\]
and 
	$$\thetaup=\frac{1}{\Gamma(1+\beta)}\underset{z\rightarrow 0}{\limsup} \frac{\Phi(z)}{z^2\hat\Sigma(1/z)} \text{ and }\thetadown=\frac{1}{\Gamma(1+\beta)}\underset{z\rightarrow 0}{\liminf}\frac{\Phi(z)}{z^2\hat\Sigma(1/z)}.$$
\item If $\hat\Phi$ is regularly varying at $0$ with index $\alpha\in (0,1]$, $\hat\Phi(x)=x^{\alpha}\hat{\ell}(x)$, and its potential measure $\hat{U}$ admits a density $\hat{u}$ that is monotone on some neighbourhood of $\infty$, then
\[B(z)=\frac{\Sigma(z)\hat{u}(z)}{z} \underset{z\rightarrow \infty}{\sim} \frac{1}{\Gamma(\alpha)}\frac{\Sigma(z)}{z^{2-\alpha}\hat{\ell}(1/z)}=\frac{1}{\Gamma(\alpha)}\frac{\Sigma(z)}{z^{2}\hat\Phi(1/z)}\]
and
\[\rhoup=\frac{1}{\Gamma(\alpha)}\underset{z\rightarrow \infty}{\limsup} \frac{\Sigma(z)}{z^2\hat\Phi(1/z)} \text{ and }\rhodown=\frac{1}{\Gamma(\alpha)}\underset{z\rightarrow \infty}{\liminf}\frac{\Sigma(z)}{z^2\hat\Phi(1/z)}.\]
\end{enumerate}
\end{lem}
\begin{proof}
We omit the proof of the first claim and focus on the second. By the Tauberian theorem and the monotone density theorem; we have
$$\int_0^z\hat{u}(v)\ddr v\underset{z\rightarrow \infty}{\sim}z^{\alpha}\hat{\ell}(z)/\Gamma(1+\alpha) \text{ and }\hat{u}(z)\underset{z\rightarrow \infty}{\sim}\frac{\alpha}{\Gamma(1+\alpha)}z^{\alpha}\hat{\ell}(z)=\frac{1}{\Gamma(\alpha)}z^{\alpha}\hat{\ell}(z).$$
 Thus,
$$\frac{\Sigma(z)\hat{u}(z)}{z}\underset{z\rightarrow \infty}{\sim} \frac{1}{\Gamma(\alpha)}\frac{\Sigma(z)}{z^2\hat\Phi(1/z)}.$$
The rest follows plainly.
\end{proof}
\begin{rem} 
	\begin{enumerate}
		\item  
		The question of whether the potential measure $\hat{U}$ admits a density that is monotone in a neighborhood of $0$ or $\infty$ is subtle and far from straightforward. A broad class of examples for which this property holds, and Lemma~\ref{lem:ABregularvarying}-(2) applies, is given by the special Bernstein functions, we refer to \cite[Chapter 2]{zbMATH05841214}.
	\item Cases for which $\thetaup>\thetadown$ and $\rhoup>\rhodown$ may occur when the slowly varying function $\ell$ and $L$ are oscillating near the boundary. Consider for instance $(0,\infty)\ni z\mapsto L(z):=C\exp\left(\sin (\log\log z)\right)$ with $C\in (0,\infty)$ (see \cite{Bingham87}). 
\end{enumerate}
\end{rem}
The next theorem provides the classification in the setting of regularly varying mechanisms. It also
shows that there is no regime in which the boundaries are regular when both functions $\hat{\Sigma}$ and $\Phi$ are regularly varying at $\infty$ and $0$ respectively, with index~$1$.
\begin{theo} Let $\hat\Sigma$ and $\Phi$ be regularly varying mechanisms respectively at $\infty$ and $0$.  Suppose that  the potential measure $U$ associated to $\Phi$ admits a density that is monotone in a neighbourhood of  $\infty$.  Assume that the conditions $\negHone$, $\negHhattwo$, $\hatHone$  and $\Htwo$ hold. Consider $Y^{\mathrm{e}0}$ the extended $\mathrm{CBDI}(\hat\Psi,\Psi)$ at $0$.
\begin{enumerate}
\item If $\hat\Sigma$ and $\Phi$ are regularly varying, at $\infty$ and $0$ respectively, with index $1$, then 
\[\thetaup=\frac{1}{\rhodowndual}=\underset{z\rightarrow 0}{\limsup} \frac{\Phi(z)}{z^2\hat\Sigma(1/z)}, \ \thetadown=\frac{1}{\rhoupdual}=\underset{z\rightarrow 0}{\liminf} \frac{\Phi(z)}{z^2\hat\Sigma(1/z)} .\]
In particular if $\thetaup<1$ then $\rhodowndual>1$, and in this case  $0$ is accessible and absorbing for $Y^{\mathrm{e}0}$ ($0$ is an exit). Similarly, if $\thetadown>1$ then $\rhoupdual<1$ and $0$ is inaccessible and non-absorbing ($0$ is an entrance).
\item Let $\alpha\in (0,1)$. Assume that $\hat{\Sigma}(x)=x^{2-\alpha}\hat{L}(x)$ and $\Phi(x)=x^{\alpha}\ell(x)$ with $\hat{L}$ and $\ell$ slowly varying at $\infty$ and $0$ respectively. Then,

\[\thetaup=\frac{1}{\Gamma(2-\alpha)\Gamma(\alpha)}\frac{1}{\rhodowndual}, \ \thetadown=\frac{1}{\Gamma(2-\alpha)\Gamma(\alpha)}\frac{1}{\rhoupdual}.\]
Define moreover the $[0,\infty]$-valued parameters:
\begin{equation}\label{eq:xiinfsup}\underline{\xi}:=\underset{z\rightarrow 0}{\liminf} \frac{\ell(z)}{\hat{L}(1/z)}, \ \overline{\xi}:=\underset{z\rightarrow 0}{\limsup} \frac{\ell(z)}{\hat{L}(1/z)}.\end{equation}
Then, $\thetaup=\frac{1}{\Gamma(2-\alpha)}\overline{\xi},\  \thetadown=\frac{1}{\Gamma(2-\alpha)}\underline{\xi}$, and $\rhoupdual=\frac{1}{\Gamma(\alpha)\underline{\xi}},\  \rhodowndual=\frac{1}{\Gamma(\alpha)\overline{\xi}}$,
\begin{enumerate}
\item If $\underline{\xi}>\Gamma(2-\alpha)$, $0$ is an entrance,
\item if $\frac{1}{\Gamma(\alpha)}<\underline{\xi}, \overline{\xi}<\Gamma(2-\alpha)$, $0$ is regular,
\item if $\overline{\xi}<\frac{1}{\Gamma(\alpha)}$, $0$ is an exit.
\end{enumerate}
\end{enumerate}
\end{theo}
\begin{proof} The proof is a direct application of Lemma~\ref{lem:ABregularvarying} and Theorem~\ref{thm:classification}.
\end{proof}
Cases of slowly varying functions $\hat L$ and $\ell$ that are oscillating below and above the critical values, i.e. $\underline{\xi}<\frac{1}{\Gamma(\alpha)}<\overline{\xi}$ or $\overline{\xi}>\Gamma(2-\alpha)>\underline{\xi}$, are not covered by our approach and we cannot conclude on the boundary behavior in these settings.

\begin{rem}
We see that in the regularly varying setting, the case for which both parameters satisfy $\thetaup>1$ and $\rhodowndual>1$ is not possible (there is no regime where the boundary would be natural).
\end{rem}
We observe in the next remark that for $\mathrm{CBDI}$ processes with large jumps governed by a slowly varying function $\Phi$ such that $\int_0^1\frac{\ddr u}{\Phi(u)}< \infty$,  a drift competition $\hat{\Sigma}$ with no quadratic part ($\hat{\mathrm{a}}=0$) cannot in general prevent explosion nor absorption at $\infty$.

\begin{rem} 
	 Let $\hat\Sigma$ be regularly varying at $\infty$ with index $1+\beta\in (0,2)$, and $\Phi$ slowly varying at $0$. Then, recall that by Lemma \ref{lem:boundstheta}, $A(z)\asymp\frac{\Phi(z)}{z^{1-\beta}}$, and  since $\beta<1$,  we have that $$\thetadown\geq \underset{z\rightarrow 0}{\liminf}\, A(z)=\infty.$$ 
	This entails that $\infty$ is an absorbing boundary for $X^{\mathrm{e}\infty}$.
	 The study of accessibility, through $\rhodowndual$, is more involved. The monotone density theorem, Theorem~\ref{theotauberian}-2 cannot be used in the slowly varying setting that is when $\alpha=0$. More refined versions of both Tauberian and monotone density theorems are available. We do not pursue full generality here, but note that if $\Phi$ is a smooth special Bernstein function with no drift, see e.g. \cite[Theorem 2.18]{zbMATH05841214}, then $u(z)\underset{z\rightarrow \infty}{\sim} \tilde{\ell}(z)/z$ for some slowly varying $\tilde{\ell}$. Recalling the definition of $B$ from Lemma \ref{lem:boundsrho}, we obtain $$B(z)=\frac{\tilde{\ell}(z)\hat{\Sigma}(z)}{z^2}=z^{\beta-1}\tilde{\ell}(z)\hat{L}(z).$$ Since $\beta<1$, it follows that
	$\rhoupdual\leq \underset{z\rightarrow \infty}{\limsup}B(z)=0.$ So that in this case,  $\infty$ is an exit boundary for $X^{\mathrm{e}\infty}$ and $0$ an entrance for $Y^{\mathrm{e}0}$.    
\end{rem}	
\subsection{Regular-for-itself and non-stickiness properties.}
The properties of non‑stickiness and regularity‑for‑itself are related through Laplace duality.
\begin{table}[h!]
\begin{center}
\begin{tabular}{|c|c|}
\hline
$X^{\mathrm{e}0}$ &  $Y^{\mathrm{e}\infty}$ \\
\hline
$0$  regular non-sticky  &  $\infty$ regular for itself\\
\hline
$0$  regular for itself  &  $\infty$ regular non-sticky\\ 
\hline
\end{tabular}
\caption{Non-sticky/regular-for-itself}
\label{correspondance5}
\end{center}
\end{table}

\begin{prop}\label{prop:regularforitselfreflecting}
   Let $Y^{\mathrm{e}0}$ be the extension of the minimal $\mathrm{CBDI}(\hat\Psi,\Psi)$  at $0$ with the latter regular (Theorem~\ref{thm2zero}). The boundary $0$ is non-sticky for $Y^{\mathrm{e}0}$ if and only if $\infty$ is regular for itself for $X^{\mathrm{e}\infty}$: 
\begin{center}
$\mathbb{P}_{\infty}(R^{\mathrm{e}\infty}=0)=1$ where $R^{\mathrm{e}\infty}:=\inf\{t>0:X_t^{\mathrm{e}\infty}=\infty\}$.     
\end{center}
Similarly, by exchanging the roles of the processes and the boundaries, $\infty$ is non-sticky for $X^{\mathrm{e}\infty}$ if and only if $0$ is regular-for-itself for $Y^{\mathrm{e}0}$. 
\end{prop}
\begin{proof}
Recall the relationship $\mathbb{E}^{0}[e^{-xY_t^{\mathrm{e}0}}]=\mathbb{P}_{x}(\sigma_\infty^{+}>t)$ for all $t\in [0,\infty)$ and $x\in [0,\infty]$, under the convention $0^{+}\cdot \infty, \infty^{-}\cdot 0$. Notice that $\sigma_\infty^{+}$ has the same law as the first return time to $\infty$ under $\mathbb{P}_x$ for all $x\in [0,\infty)$. Therefore, for all fixed $t\in (0,\infty)$
\[\underset{x\rightarrow \infty}{\lim} \mathbb{P}_x(R^{\mathrm{e}\infty}> t)=0 \text{ if and only if } \mathbb{P}^{0}(Y_t^{\mathrm{e}0}=0)=0.\]
Only remains to explain the equivalence between $\underset{x\rightarrow \infty}{\lim} \mathbb{P}_x(R^{\mathrm{e}\infty}> t)=0$ for all $t\in (0,\infty)$ and $\mathbb{P}_\infty(R^{\mathrm{e}\infty}=0)=1$. Let $t,s>0$. By the Markov property at time $s$, the fact that $X^{\mathrm{e}\infty}_s\underset{s\rightarrow 0+}{\longrightarrow}{\infty}$, $\mathbb{P}_\infty$-a.s. (right-continuity) and Lebesgue's theorem, we get
\begin{equation}\label{eq:markovreturntime}\mathbb{P}_\infty(R^{\mathrm{e}\infty}>t+s)=\mathbb{E}_\infty[\mathbb{P}_{X^{\mathrm{e}\infty}_s}(R^{\mathrm{e}\infty}>t)\mathbbm{1}_{\{R^{\mathrm{e}\infty}>s\}}]\underset{s\rightarrow 0+}{\longrightarrow} 0.\end{equation}
We conclude sufficiency by noting that  $\mathbb{P}_\infty(R^{\mathrm{e}\infty}>t)=0$ for all $t\in(0,\infty)$ and thus $R^{\mathrm{e}\infty}=0$, $\mathbb{P}_\infty$-a.s.. Necessity follows by contradiction, using \eqref{eq:markovreturntime} in conjunction with Lebesgue’s theorem.
\end{proof}

We have not found a general criterion for the regular-for-itself property of $\mathrm{CBDIs}$ boundaries. We explain a strategy, see e.g. Kolokoltsov \cite[Proposition 6.3.2, page 281]{zbMATH05875699}, and then apply it to mechanisms $\Psi,\hat{\Psi}$ with regularly varying parts in Theorem~\ref{thm:regularforitselfregularlyvarying}.
\begin{lem}\label{lem:genlyapunovh} 
Let $X$ and $Y$ be positive càdlàg strong Markov processes solving respectively the local martingale problems associated to some operators 
\begin{center}$\big(\mathcal{X},\mathcal{D}_\mathcal{X}\big) \text{ and } \big(\mathcal{Y},\mathcal{D}_\mathcal{Y}\big)$  with $\mathcal{D}_{\mathcal{X}}:=\{f: \mathcal{X}f \text{ is well defined}\}$ and similarly for $Y$.\end{center}

Denote by $\sigma_\infty^{+}$ the first hitting time of $\infty$ for $X$ and by $\tau_0^{-}$ the first hitting time of $0$ for $Y$.

\begin{enumerate}
   \item Assume that there exists $h\in \mathcal{D}_{\mathcal{X}}$, positive not identically $0$, such that $\underset{x\rightarrow \infty}{\lim} h(x)=0$ and for some $x_1,\kappa>0$, \[\mathcal{X}h(x)\leq -\kappa,\ x\in (x_1,\infty),\]
then the following holds
\begin{equation}\label{inftyregularforitself}\underset{x\rightarrow \infty}{\lim} \mathbb{P}_{x}(\sigma_\infty^{+}>t)=0 \text{ for all }t>0.\end{equation}
 \item Assume that there exists $h\in \mathcal{D}_{\mathcal{Y}}$, positive not identically $0$, such that $\underset{x\rightarrow 0}{\lim} h(x)=0$ and for some $\epsilon,\kappa>0$, \[\mathcal{Y}h(y)\leq -\kappa, \ y\in (0,\epsilon),\]
then the following holds
\begin{equation}\label{0regularforitself}\underset{y\rightarrow 0^+}{\lim} \mathbb{P}^{y}(\tau_0^{-}>t)=0 \text{ for all }t>0.\end{equation}
\end{enumerate}
\end{lem}
\begin{proof}
    Both statements are shown along similar arguments. We focus on the second. By the optional stopping theorem at $t\wedge \tau_0^{-}\wedge \tau^+_\epsilon$, for all $y\in (0,\epsilon)$,
    \begin{align*}
        \mathbb{E}^{y}[h(Y_{t\wedge \tau_0^{-}\wedge \tau^+_\epsilon})]&=h(y)+\mathbb{E}^{y}\left[\int_0^{t\wedge \tau^+_\epsilon \wedge \tau_0^{-}}\mathcal{Y}h(Y_s)\ddr s\right]\\
        &\leq h(y)-\kappa \mathbb{E}^{y}[t\wedge \tau_0^{-}\wedge \tau^+_\epsilon].
    \end{align*}
    Since $h\geq 0$ and $\kappa>0$, by letting $t$ go to $\infty$, we get $\mathbb{E}^{y}[\tau_0^{-}\wedge \tau^+_\epsilon]\leq \frac{1}{\kappa}h(y)$.
    Since $h(y)\rightarrow h(0)=0$ as $y$ goes to $0$, we have $    \underset{y\rightarrow 0^+}{\lim} \mathbb{E}^{y}[\tau_0^{-}\wedge \tau^+_\epsilon]=0$ and    plainly, by the Markov inequality, for all $t\in [0,\infty)$
    \begin{equation}\label{eq:cvtau0tauepsilon1} \mathbb{P}^{y}(\tau_0^{-}\wedge \tau^+_\epsilon>t)\underset{y\rightarrow 0^+}{\longrightarrow} 0.\end{equation}
Set $c:=\underset{[\epsilon,\infty)}{\inf}h>0$. Plainly, for all $t\in [0,\infty)$, \[h(y)\geq \mathbb{E}^{y}[h(Y_{\tau_0^{-}\wedge \tau_\epsilon^+\wedge t})]\geq c\mathbb{P}^{y}(\tau_\epsilon^+\wedge t<\tau_0^-).\]
 By letting $t$ go to $\infty$, we get $h(y)\geq c\mathbb{P}^{y}(\tau_\epsilon^+<\tau_0^-)$ and therefore
\begin{equation}\label{eq:cvtau0tauepsilon2}\mathbb{P}^{y}(\tau_\epsilon^+<\tau_0^-)\underset{y\rightarrow 0^+}{\longrightarrow} 0.\end{equation}
Plainly,
\[\mathbb{P}^y(\tau_0^{-}>t)\geq \mathbb{P}^y(\tau_0^{-}>t, \tau_0^-<\tau_\epsilon^+)\]
and by combining \eqref{eq:cvtau0tauepsilon1} and \eqref{eq:cvtau0tauepsilon2},
we see  that the limit, as $y$ goes to $0$,  on the right-hand side above is one and finally that \eqref{0regularforitself} holds. 
\end{proof}
The main result of this section is the following.
\begin{theo}\label{thm:regularforitselfregularlyvarying} Let $\Psi=\Sigma-\Phi$ and $\hat{\Psi}=\hat\Sigma-\hat\Phi$ be mechanisms such that 
\[\Phi(y)=y^{\alpha}\ell(y), \ y\in [0,\infty) \text{ and } \hat\Sigma(x)=x^{2-\alpha}\hat{L}(x), \ x\in [0,\infty) \text{ with } \alpha\in (0,1)\]
where $\ell$ and $\hat{L}$ are slowly varying respectively at $0$ and $\infty$.
Assume 

\[1/\Gamma(\alpha)<\underline{\xi}=\underset{z\rightarrow 0}{\liminf} \frac{\ell(z)}{\hat{L}(1/z)}\leq \overline{\xi}=\underset{z\rightarrow 0}{\limsup} \frac{\ell(z)}{\hat{L}(1/z)}<\Gamma(2-\alpha),\] then :
\begin{enumerate}
    \item If $\Sigma\equiv 0$, $X^{\mathrm{e}\infty}$  has $\infty$ regular-for-itself and  $Y^{\mathrm{e}0}$ has $0$  non-sticky.
    \item If $\hat\Phi\equiv 0$, $Y^{\mathrm{e}0}$ has $0$ regular-for-itself and $X^{\mathrm{e}\infty}$  has $\infty$ non-sticky. 
\end{enumerate}
When
\begin{center}
$\Psi=-\Phi$ and $\hat{\Psi}=\hat{\Sigma}$, \end{center} the boundaries $\infty$ of $X^{\mathrm{e}\infty}$ and $0$ of $Y^{\mathrm{e}0}$ are both regular-for-itself and non-sticky.
\end{theo}
\begin{proof}
    The proof is based on Lemma \ref{lem:genlyapunovh} and  Proposition \ref{prop:regularforitselfreflecting}. 
    \begin{enumerate}
        \item 
    We start by establishing that $X^{\mathrm{e}\infty}$ has its boundary regular-for-itself, namely \eqref{inftyregularforitself}. We look for a function $f$ such that $f$ is $\mathrm{C}^2((0,\infty))$, $f(x)\underset{x\rightarrow \infty}{\longrightarrow} 0$ and for some large enough $x_1$ and $\kappa>0$,  \[\mathcal{X}f(x)\leq -\kappa, \ x\in [x_1,\infty).\]
Recall \eqref{eq:gendual}, $\mathcal{X}\mathrm{e}^{y}(x)=(x\Psi(y)+y\hat{\Psi}(x))e^{-xy}$ for all $x,y\in [0,\infty)$. We make the ansatz
\begin{equation}\label{eq:ansatzf}
	f(x)=\int_0^{\infty}e^{-xy}r(y)\ddr y, \ x\in [0,\infty)\end{equation}
with $r(y):= y^{-\beta} \mathbbm{1}_{(0,1)}(y)$ where $\beta \in (\alpha,1)$. By Fubini's theorem and differentiation under the integral, one can check 
\begin{equation}\label{ineqXJ}
\begin{aligned}
\mathcal{X}f(x)&=\int_{0}^{\infty}\mathcal{X}\mathrm{e}^{y}(x)r(y)\ddr y\\
&=\int_{0}^{\infty}(x\Psi(y)+y\hat{\Psi}(x))e^{-xy}r(y)\ddr y\\
&\leq \int_0^{\infty}\big(\hat{L}(x)yx^{2-\alpha}- xy^{\alpha}\ell(y)\big)e^{-xy}r(y)\ddr y:=J(x), \quad x\in (0,\infty)
\end{aligned}
\end{equation}
where in the last inequality we used that $\Psi(y)=-y^{\alpha}\ell(y)$ and $-y\hat{\Phi}(x)\leq 0$.
Now, recalling $r$, one has for all $x\in(0,\infty)$,
\[
J(x)=:J_{\hat L}(x)-J_\ell(x)=: x^{2-\alpha}\hat{L}(x)\int_0^1 y^{1-\beta} e^{-xy}\ddr y - x \int_0^1 y^\alpha \ell(y)y^{-\beta} e^{-xy}\ddr y.
\]
By the change of variable $t=xy$, one gets
\[
J_{\hat L}(x)
= x^{\beta-\alpha}\hat{L}(x)
\int_0^{ x}t^{1-\beta}e^{-t}\ddr t \underset{x\rightarrow \infty}{\sim}  x^{\beta-\alpha}\hat{L}(x) \Gamma(2-\beta)
\text{ and } 
J_{\ell}(x)
=x^{\beta-\alpha} \int_0^{x} t^{\alpha-\beta}\ell(t/x)e^{-t}\ddr t.
\]
The slow variation of $\ell$ ensures that $\frac{\ell(t/x)}{\ell(1/x)} \underset{x\rightarrow \infty}{\longrightarrow} 1$. Moreover, Potter's bound \cite[Theorem~1.5.6(iii)]{Bingham87} yields that for any $\epsilon\in (0,\infty)$, there is some constant $C>0$ such that for $x$ large enough 
\[ \ \forall t\in (0,\infty),\ \frac{\ell(t/x)}{\ell(1/x)}\leq C(t^\epsilon+t^{-\epsilon}).\]
Lebesgue's theorem entails then 
$J_{\ell}(x)
 \underset{x\rightarrow \infty}{\sim}  x^{\beta-\alpha}\ell(1/x)\Gamma(\alpha-\beta+1).$
By our assumption $\ell(1/x)/\hat{L}(x)$ is bounded near $\infty$ and one can check that $$J(x)\underset{x\rightarrow \infty}{\sim} \Gamma(\alpha-\beta+1) x^{\beta-\alpha}\hat{L}(x)\Big(\frac{\Gamma(2-\beta)}{\Gamma(\alpha-\beta+1)}-\frac{\ell(1/x)}{\hat{L}(x)}\Big).$$ 
The last factor above is negative if and only if $\beta$ can be chosen so that \begin{equation}\label{eqfunction i}i(\beta):=\frac{\Gamma(2-\beta)}{\Gamma(\alpha-\beta+1)}<\underline{\xi}.\end{equation} In that case, since $\beta-\alpha>0$, $x^{\beta-\alpha}\to+\infty$ as $x\to \infty$ and
$\lim_{x\to\infty} J(x) = -\infty$. The inequality \eqref{ineqXJ} would then allow us to conclude. We check that the conditions $\underline{\xi}>1/\Gamma(\alpha)$ ensures that such $\beta$ can be chosen. Recall that $\frac{\ddr }{\ddr \beta } \log \Gamma(\beta)=\psi(\beta)$ with $\psi$ the Digamma function. The latter is a strictly increasing function and since $\alpha-\beta+1<2-\beta$, we get
\[\frac{\ddr }{\ddr \beta } \log i(\beta)=-\psi(2-\beta)+\psi(\alpha-\beta+1)<0.\]
Therefore the function $i$ is a decreasing continuous function with limit $\frac{1}{\Gamma(\alpha)}$ when $\beta \to 1$. 
This ensures that one can choose $\beta$ close enough to $\alpha$ so that $$\underset{x\rightarrow \infty}{\limsup}\, \mathcal{X}f(x)\leq \underset{x\rightarrow \infty}{\lim} J(x)=-\infty.$$ The assumption for applying Lemma \ref{lem:genlyapunovh} is therefore met, $\infty$ is regular-for-itself for $X^{\mathrm{e}\infty}$, hence $0$ is regular non-sticky for $Y^{\mathrm{e}0}$.
\item The proof follows the same arguments. Choose here $h$ of the form:
\begin{equation}
	\label{eq:ansatzh} h(y)=\int_0^\infty(1-e^{-yx})s(x)\ddr x, \ y\in [0,\infty),\end{equation}
    with $s(x)=x^{-\beta}\mathbbm{1}_{[1,\infty)}(x)$ and $\beta\in (1,2-\alpha)$. One has $h(0)=0$, $h\in \mathrm{C}^{2}((0,\infty))$.
In a very similar way as previously, using that $\mathcal{Y}(1-\mathrm{e}_x)(y)=-\mathcal{Y}\mathrm{e}_x(y)$, one has for all $y\in (0,\infty)$,
\begin{align*}
    \mathcal{Y}f(y)&=-\int_1^{\infty}\big(y\hat{\Psi}(x)+x\Psi(y)\big)e^{-xy}x^{-\beta}\ddr x\\
    &\leq \int_1^\infty\big(\ell(y)xy^{\alpha}-\hat{L}(x)yx^{2-\alpha}\big)e^{-xy}x^{-\beta}\ddr x=:K(y),
\end{align*}
with \begin{align*}K(y)&=y^{\alpha+\beta-2}\left(\ell(y)\int_y^{\infty}t^{1-\beta}e^{-t}\ddr t-\int_y^{\infty}t^{2-\alpha-\beta}\hat{L}(t/y)e^{-t}\ddr t\right) \\
&\underset{y\rightarrow 0}{\sim}  -\Gamma(2-\beta)y^{\alpha+\beta-2}\hat{L}(1/y)\Big(\frac{\Gamma(3-\alpha-\beta)}{\Gamma(2-\beta)}-\frac{\ell(y)}{\hat{L}(1/y)}\Big).\end{align*}
We must have this time the last factor positive, that this we must find $\beta$ such that $\frac{\Gamma(3-\alpha-\beta)}{\Gamma(2-\beta)}>\overline{\xi}$. By the same argument as before, one can choose $\beta$ close enough to $1$ so that the inequality is true. Finally, $K(y)\to -\infty$ as $\beta<2-\alpha$, $0$ is regular-for-itself for $Y^{\mathrm{e}0}$ and thus $\infty$ is regular non-sticky for $X^{\mathrm{e}\infty}$.
\end{enumerate}
The last statement claiming that the boundary is both regular-for-itself and non-sticky is the intersection of the two  previously established items.
\end{proof}
	The functions $f$ in \eqref{eq:ansatzf} and $h$ in \eqref{eq:ansatzh}, chosen in the proof of Theorem~\ref{thm:regularforitselfregularlyvarying}, are respectively a Laplace transform (equivalently, a completely monotone function) and a function of Bernstein type, as are all Lyapunov functions considered in this article. This structure plays a central role in our analysis. In this direction, we recall that Laplace duality and the complete monotonicity of the semigroup, namely, the invariance of the class of completely monotone functions under the semigroup, are known to be two sides of the same coin; see \cite[Theorem~3.8]{foucartvidmar2025}.
\appendix
\addtocontents{toc}{\protect\setcounter{tocdepth}{-1}}
\section{Analytical study of $\thetaup, \thetadown$ and $\rhoup, \rhodown$}\label{sec:appendix2}
Recall \begin{equation*}\underline{\theta}_{\,\scriptscriptstyle\Phi, \hat{\Sigma}}:=\underset{x\rightarrow \infty}{\liminf}\, x\int_0^{\infty}\frac{\Phi(z)\hat{W}(z)}{z}e^{-zx}\ddr z\ \text{  and  }\ \rhodown:= \underset{x\rightarrow 0}{\liminf}\, x\int_{0}^{\infty}e^{-xz}\frac{\Sigma(z)}{z}\hat{U}(\ddr z)
\end{equation*}
with $\thetaup$ and $\rhoup$ being the $\limsup$.
\subsection{Proof of Lemma \ref{lem:boundstheta}}
\begin{enumerate}
	\item Recall $\Phi(z)=\gamma^+z+\int_0^\infty(1-e^{-zu})\nu(\ddr u)+\lambda$, with $\nu=\pi_{|[1,\infty)}$. Plainly, 
 \[\frac{\Phi(z)}{z}=   \gamma^+ + \int_0^\infty e^{-zu}\bar{\nu}(u)\ddr u+\lambda \int_0^\infty e^{-zu}\ddr u.\]
 Since $\int_0^{\infty}\frac{\ddr x}{\hat{\Sigma}(x)}=\hat{W}(x)$, see Section~\ref{sec:scalefunction}, by Fubini-Tonelli,
 \[ x\int_0^{\infty}\frac{\Phi(z)\hat{W}(z)}{z}e^{-zx}\ddr z=\gamma^+\frac{x}{\hat{\Sigma}(x)}+\int_{0}^{\infty}\frac{\bar{\nu}(u)+\lambda}{\hat{\Sigma}(x+u)}\ddr u.\]
 The first term vanishes as $x$ goes to $\infty$ since we work under the assumption $\int_1^{\infty}\frac{\ddr u}{\hat{\Sigma}(u)}<\infty$.
 \item Let $\Phi_1$ such that $\Phi_1\sim \Phi$ at $0$. For any $\epsilon\in (0,1)$, there exists $z_0\in (0,\infty)$ such that  $\Phi(z)\geq (1-\epsilon)\Phi_1(z)$ for all $z\in (0,z_0)$. Hence
 \[ x\int_0^{\infty}\frac{\Phi(z)\hat{W}(z)}{z}e^{-zx}\ddr z\geq (1-\epsilon)x\int_0^{z_0}\frac{\Phi_1(z)\hat{W}(z)}{z}e^{-zx}\ddr z+x\int_{z_0}^{\infty}\frac{\Phi(z)\hat{W}(z)}{z}e^{-zx}\ddr z.\]
 Since $z\mapsto \frac{\Phi(z)}{z}$ decreases,  $$x\int_{z_0}^{\infty}\frac{\Phi(z)\hat{W}(z)}{z}e^{-zx}\ddr z\leq \frac{\Phi(z_0)}{z_0} x\int_{z_0}^{\infty}\hat{W}(z)e^{-zx}\ddr z,$$
 the upper bound vanishes since $\hat{W}(0+)=\hat{W}(0)=0$, see Section~\ref{sec:scalefunction}. For the same reason, $$\underset{x\rightarrow \infty}{\lim}x\int_{z_0}^{\infty}\frac{\Phi(z)\hat{W}(z)}{z}e^{-zx}\ddr z=0,$$
 and by combining all these facts, we see that
 $$\underline{\theta}_{\,\scriptscriptstyle\Phi, \hat{\Sigma}}:=\underset{x\rightarrow \infty}{\liminf}\, x\int_0^{\infty}\frac{\Phi(z)\hat{W}(z)}{z}e^{-zx}\ddr z\geq (1-\epsilon)\underline{\theta}_{\,\scriptscriptstyle\Phi_1, \hat{\Sigma}}.$$
 Since $\epsilon$ is arbitrary, we have $\underline{\theta}_{\,\scriptscriptstyle\Phi, \hat{\Sigma}}\geq \underline{\theta}_{\,\scriptscriptstyle\Phi_1, \hat{\Sigma}}$. The same argument, picking $z_0$ small enough so that 
  $\Phi(z)\leq (1+\epsilon)\Phi_1(z)$ for all $z\in (0,z_0)$ will lead to  $\underline{\theta}_{\,\scriptscriptstyle\Phi, \hat{\Sigma}}\leq \underline{\theta}_{\,\scriptscriptstyle\Phi_1, \hat{\Sigma}}$.
  \item Let $A(z)=\frac{\Phi(z)\hat{W}(z)}{z}$ for $z\in (0,\infty)$,  the identities follow readily \[\thetadown=\underset{x\rightarrow \infty}{\liminf}\,\mathbb{E}[A(\mathbbm{e}_x)], \quad \thetaup=\underset{x\rightarrow \infty}{\limsup}\,\mathbb{E}[A(\mathbbm{e}_x)],\]
where $\mathbbm{e}_x:=\mathbbm{e}/x$ is an exponential r.v. with parameter $x$. Since $\mathbbm{e}/x$ converges towards $0$ as $x$ goes to $\infty$, we have by Fatou's lemma 
\[\liminf_{z\rightarrow 0}A(z)\leq \underline{\theta}_{\,\scriptscriptstyle\Phi, \hat{\Sigma}}\leq \overline{\theta}_{\,\scriptscriptstyle\Phi, \hat{\Sigma}}\leq \limsup_{z\rightarrow 0} A(z).\]    
In particular if $\theta:=\underset{z\rightarrow 0}{\lim}\frac{\Phi(z)\hat{W}(z)}{z}\text{ exists in } [0,\infty]$, then $$\theta=\underline{\theta}_{\,\scriptscriptstyle\Phi, \hat{\Sigma}}=\overline{\theta}_{\,\scriptscriptstyle\Phi, \hat{\Sigma}}.$$
Last,  $\hat{W}(z)\asymp \frac{1}{z\hat{\Sigma}(1/z)}$, see Section~\ref{sec:scalefunction}, entails that
$A(z)\asymp \frac{\Phi(z)}{z^2\hat{\Sigma}(1/z)}$.
\end{enumerate}
\subsection{Proof of Lemma \ref{lem:boundsrho}}

\begin{enumerate}
\item Denote the drift and the L\'evy measure of $\Sigma$ by $\gamma^-, \eta$. Recall that $\mathrm{a}=0$. One has
\[\frac{\Sigma(z)}{z}=\gamma^{-}+\int_{0}^{\infty}(1-e^{-zu})\bar{\eta}(u)\ddr u, \ \ z\in (0,\infty).\]
For all $x\in (0,\infty)$,
\begin{align*}
    x\int_0^\infty e^{-xz}\frac{\Sigma(z)}{z}\hat{U}(\ddr z)&=\gamma^-\frac{x}{\hat{\Phi}(x)}+x\int_0^{\infty}e^{-xz}\int_0^\infty(1-e^{-uz})\bar{\eta}(u)\ddr u\ \hat{U}(\ddr z)\\
    &=\gamma^-\frac{x}{\hat{\Phi}(x)}+x\int_0^{\infty}\bar{\eta}(u)\Big(\frac{1}{\hat\Phi(x)}-\frac{1}{\hat{\Phi}(x+u)}\Big)\ddr u. 
\end{align*}
The first term vanishes as $x$ goes to $0$ and the claim follows.
	\item Let $z_0\in (0,\infty)$, one has
	\[\int_0^{z_0}e^{-xz}\frac{\Sigma(z)}{z}\hat{U}(\ddr z)\leq \frac{\Sigma(z_0)}{z_0}\hat{U}([0,z_0])<\infty.\]
	Therefore, $\underset{x\rightarrow 0}{\limsup}\, x\int_0^{z_0}e^{-xz}\frac{\Sigma(z)}{z}\hat{U}(\ddr z)=0$ and if $\Sigma\underset{\infty}{\sim} \Sigma_1$; then for all $\epsilon\in (0,1)$, there is $z_0$ such that for all $z\geq z_0$, $(1-\epsilon)\Sigma_1(z)\leq \Sigma(z)\leq (1+\epsilon)\Sigma_1(z)$ and 
    one has
	\begin{align*} xG(x)&\leq x\int_0^{z_0}e^{-xz}\frac{\Sigma(z)}{z}\hat{U}(\ddr z)+(1+\epsilon)x\int_{z_0}^{\infty}e^{-xz}\frac{\Sigma_1(z)}{z}\hat{U}(\ddr z),\\
		xG(x)&\geq x\int_0^{z_0}e^{-xz}\frac{\Sigma(z)}{z}\hat{U}(\ddr z)+(1-\epsilon)x\int_{z_0}^{\infty}e^{-xz}\frac{\Sigma_1(z)}{z}\hat{U}(\ddr z).
	\end{align*}
	Thus, $(1-\epsilon)\overline{\varrho}_{\,\scriptscriptstyle\Sigma_1, \hat{\Phi}}\leq \rhoup\leq (1+\epsilon)\overline{\varrho}_{\,\scriptscriptstyle\Sigma_1, \hat{\Phi}}$ and since $\epsilon$ is arbitrarily small, $\overline{\varrho}_{\,\scriptscriptstyle\Sigma_1, \hat{\Phi}}=\rhoup$ and similarly for $\rhodown$.
	\item Recall Remark~\ref{rem:nolossofgen}. For any $c\in (0,\infty)$, the $\mathrm{CBDI}(\Psi,\hat{\Psi})$ has the same law as the $\mathrm{CBDI}(\Psi_{c},\hat{\Psi}_{-c})$, with $\Psi_c(x)=cx+\Psi(x)$ and $\hat{\Psi}_c(x)=\hat{\Psi}(x)-cx$. Thus, we can always assume that the cooperation part $\hat{\Phi}$ has a drift component. This entails that $\hat{U}$ has a density, see Section \ref{sec:potentialelements}. The expressions of $\rhodown$ and $\rhoup$ with $B(\mathbbm{e}_x)$ follows clearly from  the form of $xG(x)$, see \eqref{eq:G}. The bounds are then obtained by applying Fatou's lemma.
\end{enumerate}
\textbf{Acknowledgements.}
The authors are supported  by the European Union (ERC, SINGER, 101054787). Views and opinions expressed are however those of the authors only and do not necessarily reflect those of the European Union or the European Research Council. Neither the European Union nor the granting authority can be held responsible for them.

\bibliographystyle{abbrv}
\bibliography{doku}

@article{zbMATH05070581,
	author = {Fitzsimmons, Patrick J.},
	title = {On the existence of recurrent extensions of self-similar {Markov} processes},
	fjournal = {Electronic Communications in Probability},
	journal = {Electron. Commun. Probab.},
	issn = {1083-589X},
	volume = {11},
	pages = {230--241},
	year = {2006},
	language = {English},
	doi = {10.1214/ECP.v11-1222},
	url = {https://eudml.org/doc/127244},
	zbMATH = {5070581},
	Zbl = {1110.60036}
}

@incollection{zbMATH05919793,
 author = {Kurtz, Thomas G.},
 title = {Equivalence of stochastic equations and martingale problems},
 booktitle = {Stochastic analysis 2010. Selected papers based on the presentations at the 7th congress of the International Society for Analysis, its Applications and Computations, London, GB, July 2009},
 isbn = {978-3-642-15357-0; 978-3-642-15358-7},
 pages = {113--130},
 year = {2011},
 publisher = {Berlin: Springer},
 language = {English},
 doi = {10.1007/978-3-642-15358-7_6},
 zbMATH = {5919793},
 Zbl = {1236.60073}
}

@book{zbMATH01834045,
 author = {Jacod, Jean and Shiryaev, Albert N.},
 title = {Limit theorems for stochastic processes.},
 edition = {2nd ed.},
 fseries = {Grundlehren der Mathematischen Wissenschaften},
 series = {Grundlehren Math. Wiss.},
 issn = {0072-7830},
 volume = {288},
 isbn = {3-540-43932-3},
 year = {2003},
 publisher = {Berlin: Springer},
 language = {English},
 zbMATH = {1834045},
 Zbl = {1018.60002}
}

@article{CARLOS201726,
title = {General population growth models with {A}llee effects in a random environment},
journal = {Ecological Complexity},
volume = {30},
pages = {26-33},
year = {2017},
note = {Dynamical Systems In Biomathematics},
issn = {1476-945X},
doi = {https://doi.org/10.1016/j.ecocom.2016.09.003},
url = {https://www.sciencedirect.com/science/article/pii/S1476945X16300745},
author = {Clara Carlos and Carlos A. Braumann},
}

@article{zbMATH07557537,
 author = {Foucart, Cl{\'e}ment and Zhou, Xiaowen},
 title = {On the explosion of the number of fragments in simple exchangeable fragmentation-coagulation processes},
 fjournal = {Annales de l'Institut Henri Poincar{\'e}. Probabilit{\'e}s et Statistiques},
 journal = {Ann. Inst. Henri Poincar{\'e}, Probab. Stat.},
 issn = {0246-0203},
 volume = {58},
 number = {2},
 pages = {1182--1207},
 year = {2022},
 language = {English},
 doi = {10.1214/21-AIHP1191},
 zbMATH = {7557537},
 Zbl = {1492.60244}
}

@misc{zbMATH03052578,
 author = {L{\'e}vy, Paul},
 title = {Processus stochastiques et mouvement brownien. {Suivi} d'une note de {M}. {Lo{\`e}ve}},
 year = {1948},
 language = {French},
 howpublished = {Paris: {Gauthier}-{Villars}, {\'E}diteur 365 p. (1948).},
 zbMATH = {3052578},
 Zbl = {0034.22603}
}

@article{zbMATH06873684,
 author = {Casanova, Adri{\'a}n Gonz{\'a}lez and Span{\`o}, Dario},
 title = {Duality and fixation in {{\(\Xi\)}}-{Wright}-{Fisher} processes with frequency-dependent selection},
 fjournal = {The Annals of Applied Probability},
 journal = {Ann. Appl. Probab.},
 issn = {1050-5164},
 volume = {28},
 number = {1},
 pages = {250--284},
 year = {2018},
 language = {English},
 doi = {10.1214/17-AAP1305},
 zbMATH = {6873684},
 Zbl = {1391.92037}
}

@book{bernstein,
  title={Bernstein Functions: Theory and Applications},
  author={Schilling, Ren\'e  and Song, Renming and Vondra\v{c}ek, Zoran},
  isbn={9783110269338},
  lccn={2012006137},
  series={De Gruyter Studies in Mathematics},
  url={https://books.google.si/books?id=QVmHfZZQ0k0C},
  year={2012},
  publisher={De Gruyter}
}

@book{zbMATH05875699,
 title = {Markov processes, semigroups and generators.},
 author = {Kolokoltsov, Vassili N.},
 issn = {0179-0986},
 isbn = {978-3-11-025010-7; 978-3-11-025011-4},
 series = {De Gruyter Studies in Mathematics},
 year = {2011},
 publisher = {de Gruyter}
 }

@article{zbMATH07553689,
 author = {Le, Vi and Pardoux, Etienne},
 title = {Extinction time and the total mass of the continuous-state branching processes with competition},
 fjournal = {Stochastics},
 journal = {Stochastics},
 issn = {1744-2508},
 volume = {92},
 number = {6},
 pages = {852--875},
 year = {2020},
 language = {English},
 doi = {10.1080/17442508.2019.1677661},
 zbMATH = {7553689},
 Zbl = {1490.60231}
}

@article{zbMATH07734715,
 author = {Foucart, Cl{\'e}ment and Zhou, Xiaowen},
 title = {On the boundary classification of {{\(\Lambda\)}}-{Wright}-{Fisher} processes with frequency-dependent selection},
 fjournal = {Annales Henri Lebesgue},
 journal = {Ann. Henri Lebesgue},
 issn = {2644-9463},
 volume = {6},
 pages = {493--539},
 year = {2023},
 language = {English},
 doi = {10.5802/ahl.170},
 zbMATH = {7734715},
 Zbl = {1521.60053}
}

@book {MR1398879,
    AUTHOR = {Durrett, Richard},
     TITLE = {Stochastic calculus},
    SERIES = {Probability and Stochastics Series},
      NOTE = {A practical introduction},
 PUBLISHER = {CRC Press, Boca Raton, FL},
      YEAR = {1996},
     PAGES = {x+341},
      ISBN = {0-8493-8071-5},
   MRCLASS = {60H05 (60J60 60J65)},
  MRNUMBER = {1398879},
MRREVIEWER = {S. Ramasubramanian},
}

@book{Bingham87,
   title = "Regular variation",
   author = "Bingham, N. H. and Goldie, C. M. and Teugels, Jef L.",
   series = "Encyclopedia of mathematics and its applications",
   publisher = "Cambridge University Press",
   address = "Cambridge, Cambridgeshire, New York",
   url = "http://opac.inria.fr/record=b1086118",
   isbn = "0-521-30787-2",
   note = "Includes indexes",
   year = 1987
}

@article{zbMATH02072698,
 author = {Etheridge, A. M.},
 title = {Survival and extinction in a locally regulated population},
 fjournal = {The Annals of Applied Probability},
 journal = {Ann. Appl. Probab.},
 issn = {1050-5164},
 volume = {14},
 number = {1},
 pages = {188--214},
 year = {2004},
 language = {English},
 doi = {10.1214/aoap/1075828051},
 zbMATH = {2072698},
 Zbl = {1043.92030}
}

@book{EthierKurtz,
    AUTHOR = {Ethier, Steven N. and Kurtz, Thomas G.},
     TITLE = {Markov {P}rocesses. Characterization and  {C}onvergence},
    SERIES = {Wiley Series in Probability and Mathematical Statistics:
              Probability and Mathematical Statistics},
 PUBLISHER = {John Wiley \& Sons Inc.},
   ADDRESS = {New York},
      YEAR = {1986}
      }

@Article{zbMATH06836271,
 Author = {{Palau}, Sandra  and  {Pardo}, Juan Carlos},
 Title = {{Branching processes in a L\'evy random environment}},
 FJournal = {{Acta Applicandae Mathematicae}},
 Journal = {{Acta Appl. Math.}},
 ISSN = {0167-8019},
 Volume = {153},
 Number = {1},
 Pages = {55--79},
 Year = {2018},
 Publisher = {Springer Netherlands, Dordrecht},

 DOI = {10.1007/s10440-017-0120-7},
 MSC2010 = {60G17 60G51 60J80},
 Zbl = {1380.60043}
}

@article{arXiv:2410.07664,
 author = {Baguley, Samuel and D{\"o}ring, Leif and Shi, Quan},
 title = {The structure of entrance and exit at infinity for time-changed {L{\'e}vy} processes},
 year = {2024},
 journal = {Preprint, {arXiv}:2410.07664 [math.{PR}] (2024)},
 url = {https://arxiv.org/abs/2410.07664},
 arXiv = {arXiv:2410.07664}
}

@article{zbMATH05232980,
 author = {Rivero, V{\'{\i}}ctor},
 title = {Recurrent extensions of self-similar {Markov} processes and {Cram{\'e}r}'s condition. {II}},
 fjournal = {Bernoulli},
 journal = {Bernoulli},
 issn = {1350-7265},
 volume = {13},
 number = {4},
 pages = {1053--1070},
 year = {2007},
 language = {English},
 doi = {10.3150/07-BEJ6082},
 zbMATH = {5232980},
 Zbl = {1132.60056}
}

@article{zbMATH02209766,
 author = {Rivero, V{\'{\i}}ctor},
 title = {Recurrent extensions of self-similar {Markov} processes and {Cram{\'e}r}'s condition},
 fjournal = {Bernoulli},
 journal = {Bernoulli},
 issn = {1350-7265},
 volume = {11},
 number = {3},
 pages = {471--509},
 year = {2005},
 language = {English},
 doi = {10.3150/bj/1120591185},
 zbMATH = {2209766},
 Zbl = {1077.60055}
}

@article{zbMATH06245578,
 author = {D{\"o}ring, Leif and Barczy, M{\'a}ty{\'a}s},
 title = {Jump type {SDEs} for self-similar processes},
 fjournal = {Electronic Journal of Probability},
 journal = {Electron. J. Probab.},
 issn = {1083-6489},
 volume = {17},
 pages = {39},
 note = {Id/No 94},
 year = {2012},
 language = {English},
 doi = {10.1214/EJP.v17-2402},
 zbMATH = {6245578},
 Zbl = {1286.60036}
}

@article{zbMATH07226359,
 author = {D{\"o}ring, Leif and Kyprianou, Andreas E.},
 title = {Entrance and exit at infinity for stable jump diffusions},
 fjournal = {The Annals of Probability},
 journal = {Ann. Probab.},
 issn = {0091-1798},
 volume = {48},
 number = {3},
 pages = {1220--1265},
 year = {2020},
 language = {English},
 doi = {10.1214/19-AOP1389},
 zbMATH = {7226359},
 Zbl = {1469.60211}
}

@article{zbMATH07453013,
 author = {Ma, Shaojuan and Yang, Xu and Zhou, Xiaowen},
 title = {Boundary behaviors for a class of continuous-state nonlinear branching processes in critical cases},
 fjournal = {Electronic Communications in Probability},
 journal = {Electron. Commun. Probab.},
 issn = {1083-589X},
 volume = {26},
 pages = {10},
 year = {2021},
 language = {English},
 doi = {10.1214/21-ECP374},
 zbMATH = {7453013},
 Zbl = {1493.60127}
}

@Article{zbMATH07317338,
 Author = {H. {Leman} and J. C. {Pardo}},
 Title = {{Extinction and coming down from infinity of continuous-state branching processes with competition in a L\'evy environment}},
 FJournal = {{Journal of Applied Probability}},
 Journal = {{J. Appl. Probab.}},
 ISSN = {0021-9002},
 Volume = {58},
 Number = {1},
 Pages = {128--139},
 Year = {2021},
 Publisher = {Applied Probability Trust, Sheffield},

 DOI = {10.1017/jpr.2020.77},
 MSC2010 = {60J80 60J76 60J85}
}

@article{duality,
author = {Sabine Jansen and Noemi Kurt},
title = {{On the notion(s) of duality for Markov processes}},
volume = {11},
journal = {Probability Surveys},
publisher = {Institute of Mathematical Statistics and Bernoulli Society},
pages = {59--120},
year = {2014},
doi = {10.1214/12-PS206},
URL = {https://doi.org/10.1214/12-PS206}
}

@article{zbMATH08062150,
	author = {Berzunza Ojeda, Gabriel and Pardo, Juan Carlos},
	title = {Branching processes with pairwise interactions},
	fjournal = {ALEA. Latin American Journal of Probability and Mathematical Statistics},
	journal = {ALEA, Lat. Am. J. Probab. Math. Stat.},
	issn = {1980-0436},
	volume = {22},
	number = {1},
	pages = {711--748},
	year = {2025},
	language = {English},
	doi = {10.30757/ALEA.v22-28},
	zbMATH = {8062150},
	Zbl = {1569.60141}
}

@article{zbMATH07458586,
 author = {Gonz{\'a}lez Casanova, Adri{\'a}n and Pardo, Juan Carlos and P{\'e}rez, Jos{\'e} Luis},
 title = {Branching processes with interactions: subcritical cooperative regime},
 fjournal = {Advances in Applied Probability},
 journal = {Adv. Appl. Probab.},
 issn = {0001-8678},
 volume = {53},
 number = {1},
 pages = {251--278},
 year = {2021},
 language = {English},
 doi = {10.1017/apr.2020.59},
 zbMATH = {7458586},
 Zbl = {1490.60228}
}

@article {MR3940763,
    AUTHOR = {Foucart, Cl\'{e}ment},
     TITLE = {Continuous-state branching processes with competition: duality
              and reflection at infinity},
   JOURNAL = {Electron. J. Probab.},
  FJOURNAL = {Electronic Journal of Probability},
    VOLUME = {24},
      YEAR = {2019},
     PAGES = {no. 33, 1--38},
      ISSN = {1083-6489},
   MRCLASS = {60J80 (60J70 92D25)},
  MRNUMBER = {3940763},
MRREVIEWER = {B. L. Granovsky},
       DOI = {10.1214/19-EJP299},
       URL = {https://doi.org/10.1214/19-EJP299},
}

@book{zbMATH02001586,
 author = {Harris, Theodore E.},
 title = {The theory of branching processes.},
 edition = {Corrected reprint of the 1963 original},
 isbn = {0-486-49508-6},
 year = {2002},
 publisher = {Mineola, NY: Dover Publications},
 language = {English},
 zbMATH = {2001586},
 Zbl = {1037.60001}
}

@book {MR3496029,
    AUTHOR = {Pardoux, \'Etienne},
     TITLE = {Probabilistic models of population evolution},
    SERIES = {Mathematical Biosciences Institute Lecture Series. Stochastics
              in Biological Systems},
    VOLUME = {1},
      NOTE = {Scaling limits, genealogies and interactions},
 PUBLISHER = {Springer, [Cham]; MBI Mathematical Biosciences Institute, Ohio
              State University, Columbus, OH},
      YEAR = {2016},
     PAGES = {viii+125},
      ISBN = {978-3-319-30326-0; 978-3-319-30328-4},
   MRCLASS = {60J80 (60F17 60H10 60J85 92D10 92D25)},
  MRNUMBER = {3496029},
MRREVIEWER = {Chao Zhu},
       URL = {https://doi.org/10.1007/978-3-319-30328-4},
}

@book{Bertoin96,
   title = "{L}\'{e}vy processes",
   author = "Bertoin, Jean",
   series = "Cambridge tracts in mathematics",
   publisher = "Cambridge University Press",
   address = "Cambridge, New York",
   url = "http://opac.inria.fr/record=b1080125",
   isbn = "0-521-56243-0",
   year = 1996
}

@article {MR0410961,
    AUTHOR = {Bingham, N. H.},
     TITLE = {Continuous branching processes and spectral positivity},
   JOURNAL = {Stochastic Processes Appl.},
  FJOURNAL = {Stochastic Processes and their Applications},
    VOLUME = {4},
      YEAR = {1976},
    NUMBER = {3},
     PAGES = {217--242},
      ISSN = {0304-4149},
   MRCLASS = {60J80},
  MRNUMBER = {0410961},
MRREVIEWER = {E. Seneta},
}

@article{MR0431386,
    AUTHOR = {Siegmund, David},
     TITLE = {The equivalence of absorbing and reflecting barrier problems
              for stochastically monotone {M}arkov processes},
   JOURNAL = {Ann. Probability},
    VOLUME = {4},
      YEAR = {1976},
    NUMBER = {6},
     PAGES = {914--924},
   MRCLASS = {60J05 (60J25)},
MRREVIEWER = {D. J. Daley},
}

@article{zbMATH07493833,
 author = {Foucart, Cl{\'e}ment},
 title = {A phase transition in the coming down from infinity of simple exchangeable fragmentation-coagulation processes},
 fjournal = {The Annals of Applied Probability},
 journal = {Ann. Appl. Probab.},
 issn = {1050-5164},
 volume = {32},
 number = {1},
 pages = {632--664},
 year = {2022},
 language = {English},
 doi = {10.1214/21-AAP1691},
 zbMATH = {7493833},
 Zbl = {1485.60082}
}

@article{kyprianou2017,
author = "Kyprianou, Andreas E. and Pagett, Steven W. and Rogers, Tim and Schweinsberg, Jason",
doi = "10.1214/16-AOP1150",
fjournal = "The Annals of Probability",
journal = "Ann. Probab.",
month = "11",
number = "6A",
pages = "3829--3849",
publisher = "The Institute of Mathematical Statistics",
title = "A phase transition in excursions from infinity of the fast fragmentation-coalescence process",
url = "https://doi.org/10.1214/16-AOP1150",
volume = "45",
year = "2017"
}

@book {MR1138461,
    AUTHOR = {Blumenthal, Robert M.},
     TITLE = {Excursions of {M}arkov processes},
    SERIES = {Probability and its Applications},
 PUBLISHER = {Birkh\"auser Boston, Inc., Boston, MA},
      YEAR = {1992},
     PAGES = {xii+275},
      ISBN = {0-8176-3575-0},
   MRCLASS = {60J25 (60G55 60J55)},
  MRNUMBER = {1138461},
MRREVIEWER = {L. C. G. Rogers},
       DOI = {10.1007/978-1-4684-9412-9},
       URL = {http://dx.doi.org/10.1007/978-1-4684-9412-9},
}

@article {MR1241926,
    AUTHOR = {Pakes, Anthony G.},
     TITLE = {Explosive {M}arkov branching processes: entrance laws and
              limiting behaviour},
   JOURNAL = {Adv. in Appl. Probab.},
  FJOURNAL = {Advances in Applied Probability},
    VOLUME = {25},
      YEAR = {1993},
    NUMBER = {4},
     PAGES = {737--756},
      ISSN = {0001-8678},
   MRCLASS = {60J80 (60J10)},
  MRNUMBER = {1241926},
MRREVIEWER = {A. Sp\u ataru},
       DOI = {10.2307/1427789},
       URL = {http://dx.doi.org/10.2307/1427789},
}

@article {MR724061,
    AUTHOR = {Cox, J. Theodore and R\"osler, Uwe},
     TITLE = {A duality relation for entrance and exit laws for {M}arkov
              processes},
   JOURNAL = {Stochastic Process. Appl.},
  FJOURNAL = {Stochastic Processes and their Applications},
    VOLUME = {16},
      YEAR = {1984},
    NUMBER = {2},
     PAGES = {141--156},
      ISSN = {0304-4149},
   MRCLASS = {60J50 (60J10 60J60)},
  MRNUMBER = {724061},
MRREVIEWER = {Yves Le Jan},
       DOI = {10.1016/0304-4149(84)90015-2},
       URL = {http://dx.doi.org/10.1016/0304-4149(84)90015-2},
}

@article {MR2134113,
    AUTHOR = {Lambert, Amaury},
     TITLE = {The branching process with logistic growth},
   JOURNAL = {Ann. Appl. Probab.},
  FJOURNAL = {The Annals of Applied Probability},
    VOLUME = {15},
      YEAR = {2005},
    NUMBER = {2},
     PAGES = {1506--1535},
      ISSN = {1050-5164},
   MRCLASS = {60J80 (60J60 60J70 92D15 92D25 92D40)},
  MRNUMBER = {2134113},
MRREVIEWER = {D. R. Grey},
       DOI = {10.1214/105051605000000098},
       URL = {http://dx.doi.org/10.1214/105051605000000098},
}

@article {DawsonLi,
    AUTHOR = {Dawson, Donald A. and Li, Zenghu},
     TITLE = {Stochastic equations, flows and measure-valued processes},
   JOURNAL = {Ann. Probab.},
  FJOURNAL = {The Annals of Probability},
    VOLUME = {40},
      YEAR = {2012},
    NUMBER = {2},
     PAGES = {813--857},
      ISSN = {0091-1798},
   MRCLASS = {60J68 (60H20 60J25 60J80)},
  MRNUMBER = {2952093},
MRREVIEWER = {Jos{\'e} Villa Morales},
       DOI = {10.1214/10-AOP629},
       URL = {http://dx.doi.org/10.1214/10-AOP629},
}

@book{zbMATH06256582,
 author = {B{\"o}ttcher, Bj{\"o}rn and Schilling, Ren{\'e} and Wang, Jian},
 title = {L{\'e}vy {M}atters {III}. {L{\'e}vy}-{T}ype {P}rocesses: {C}onstruction, {A}pproximation and {S}ample {P}ath {P}roperties},
 fseries = {Lecture Notes in Mathematics},
 series = {Lect. Notes Math.},
 issn = {0075-8434},
 volume = {2099},
 isbn = {978-3-319-02683-1; 978-3-319-02684-8},
 year = {2013},
 publisher = {Cham: Springer},
 language = {English},
 doi = {10.1007/978-3-319-02684-8},
 zbMATH = {6256582},
 Zbl = {1384.60004}
}

@book {Li-book,
    AUTHOR = {Li, Z. H.},
     TITLE = {Measure-Valued Branching {M}arkov Processes},
 PUBLISHER = {Springer},
      YEAR = {2011},
}

@article{cattiaux2009,
author = "Cattiaux, Patrick and Collet, Pierre and Lambert, Amaury and Martinez, Servet and M\'el\'eard, Sylvie and San Martin, Jaime",
doi = "10.1214/09-AOP451",
fjournal = "The Annals of Probability",
journal = "Ann. Probab.",
month = "09",
number = "5",
pages = "1926--1969",
publisher = "The Institute of Mathematical Statistics",
title = "Quasi-stationary distributions and diffusion models in population dynamics",
url = "https://doi.org/10.1214/09-AOP451",
volume = "37",
year = "2009"
}

@article {MR0408016,
    AUTHOR = {Grey, D. R.},
     TITLE = {Asymptotic behaviour of continuous time, continuous
              state-space branching processes},
   JOURNAL = {J. Appl. Probability},
  FJOURNAL = {Journal of Applied Probability},
    VOLUME = {11},
      YEAR = {1974},
     PAGES = {669--677},
      ISSN = {0021-9002},
   MRCLASS = {60J80},
  MRNUMBER = {0408016},
MRREVIEWER = {N. H. Bingham},
}

@book{MR3185174,
    AUTHOR = {Sato, Ken-iti},
     TITLE = {L\'evy processes and infinitely divisible distributions},
    SERIES = {Cambridge Studies in Advanced Mathematics},
    VOLUME = {68},
 PUBLISHER = {Cambridge University Press, Cambridge},
      YEAR = {2013},
     PAGES = {xiv+521},
      ISBN = {978-1-107-65649-9},
   MRCLASS = {60G51 (60E07 60G18 60G52 60J45)},
  MRNUMBER = {3185174},
}

@article{zbMATH05043274,
 author = {Caballero, M. E. and Chaumont, L.},
 title = {Weak convergence of positive self-similar {Markov} processes and overshoots of {L{\'e}vy} processes},
 fjournal = {The Annals of Probability},
 journal = {Ann. Probab.},
 issn = {0091-1798},
 volume = {34},
 number = {3},
 pages = {1012--1034},
 year = {2006},
 language = {English},
 doi = {10.1214/009117905000000611},
 zbMATH = {5043274},
 Zbl = {1098.60038}
}

@book{courchamp2008allee,
  title={Allee effects in ecology and conservation},
  author={Courchamp, Franck and Berec, Ludek and Gascoigne, Joanna},
  year={2008},
  publisher={OUP Oxford}
}

@article{zbMATH07373488,
 author = {Marguet, Aline and Smadi, Charline},
 title = {Long time behaviour of continuous-state nonlinear branching processes with catastrophes},
 fjournal = {Electronic Journal of Probability},
 journal = {Electron. J. Probab.},
 issn = {1083-6489},
 volume = {26},
 pages = {32},
 note = {Id/No 95},
 year = {2021},
 language = {English},
 doi = {10.1214/21-EJP664},
 zbMATH = {7373488},
 Zbl = {1480.60263}
}

@book{zbMATH05819412,
 author = {Etheridge, Alison},
 title = {Some mathematical models from population genetics. {{\'E}cole} d'{{\'E}t{\'e}} de {Probabilit{\'e}s} de {Saint}-{Flour} {XXXIX}-2009},
 fseries = {Lecture Notes in Mathematics},
 series = {Lect. Notes Math.},
 issn = {0075-8434},
 volume = {2012},
 isbn = {978-3-642-16631-0; 978-3-642-16632-7},
 year = {2011},
 publisher = {Berlin: Springer},
 language = {English},
 doi = {10.1007/978-3-642-16632-7},
 zbMATH = {5819412},
 Zbl = {1320.92003}
}

@incollection{zbMATH05841214,
 author = {Song, R. and Vondra{\v{c}}ek, Z.},
 title = {Potential theory of subordinate {Brownian} motion},
 booktitle = {Potential analysis of stable processes and its extensions. With some of the papers based on the presentations at the workshop on stochastic and harmonic analysis of processes with jumps, Angers, France, May 2--9, 2006.},
 isbn = {978-3-642-02140-4; 978-3-642-02141-1},
 pages = {87--176},
 year = {2009},
 publisher = {Berlin: Springer},
 language = {English},
 doi = {10.1007/978-3-642-02141-1_5},
 zbMATH = {5841214},
 Zbl = {1203.60118}
}

@book{Beres2,
	Address = {Rio de Janeiro},
	Author = {Berestycki, Nathana{\"e}l},
	Isbn = {978-85-85818-40-1},
	Mrclass = {60J25 (60J80 60K35)},
	Mrnumber = {2574323 (2011d:60225)},
	Mrreviewer = {Martin M{\"o}hle},
	Pages = {193},
	Publisher = {Sociedade Brasileira de Matem\'atica},
	Series = {Ensaios Matem\'aticos [Mathematical Surveys]},
	Title = {Recent progress in coalescent theory},
	Volume = {16},
	Year = {2009}}

@article{zbMATH07242423,
 author = {Foucart, Cl{\'e}ment and Li, Pei-Sen and Zhou, Xiaowen},
 title = {On the entrance at infinity of {Feller} processes with no negative jumps},
 fjournal = {Statistics \& Probability Letters},
 journal = {Stat. Probab. Lett.},
 issn = {0167-7152},
 volume = {165},
 pages = {8},
 note = {Id/No 108859},
 year = {2020},
 language = {English},
 doi = {10.1016/j.spl.2020.108859},
 zbMATH = {7242423},
 Zbl = {1450.60034}
}

@article{Lambert,
	Author = {Lambert, Amaury},
	Coden = {PTRFEU},
	Doi = {10.1007/s004400100155},
	Fjournal = {Probability Theory and Related Fields},
	Issn = {0178-8051},
	Journal = {Probab. Theory Related Fields},
	Mrclass = {60J80},
	Mrnumber = {1883717 (2003a:60130)},
	Mrreviewer = {Fima Klebaner},
	Number = {1},
	Pages = {42--70},
	Title = {The genealogy of continuous-state branching processes with immigration},
	Url = {http://dx.doi.org/10.1007/s004400100155},
	Volume = {122},
	Year = {2002}}

@article{zbMATH06837778,
 author = {Kukla, Jonas and M{\"o}hle, Martin},
 title = {On the block counting process and the fixation line of the {Bolthausen}-{Sznitman} coalescent},
 fjournal = {Stochastic Processes and their Applications},
 journal = {Stochastic Processes Appl.},
 issn = {0304-4149},
 volume = {128},
 number = {3},
 pages = {939--962},
 year = {2018},
 language = {English},
 doi = {10.1016/j.spa.2017.06.012},
 zbMATH = {6837778},
 Zbl = {1390.60089}
}

@article{rebotier,
      title={On the coming down from infinity of continuous-state branching processes with drift-interaction}, 
      author={Félix Rebotier},
      year={2025},
      journal={ArXiv Preprint ArXiv:2510.05958}
}

@article{foucartvidmar,
title = {Continuous-state branching processes with collisions: First passage times and duality},
 journal = {Stochastic Processes Appl.},
 fjournal = {Stochastic Processes and their Applications},
volume = {167},
pages = {no. 104230, 1-35},
year = {2024},
author = {Foucart, Cl\'ement and Vidmar, Matija},
}

@article {Duhalde,
    AUTHOR = {Duhalde, Xan and Foucart, Cl{\'e}ment and Ma, Chunhua},
     TITLE = {On the hitting times of continuous-state branching processes
              with immigration},
   JOURNAL = {Stochastic Process. Appl.},
  FJOURNAL = {Stochastic Processes and their Applications},
    VOLUME = {124},
      YEAR = {2014},
    NUMBER = {12},
     PAGES = {4182--4201},
      ISSN = {0304-4149},
   MRCLASS = {60J80 (60G17 60J25)},
  MRNUMBER = {3264444},
       DOI = {10.1016/j.spa.2014.07.019},
       URL = {http://dx.doi.org/10.1016/j.spa.2014.07.019},
}

@book {Kyprianoubook,
    AUTHOR = {Kyprianou, Andreas E.},
     TITLE = {Fluctuations of {L}\'evy {P}rocesses with {A}pplications. {I}ntroductory {L}ectures},
    SERIES = {Universitext},
   EDITION = {Second},
 PUBLISHER = {Springer, Heidelberg},
      YEAR = {2014},
     PAGES = {xviii+455},
      ISBN = {978-3-642-37631-3; 978-3-642-37632-0},
   MRCLASS = {60-01 (60E07 60E10 60Gxx)},
  MRNUMBER = {3155252},
MRREVIEWER = {Ren\'e L. Schilling},
       DOI = {10.1007/978-3-642-37632-0},
       URL = {http://dx.doi.org/10.1007/978-3-642-37632-0},
}

@article{zbMATH03294035,
 author = {Silverstein, Martin L.},
 title = {A new approach to local times},
 fjournal = {Journal of Mathematics and Mechanics},
 journal = {J. Math. Mech.},
 issn = {0095-9057},
 volume = {17},
 number={11},
 pages = {1023--1054},
 year = {1968},
 language = {English},
 zbMATH = {3294035},
 Zbl = {0184.41101}
}

@article{foucartvidmar2025,
  title   = {Positive Markov processes in {L}aplace duality},
  author  = {Foucart, Clément and Vidmar, Matija},
  journal = {arXiv preprint arXiv:2507.09641},
  year    = {2025},
}

@article{zbMATH07120715,
 author = {Li, Pei-Sen and Yang, Xu and Zhou, Xiaowen},
 title = {A general continuous-state nonlinear branching process},
 fjournal = {The Annals of Applied Probability},
 journal = {Ann. Appl. Probab.},
 issn = {1050-5164},
 volume = {29},
 number = {4},
 pages = {2523--2555},
 year = {2019},
 language = {English},
 doi = {10.1214/18-AAP1459},
 zbMATH = {7120715},
 Zbl = {1466.60179}
}

@article {MR2308333,
    AUTHOR = {Hutzenthaler, Martin and Wakolbinger, Anton},
     TITLE = {Ergodic behavior of locally regulated branching populations},
   JOURNAL = {Ann. Appl. Probab.},
  FJOURNAL = {The Annals of Applied Probability},
    VOLUME = {17},
      YEAR = {2007},
    NUMBER = {2},
     PAGES = {474--501},
      ISSN = {1050-5164},
   MRCLASS = {60K35 (60J60 60J80 92D25)},
  MRNUMBER = {2308333},
MRREVIEWER = {Jan M. Swart},
       URL = {https://doi.org/10.1214/105051606000000745},
}

@article{zbMATH06496393,
 author = {H{\'e}nard, Olivier},
 title = {The fixation line in the {{\(\Lambda\)}}-coalescent},
 fjournal = {The Annals of Applied Probability},
 journal = {Ann. Appl. Probab.},
 issn = {1050-5164},
 volume = {25},
 number = {5},
 pages = {3007--3032},
 year = {2015},
 language = {English},
 doi = {10.1214/14-AAP1077},
 zbMATH = {6496393},
 Zbl = {1325.60124}
}

@article{zbMATH00681374,
	author = {Vuolle-Apiala, J.},
	title = {It{\^o} excursion theory for self-similar {Markov} processes},
	fjournal = {The Annals of Probability},
	journal = {Ann. Probab.},
	issn = {0091-1798},
	volume = {22},
	number = {2},
	pages = {546--565},
	year = {1994},
	language = {English},
	doi = {10.1214/aop/1176988721},
	zbMATH = {681374},
	Zbl = {0810.60067}
}

@article{zbMATH03457981,
 author = {Stroock, Daniel W.},
 title = {Diffusion processes associated with {L\'evy} generators},
 fjournal = {Zeitschrift f{\"u}r Wahrscheinlichkeitstheorie und Verwandte Gebiete},
 journal = {Z. Wahrscheinlichkeitstheor. Verw. Geb.},
 issn = {0044-3719},
 volume = {32},
 pages = {209--244},
 year = {1975},
 language = {English},
 doi = {10.1007/BF00532614},
 zbMATH = {3457981},
 Zbl = {0292.60122}
}

@article{zbMATH06573013,
 author = {Barczy, M{\'a}ty{\'a}s and Li, Zenghu and Pap, Gyula},
 title = {Yamada-{Watanabe} results for stochastic differential equations with jumps},
 fjournal = {International Journal of Stochastic Analysis},
 journal = {Int. J. Stoch. Anal.},
 issn = {2090-3332},
 volume = {2015},
 pages = {23},
 note = {Id/No 460472},
 year = {2015},
 language = {English},
 doi = {10.1155/2015/460472},
 zbMATH = {6573013},
 Zbl = {1337.60116}
}
\end{document}